\documentclass[12pt]{amsart}

\usepackage{amsmath,amsthm,amssymb, amscd, amsfonts,enumerate,array,bbm,bm,xypic}
\usepackage{hyperref}

\allowdisplaybreaks[2]
\makeatletter
\def\mhline{\noalign{\ifnum0=`}\fi\hrule height 4\arrayrulewidth \futurelet
   \@tempa\oxhline}
\def\oxhline{\ifx\@tempa\hline\vskip \doublerulesep\fi
      \ifnum0=`{\fi}}

\makeatother

\numberwithin{equation}{section}

\newtheorem{theorem}{Theorem}[section]
\newtheorem{proposition}[theorem]{Proposition}
\newtheorem{conjecture}[theorem]{Conjecture}
\newtheorem{corollary}[theorem]{Corollary}
\newtheorem{lemma}[theorem]{Lemma}

\theoremstyle{definition}

\newtheorem{remark}[theorem]{Remark}
\newtheorem{example}[theorem]{Example}
\newtheorem{definition}[theorem]{Definition}

\renewcommand{\eqref}[1]{{\rm (\ref{#1})}}

\def\endproof{\hfill$\square$\medskip}

\def\ZZ{\mathbb{Z}}
\def\CC{\mathbb{C}}

\def\RR{\mathbb{R}}

\def\kk{\Bbbk}

\def\ii{\mathbf{i}}
\def\jj{\mathbf{j}}

\begin{document}

\title{Littlewood-Richardson coefficients for reflection groups}

\author{Arkady Berenstein}
\address{\noindent Department of Mathematics, University of Oregon,
Eugene, OR 97403, USA} \email{arkadiy@math.uoregon.edu}

\author{Edward Richmond}
\address{Department of Mathematics, University of British Columbia, Vancouver, BC V6T 1Z2, Canada}
\email{erichmond@math.ubc.ca}

\thanks{This work was partially supported by the NSF grant DMS-0800247}

\begin{abstract} In this paper we explicitly compute all Littlewood-Richardson coefficients for semisimple and Kac-Moody groups $G$,
that is, the structure constants (also known as the {\it Schubert structure constants}) of the
cohomology algebra $H^*(G/P,\CC)$, where $P$ is a parabolic subgroup of $G$. These coefficients are
of importance in enumerative geometry, algebraic combinatorics and representation theory. Our
formula for the Littlewood-Richardson coefficients is purely combinatorial and is given in terms of
the Cartan matrix and the Weyl group of $G$. However, if some off-diagonal entries of the Cartan
matrix are $0$ or $-1$, the formula may contain negative summands. On the other hand, if the Cartan
matrix satisfies $a_{ij}a_{ji}\ge 4$ for all $i,j$, then each summand in our formula is nonnegative
that implies nonnegativity of all Littlewood-Richardson coefficients. We extend this and other
results to the structure coefficients of the $T$-equivariant cohomology of flag varieties $G/P$ and
Bott-Samelson varieties $\Gamma_\ii(G)$.
\end{abstract}

\maketitle

\tableofcontents

\section{Introduction}

The goal of this paper is to explicitly compute the Littlewood-Richardson coefficients which are
structure constants of the cohomology algebra $H^*(G/B,\CC)$ for the flag variety $G/B$ of an
arbitrary semisimple or Kac-Moody group. More precisely, let $W$ be the Weyl group of $G$ and let
$\sigma_w\in H^*(G/B,\CC)$ denote the  Schubert cocycle corresponding to an element $w\in W$.  Then
the Littlewood-Richardson coefficients $c_{u,v}^w\in \ZZ_{\ge 0}$, $u,v,w\in W$ are defined as the
structure constants of the cup product in $H^*(G/B,\CC)$ with respect to the basis $\{\sigma_w,w\in
W\}$:
\begin{equation}
\label{eq:cup GB}
\sigma_u\cup \sigma_v=\sum_{w\in W} c_{u,v}^w \sigma_w\ .
\end{equation}

The study of Littlewood-Richardson coefficients $c_{u,v}^w$ has a long history and is an essential
part of Schubert calculus.  In enumerative geometry, these coefficients are realized as the
cardinality of triple intersections of certain Schubert varieties in general position.  From this
point of view, there have been several formulas for $c_{u,v}^w$ in using transversality and
degeneration techniques \cite{BelKum,BeSo02,Co09A,Knut1,Purb1,Richm,Richm2,Va06}.   In algebraic
combinatorics, the numbers $c_{u,v}^w$ for special $u,v,w$ can be determined via puzzles or
counting problems using Young tableaux \cite{Fu97, Fult,Knut,Kogan,Lenart,TY06}. Other
combinatorial approaches for computing $c_{u,v}^w$ include coinvariant algebras with Schubert
polynomial bases (\cite{BeSo98,BiHa95, FiKi96, LaSc82, PoSts09}) or recursions over the Weyl groups
(\cite{Bi99, KK86}).  While there have been many interesting formulas and algorithms for computing
these numbers, they are mostly limited to special cases of reductive Lie groups $G$.

To the best of our knowledge, the only non-recursive formula for  Littlewood-Richardson
coefficients was obtained by H.~Duan in a remarkable paper \cite{Duan2005}  and the equivariant
generalization was later obtained by M.~Willems in \cite{wi06}. The major issue here is that both
Duan's and Willems formulas contain a large number of summands, including several negative terms.
For instance, if $G=\widehat{SL_2}$, $u=v=(s_1s_2)^2$, $w=u^2$, then Duan's formula for $c_{u,v}^w$
has about 17,000 summands, but our formula \eqref{eq:cuvw} contains only 19 summands (see Remark
\ref{rem:Duan} for details). However, a comparison with Duan's and Willems approaches was very
productive and resulted in a discovery of a combinatorial formula for {\it Bott-Samelson} numbers,
i.e., structure constants of the (equivariant) cohomology algebra of Bott-Samelson varieties
(Theorem \ref{th:cuvw refined T BS}).

\begin{remark} In fact, the Littlewood-Richardson coefficients for partial flag varieties $G/P$,
where $P\supset B$ (e.g., for Grassmannians) is a parabolic subgroup of $G$ are determined by the
respective coefficients for $G/B$ because the pullback $H^*(G/P,\CC)\hookrightarrow H^*(G/B,\CC)$ of
canonical projection $G/B\to G/P$ turns $H^*(G/P,\CC)$ into the subalgebra of  $H^*(G/B,\CC)$ spanned
by a part of the Schubert basis.

\end{remark}

\begin{remark} In some of the above mentioned papers and several other papers the coefficients
$c_{u,v}^w$ have been referred  to as {\it Schubert structure constants}. However, we believe that
the ``Littlewood-Richardson" terminology for these constants is justified historically (it is used
e.g., in
\cite{BeSo98,Bert99,bkt04,Co09A,Kir07,Knut,KnutY04,Kog01,Kreim10,LapMor08,Mih06,molev04,molev09,PoSts09,Va06})
and by a number of purely  mathematical reasons. First, the classical Jacobi-Trudy formula for
Schubert classes in Grassmannians $Gr_k(\CC^n)$ (see e.g., \cite{bkt04}) implies that the structure
constants of $H^*(Gr_k(\CC^n),\CC)$ are the classical Littlewood-Richardson coefficients (in fact,
\cite[Theorem 2]{bkt04} asserts  that the structure constants for maximal isotropic Grassmannians
are also given by Littlewood-Richardson-Stembridge rule). Second, based on results of Klyachko on
Horn inequalities and follow-up work (see e.g., \cite{BelKum,BS00,Kly98,KTW04,Res07}),  if
$V_\lambda,V_\mu,V_\nu$ are simple $G$-modules and $c_{\lambda,\mu}^\nu$ denotes the multiplicity
$\dim_\CC Hom_G(V_\nu,V_\lambda\otimes V_\mu)$, then the set of triples $(\lambda,\mu,\nu)$  such
that  $c_{\lambda,\mu}^\nu \ne 0$ is determined (up to saturation) in terms of the set of all
triples $(u,v,w)\in W^3$ such that $c_{u,v}^w\ne 0$. Another similarity (and complementarity) of
these coefficients is that, when $\dim~G<\infty$,  the representation ring $R(G)$ (the carrier of
the ``genuine" Littlewood-Richardson coefficients $c_{\lambda,\mu}^\nu$) is the invariant algebra
$\CC[T]^W$, where $T$ is the maximal torus of $G$ and the cohomology algebra $H^*(G/B,\CC)$ is the
coinvariant algebra $\CC[Lie(T)]_W$ of the Lie algebra  $Lie(T)$.

\end{remark}

We compute $c_{u,v}^w$ in terms of the $W$-action on the root lattice of $G$. Our main formula
\eqref{eq:cuvw} is very different from Duan's and, in particular, if the Cartan matrix satisfies
$a_{ij}a_{ji}\ge 4$ for all $i,j$, then all summands in  \eqref{eq:cuvw} turn out to be nonnegative
which implies  $c_{u,v}^w\ge 0$ for all relevant $u,v,w$ (Theorem \ref{th:nonneg}). However, if the
Cartan matrix $A$ of $G$ has entries $a_{ij}\in \{0,-1\}$, the right hand side of \eqref{eq:cuvw}
may contain negative summands. Nevertheless, we believe that the negative terms can be effectively
canceled (see Remark \ref{rem:cuvw+} for details). We discuss positivity in greater details in
Theorem \ref{th:nonneg} and Conjectures \ref{conj:positive puvw(t)} and \ref{conj:fully
commutative}.

\bigskip
\noindent\textbf{Acknowledgments}. The authors thank David Anderson,  Prakash Belkale, Misha Kapovich, Shrawan Kumar for stimulating discussions. We are grateful to Harry Tamvakis for historical insights. Special thanks are due to Vadim Vologodsky for explaining equivariant homology theories.

\section{Definitions and main results}

\label{sect:Definitions and main results}

Let $G$ be a Kac-Moody group and let $A=(a_{ij})$ be the $I\times I$ Cartan matrix of $G$ (where
$I$ denotes the indexing set, i.e., the set of vertices of the Dynkin diagram).
The Weyl group $W$ of $G$ is generated by  simple reflections $s_i$, $i\in I$ that act on the {\it root space} $V=\bigoplus\limits_{i\in I} \CC\cdot  \alpha_i$ by:
\begin{equation}
\label{eq:reflection}
s_i(\alpha_j)=\alpha_j-a_{ij} \alpha_i \end{equation}
for $i,j\in I$.

\smallskip

\begin{definition}
\label{def:admissible} Let $m$ be a positive integer and let $\ii\in I^m$. For each subset
$M=\{m_1<\cdots<m_r\}$ of the interval $[m]:=\{1,2,\ldots,m\}$ denote by $\ii_M$ the subsequence
$(i_{m_1},\ldots,i_{m_r})\in I^r$ of $\ii$.  We say that a sequence $\ii=(i_1,\ldots,i_m)\in I^m$
is {\it reduced} if the element $w=w_\ii:=s_{i_1}\cdots s_{i_m}\in W$ is shortest possible and
define its {\it Coxeter length} $\ell(w):=m$.  We say that a sequence $\ii$ is {\it admissible} if
$i_k\ne i_{k+1}$ for all $j\in[m-1]$ (clearly, every reduced sequence is admissible).
Given $w\in W$, denote by $R(w)$ the set of all {\it reduced words} of $w$, i.e, all $\ii\in I^{\ell(w)}$ such that $w_\ii=w$.
\end{definition}

\begin{definition}
\label{def:bounded} Let $m\ge 0$ and let  $L,M$ be subsets of $[m]$ such that $|L|+|M|=m$. We say
that a bijection
$$\varphi:L\, \widetilde \to\, [m]\setminus M$$
is {\it bounded} if $\varphi(\ell)<\ell$ for each $\ell\in L$. Given a sequence
$\ii=(i_1,\ldots,i_m)\in I^m$, we say that a bounded bijection $\varphi:L\ \widetilde \to\
[m]\setminus M$ is $\ii$-{\it admissible} if the sequence $\ii_{M\cup \varphi(L_{<\ell})}$ is
admissible for all $\ell\in L$, where we abbreviate $L_{<\ell}:=L\cap [\ell-1]$ (in particular
$L_{<\ell}=\emptyset$ if $\ell$ is the minimal element of $L$).
\end{definition}

\begin{example}
Let $I=\{1,2\}$ with $\ii=(1,2,1,2).$ If $L=M=\{3,4\},$ then $$\phi_1:(3,4)\rightarrow(2,1)\qquad\phi_2:(3,4)\rightarrow(1,2).$$
are both bounded bijections.  In particular, $\phi_1$ is $\ii$-admissible and $\phi_2$ is not $\ii$-admissible.  
\end{example}


For $j\in I$ denote by $\langle \cdot, \alpha_j^\vee\rangle$ the linear function on $V$ given by
$\langle \alpha_i,\alpha_j^\vee\rangle=a_{ij}$. For any sequence $\ii=(i_1,\ldots,i_m)\in I^m$ and
bijection $\varphi:L\ \widetilde \to\ [m]\setminus M$ we define the integer $p_{\varphi}$ by the
formula

\begin{equation}\label{eq:plambda1}
p_\varphi:=\prod_{\ell\in L} \langle
w_{\ii_{M(\ell)}}(-\alpha_{i_\ell}),\alpha_{i_{\varphi(\ell)}}^\vee \rangle\end{equation}

where $$M(\ell):=\{r\in M\cup\varphi(L_{<\ell})\ |\ \varphi(\ell)<r<\ell\}.$$ If
$M(\ell)=\emptyset$, then $w_{\ii_{M(\ell)}}=1$ and if $L=\emptyset$ then $p_\varphi=1$.  The
following is our main result (in which we implicitly use the  well-known fact that $c_{u,v}^w=0$
unless $\ell(w)=\ell(u)+\ell(v)$).

\begin{theorem}\label{th:cuvw} Let $G$ be a Kac-Moody group and $W=\langle s_i,i\in I\rangle$ be its Weyl group. Then for any $u,v,w\in W$ such that $\ell(w)=\ell(u)+\ell(v)$ and  any given $\ii\in R(w)$ one has:
\begin{equation}
\label{eq:cuvw}
c_{u,v}^w=\sum p_\varphi
\end{equation}
with the summation over all triples $({\bf u},{\bf v},\varphi)$, where
\begin{itemize}
\item ${\bf u},{\bf v}\subset [m]$ such that $\ii_{\bf u}\in R(u)$, $\ii_{\bf v}\in R(v)$ (hence $|{\bf u}\cap {\bf v}|+|{\bf u}\cup {\bf v}|=m$);
\item $\varphi:{\bf u}\cap {\bf v}\ \to\ [m]\setminus ({\bf u}\cup {\bf v})$ is an $\ii$-admissible bounded bijection.
\end{itemize}\end{theorem}

\begin{remark} The right hand side of \eqref{eq:cuvw} depends on a choice of $\ii=(i_1,\ldots,i_m)\in R(w)$. It would be interesting to find an ``optimal'' $\ii$ (depending on $u$, $v$ and $w$)  that would minimize the number of summands.

\end{remark}

\begin{remark}
It turns out that \eqref{eq:cuvw}  holds if we drop the condition of $\ii$-admissibility. See
discussion after Theorem \ref{th:cuvw refined T BS} for details.
\end{remark}

\begin{remark}
\label{rem:cuvw+} It frequently happens that each summand of \eqref{eq:cuvw} is positive (see
Theorem \ref{th:nonneg} for details). Based on this observation, we can define for each $u,v,w\in
W$, and $\ii\in R(w)$ the coefficient $c_{u,v}^{\ii,+}$ by replacing each $p_\varphi$ in
\eqref{eq:cuvw} with $p_\varphi^+$, where $p_\varphi^+$ is obtained by replacing each factor
$p_\ell:=\langle w_{\ii_{M(\ell)}}(-\alpha_{i_\ell}),\alpha_{i_{\varphi(\ell)}}^\vee \rangle$ in
\eqref{eq:plambda1} with $\max(p_\ell,0)$. We then define $$c_{u,v}^{w,+}:=\min_{\ii\in R(w)}
c_{u,v}^{\ii,+}\ .$$ Based on numerous examples, including all $G=SL_n$, $n\le 6$, we can
conjecture that
\begin{equation}
\label{eq:cuvw+}
c_{u,v}^w\le c_{u,v}^{w,+}
\end{equation}
for all $u,v,w\in W$ with $\ell(u)+\ell(v)=\ell(w)$. In fact, in most of examples, the inequality \eqref{eq:cuvw+} was an equality. In any case, if the inequality \eqref{eq:cuvw+} holds,  we can try to express $c_{u,v}^w$ as a``sub-sum" of one of $c_{u,v}^{\ii,+}=\sum p_\varphi^+$, i.e.,
express the Littlewood-Richardson coefficient $c_{u,v}^w$ in purely nonnegative terms.

\end{remark}

\begin{remark} \label{rem:Duan} H.~Duan proved in \cite{Duan2005} that if $G$ is semisimple, then for each $u,v,w\in W$ such that $\ell(w)=\ell(u)+\ell(v)$ and a given $\ii=(i_1,\ldots,i_m)\in R(w)$ one has in the notation as above:
\begin{equation}
\label{eq:duan}
c_{u,v}^w=\sum  \prod\limits_{1< \ell \le m} \frac{ c_{*,\ell}!}{\prod\limits_{1\le k < \ell} c_{k,\ell}!}  \prod_{1\le k< \ell \le m} b_{k,\ell}^{c_{k,\ell}}
\end{equation}
where $b_{k,\ell}=\langle s_{i_{k+1}}\cdots s_{i_{\ell-1}}(-\alpha_{i_\ell}),\alpha_{i_k}^\vee\rangle$, $1\le k<\ell \le m$ and the summation is over all triples $({\bf u},{\bf v},{\bf c})$ such that

$\bullet$ ${\bf u},{\bf v}\subset [1,m]$, $\ii_{\bf u}\in R(w)$, $\ii_{\bf v}\in R(v)$.

$\bullet$ ${\bf c}=(c_{k,\ell}|1\le k<\ell\le m)$ is a triangular array of nonnegative integers such that
\begin{equation}
\label{eq:duans array}
c_{*,k}-c_{k,*}=
\begin{cases}
1 & \text{if $k\in {\bf u}\cap {\bf v}$} \\
-1 & \text{if $k\in {\bf u}\cup {\bf v}\setminus ({\bf u}\cap {\bf v})$}\\
0 & \text{otherwise}
\end{cases}
\end{equation}
for $k=1,\ldots,m$
(here we abbreviated $c_{*,k}=\sum\limits_{1\le s<k} c_{s,k}$, $c_{k,*}=\sum\limits_{k<s\le m} c_{k,s}$).

Even though Duan's formula \eqref{eq:duan} makes sense for any Kac-Moody group $G$ and bears some
resemblance with our formula \eqref{eq:cuvw} (see also Remark \ref{rem:Duan BS}), the formulas are
still very different:  the set of all triangular arrays ${\bf c}$ in \eqref{eq:duan} is much larger
than the set of all relevant bounded bijections in Theorem \ref{th:cuvw}. For instance, if
$G=\widehat{SL_2}$, $u=v=(s_1s_2)^2$, $w=u^2$, then Duan's formula for $c_{u,v}^w$ has about 17,000
summands, but our formula \eqref{eq:cuvw} contains only 19 summands.
\end{remark}

In fact, Theorem \ref{th:cuvw} is a particular case of more general result (Theorem \ref{th:cuvw
refined T}) in which we compute all $T$-{\it equivariant} Littlewood-Richardson coefficients
$p_{u,v}^w\in S^{\ell(u)+\ell(v)-\ell(w)}(V)$ that are defined  by:
\begin{equation}
\label{eq:cup GB T}
\sigma_u^T\cup \sigma_v^T=\sum_{w\in W} p_{u,v}^w \sigma_w^T\ ,
\end{equation}
where $H_T^*(G/B)$ is the $T$-equivariant cohomology algebra of $G/B$ (see e.g., \cite[Section
11.3]{Ku02}), $T\subset B$ is a maximal torus of $G$, and $\sigma_w^T$ is the $T$-equivariant
Schubert cocycle (it is well-known that $p_{u,v}^w$ is homogeneous of degree
$\ell(u)+\ell(v)-\ell(w)$, in particular,  $c_{u,v}^w=\delta_{\ell(w),\ell(u)+\ell(v)}p_{u,v}^w$).

Using results of \cite{Duan2005,wi04,wi06}, we compute  $p_{u,v}^w$ via the {\it Bott-Samelson
coefficients} $p_{K',K''}^{\ii,K}$ as follows.

Recall that for each sequence $\ii=(i_1,\ldots,i_m)\in I^m$  the  $\ii$-th {\it Bott-Samelson
variety} $\Gamma_\ii=\Gamma_\ii(G)$ of $G$ is defined by:
$$\Gamma_\ii=(P_{i_1}\times_B P_{i_2}\times_B\cdots \times_B P_{i_m})/B $$
where $P_i$, $i\in I$ stands for the $i$-th minimal parabolic subgroup
(see e.g., \cite{wi04}).

It is well-known (see e.g., \cite{AndFul,Ku02}) that the $T$-equivariant cohomology algebra
$H_T^*(\Gamma_\ii,\CC)$ has a $\CC$-linear basis $\{\sigma_K^T\}$, where $K$ runs over all subsets
of $[m]$. Therefore, one defines the {\it equivariant Bott-Samelson coefficients}
$p_{K',K''}^{\ii,K}\in S^{|K|-|K'\cup K''|}(V)$ similarly to \eqref{eq:cup GB T} by:
\begin{equation}
\label{eq:cup BS T}
\sigma_{K'}^T\cup \sigma_{K''}^T=\sum p_{K',K''}^{\ii,K} \sigma_K^T\ ,
\end{equation}
where the summation is over all subsets $K\subset [m]$ such that $K'\cup K''\subset K$ and $|K|\le
|K'|+|K''|$.

Note that $p_{K',K''}^{\ii,K}=\delta_{K,K'\cup K''}$ if $K'\cap K''=\emptyset$.

The following result was proved by H.~Duan in \cite[Lemma 5.1]{Duan2005} for the ordinary
cohomology and by M.~Willems in \cite[Proposition 9]{wi04} in  the $T$-equivariant setting (see
also our algebraic generalization, Theorem \ref{th:general BS} and Corollary \ref{cor:general BS
quasi-cartan}).
\begin{proposition}
\label{pr:willems} Let $G$ be a Kac-Moody group and $W$ be its Weyl group. Then for any sequence
$\ii=(i_1,\ldots,i_m)\in I^m$ one has:

(a) The pullback of the canonical projection $\mu_\ii:\Gamma_\ii\to G/B$ is an algebra homomorphism
$\mu_\ii^*:H^*_T(G/B,\CC)\to H^*_T(\Gamma_\ii,\CC)$ given by
$$\mu_\ii^*(\sigma_w^T)= \sum_{K\subset [m]:\ii_K\in R(w)} \sigma_K^T$$
for all $w\in W$ (with the convention that $\mu_\ii^*(\sigma_w)=0$ if $\ii_K\notin R(w)$ for all $k\subset [m]$).

(b) If $\ii\in R(w)$ for some $w\in W$, then for any $u,v\in W$  the equivariant Littlewood-Richardson and Bott-Samelson coefficients are related by:
\begin{equation}
\label{eq:puvw via BS}
p_{u,v}^w=\sum p_{K',K''}^{\ii,[m]}\ ,
\end{equation}
with the summation over all subsets $K',K''\in [m]$ such that $\ii_{K'}\in R(u)$, $\ii_{K''}\in R(v)$.

\end{proposition}

Thus, according to \eqref{eq:puvw via BS}, in order to compute all $p_{u,v}^w$, it suffices to
compute $p_{K',K''}^{\ii,K}$. To do so, we need some notation.

For each bijection $\varphi:L\widetilde \to [m]\setminus M$ and any $k\notin L$ we define a root
$\alpha_k^{(\varphi)}\in V$ by
\begin{equation}
\label{def:plambda2}
\alpha_k^{(\varphi)}:=\left(\buildrel\longrightarrow \over {\prod_{r\in M_{<k} \cup \varphi(L_{<k})}} s_{i_r}\right)(\alpha_{i_k}) \ .
\end{equation}


The following result, to the best of our knowledge, was previously unknown.

\begin{theorem}
\label{th:cuvw refined T BS} Let $G$ be a Kac-Moody group and $W=\langle s_i,i\in I\rangle$ be its
Weyl group. Then (in the notation of Theorem \ref{th:cuvw}) for each   $\ii\in I^m$  and
$K,K',K''\subset [m]$ with $K'\cup K''\subset K$, $|K|\le |K'|+|K''|$ one has:
\begin{equation}
\label{eq:cuvw refined T BS}
p_{K',K''}^{\ii,K}=\sum p_\varphi\cdot \prod_{k\in (K'\cap K'')\setminus L} \alpha_k^{(\varphi)}
\end{equation}
with the summation over all pairs  $(L,\varphi)$, where
\begin{itemize}
\item $L$ is a subset of $K'\cap K''$ such that $|L|+|K'\cup K''|=|K|$;
\item $\varphi:L\ \to K\setminus (K'\cup K'')$ is a bounded bijection.
\end{itemize}
\end{theorem}

We prove Theorem \ref{th:cuvw refined T BS} in Section \ref{sect:proofs}.

\begin{remark} The Bott-Samelson coefficients $c_{K',K''}^{\ii,K}\in \ZZ$ for the ordinary cohomology
$H^*(\Gamma_\ii(G),\CC)$ are given by
$$c_{K',K''}^{\ii,K}=\delta_{|K|,|K'|+|K''|}p_{K',K''}^{\ii,K}.$$
In this case, i.e., when $|K|=|K'|+|K''|$, the formula \eqref{eq:cuvw refined T BS} simplifies
because $L=K'\cap K''$ and the thus summation in \eqref{eq:cuvw refined T BS} is over all bounded
bijections $\varphi:K'\cap K''\ \to K\setminus (K'\cup K'')$.
\end{remark}

\begin{remark}
\label{rem:Duan BS} The coefficients $c_{K',K''}^{\ii,K}$ were  computed in \cite{Duan2005}  and
$p_{K',K''}^{\ii,K}$ were computed for any $K,K',K''$ in \cite{wi06}. In the notation of Remark
\ref{rem:Duan} one has:
\begin{equation}
\label{eq:duan eqivar}
c_{K',K''}^{\ii,[m]}=\sum  \prod\limits_{1< \ell \le m} \frac{ c_{*,\ell}!}{\prod\limits_{1\le k < \ell} c_{k,\ell}!}  \prod_{1\le k< \ell \le m} b_{k,\ell}^{c_{k,\ell}}\ ,
\end{equation}
where the summation is over all ${\bf c}$ such that
$$c_{*,k}-c_{k,*}=
\begin{cases}
1 & \text{if $k\in K'\cap K''$} \\
-1 & \text{if $k\in K'\cup K''\setminus (K'\cap K'')$}\\
0 & \text{otherwise}
\end{cases}
$$
for $k=1,\ldots,m$.

Indeed, proving that the right hand sides of \eqref{eq:cuvw refined T BS}
and \eqref{eq:duan eqivar} are equal, is a rather interesting and challenging combinatorial problem.

\end{remark}

Note that, unlike (equivariant) Littlewood-Richardson coefficients, the (equivariant) Bott-Samelson
coefficients $p_{K',K''}^{\ii,K}$ are not always positive.  Therefore, the formula \eqref{eq:puvw
via BS} is not optimal. In fact, combining \eqref{eq:cuvw refined T BS} with \eqref{eq:puvw via BS}
(provided that $\ell(u)+\ell(v)=\ell(w)$), we obtain the assertion of Theorem \ref{th:cuvw} with
the $\ii$-admissibility condition dropped. For instance, in the example of Remark \ref{rem:Duan},
dropping $\ii$-admissibility would (modestly) increase the number of summands in \eqref{eq:cuvw}
from 19 to 190.

Instead of directly combining \eqref{eq:cuvw refined T BS} with \eqref{eq:puvw via BS}, we
introduce (and compute) the ``interpolating" coefficients as follows. For any triple of sequences
$\ii\in I^m$, $\ii'\in I^{m'}$, $\ii''\in I^{m''}$ define the {\it relative} ($T$-equivariant)
Littlewood-Richardson coefficient $p_{\ii',\ii''}^\ii$ by
\begin{equation}
\label{eq:relative LR via BS}
p_{\ii',\ii''}^\ii=\sum p_{K',K''}^{\ii,[m]}
\end{equation}
where the summation is over all pairs $K',K''\subset [m]$ such that $\ii_{K'}=\ii'$,
$\ii_{K''}=\ii''$ (in the notation of Definition \ref{def:admissible}).

We get that equation \eqref{eq:puvw via BS} simplifies to:
\begin{equation}
\label{eq:sum of relatives T}
p_{u,v}^w=\sum p_{\ii',\ii''}^\ii
\end{equation}
for all $u,v,w\in W$  and  any  given $\ii\in R(w)$, where  the summation is over all
sub-sequences $\ii',\ii''$ of $\ii$ such that $\ii'\in R(u)$, $\ii''\in R(v)$.

The following Theorem is our second main result which, taken together with  \eqref{eq:sum of
relatives T}, finishes the computation of all equivariant Littlewood-Richardson coefficients (and
thus verifies Theorem \ref{th:cuvw}).

\begin{theorem}
\label{th:cuvw refined T} Let $G$ be a Kac-Moody group and $W=\langle s_i,i\in I\rangle$ be its
Weyl group. Then  for each  admissible sequence $\ii\in I^m$ one has:
\begin{equation}
\label{eq:cuvw refined T}
p_{\ii',\ii''}^\ii=\sum p_\varphi\cdot \prod_{k\in (K'\cap K'')\setminus L} \alpha_k^{(\varphi)}
\end{equation}
with the summation over all quadruples  $(K',K'',L,\varphi)$, where
\begin{itemize}
\item $K'$ and $K''$ are subsets of $[m]$ such that $\ii_{K'}=\ii'$, $\ii_{K''}=\ii''$;
\item $L$ is a subset of $K'\cap K''$ such that $|L|+|K'\cup K''|=m$;
\item $\varphi:L\ \to [m]\setminus (K'\cup K'')$ is an $\ii$-admissible bounded bijection.
\end{itemize}
\end{theorem}

%

We prove Theorem \ref{th:cuvw refined T} in Section \ref{sect:proofs}.  Note that the right hand
side of  \eqref{eq:cuvw} and \eqref{eq:cuvw refined T} makes sense for any Coxeter group $W$ acting
on $V$ by \eqref{eq:reflection}, even if the group $G$ does not exist.  The only data needed is the
Cartan matrix $A$. In fact, the main ingredient of the proof is the realization of the relative
coefficients $p_{\ii',\ii''}^\ii$ as the structure constants in the dual of the generalized nil-Hecke
algebra of the free  {\it Coxeter semigroup} as follows (see also Section \ref{sect:nil hecke}).

\begin{theorem}
\label{th:embedding H(G/B) to H(hat G/hat B)} For each Kac-Moody group $G$  there exists a
commutative $S(V)$-algebra ${\mathcal A}(G)$ with the basis $\{\sigma_\ii\}$, where $\ii$ runs over
all sequences in $I^m$, $m\ge 0$ such that:

(a) For any  sequences $\ii'\in I^{m'}$ and $\ii''\in I^{m''}$ one has:
$$\sigma_{\ii'} \sigma_{\ii''}= \sum_\ii p_{\ii',\ii''}^\ii \sigma_\ii$$
with the summation over all sequences $\ii\in I^m$, $m\ge 0$ containing $\ii'$ and $\ii''$ as
sub-sequences and such that $m\le m'+m''$.

(b) The linear span of all $\sigma_\ii$ with admissible $\ii$ is a subalgebra ${\mathcal
A}^{adm}(G)$ of ${\mathcal A}(G)$.

(c) The association
$\sigma_w^T \mapsto \sum\limits_{\ii\in R(w)} \sigma_\ii$
defines an injective algebra homomorphism
\begin{equation}
\label{eq:embedding H(G/B) to A(G)}
H_T^*(G/B)\hookrightarrow {\mathcal A}^{adm}(G) \ .
\end{equation}

(d) Given $\ii\in I^m$, the association
\begin{equation}
\label{eq:homomorphism A(G)->HGammaii}
\displaystyle{\sigma_{\ii'}\mapsto \begin{cases} \sum\limits_{K\subset [m]:\ii_K=\ii'}\sigma_K^T &\text{if $\ii'$ is a subsequence of $\ii$}\\
                                    0 & \it{otherwise}\\
                                    \end{cases}
                                    }
\end{equation}
for all $\ii'\in I^{m'}$, $m'\ge 0$ defines an algebra homomorphism
${\mathcal A}(G)\to  H^*(\Gamma_\ii(G),\CC)$; moreover,
the canonical algebra homomorphism $\varphi_\ii^*:H^*_T(G/B,\CC)\to H^*_T(\Gamma_\ii,\CC)$ from Proposition \ref{pr:willems}(a)
factors through as the composition of \eqref{eq:embedding H(G/B) to A(G)} and \eqref{eq:homomorphism A(G)->HGammaii}.

(e) For each $\ii\in I^m$ the linear span $J_\ii\subset {\mathcal A}(G)$ of all $\sigma_{\ii'}$ such that $\ii'$ is not a subsequence of $\ii$ is an ideal in ${\mathcal A}(G)$; moreover, the kernel of \eqref{eq:homomorphism A(G)->HGammaii} is $J_\ii$, hence one has an injective homomorphism of algebras
\begin{equation}
\label{eq:homomorphism Aii(G)->HGammaii}
{\mathcal A}(G)/J_\ii\hookrightarrow  H^*(\Gamma_\ii(G),\CC)
\end{equation}
\end{theorem}

We prove Theorem \ref{th:embedding H(G/B) to H(hat G/hat B)}  in Section \ref{sect:nil hecke} (as a
corollary of Proposition \ref{pr:dual nil Hecke} and Theorem \ref{the:cohomology of flags}).



\begin{remark} It follows from the results of Section \ref{sect:nil hecke} that the linear
span of all $\sigma_{\ii'}$ with admissible $\ii'$ is an $S(V)$-subalgebra (which we denote
${\mathcal A}^{adm}(G)$) of ${\mathcal A}(G)$.


Therefore, it would be natural to conjecture that for each $\ii$ there is are varieties (or at
least a topological spaces) $X_\ii=X_\ii(G)$ and $X^{adm}_\ii=X^{adm}_\ii(G)$ with the $T$-action
and $T$-equivariant morphisms $\Gamma_\ii(G)\twoheadrightarrow X_\ii\twoheadrightarrow
X_\ii^{adm}\to G/B$ such that:

$\bullet$ the canonical projection  $\mu_\ii:\Gamma_\ii(G)\to G/B$ factors through as $X_\ii$ and $X_\ii^{adm}$.

$\bullet$ $H_T^*(X_\ii^{adm},\CC)\cong {\mathcal A}_\ii^{adm}(G)$, $H_T^*(X_\ii,\CC)\cong {\mathcal A}_\ii(G)$ and the algebra homomorphisms
$$H_T^*(G/B,\CC)\to {\mathcal A}_\ii^{adm}(G)/J_\ii^{adm}\hookrightarrow {\mathcal A}_\ii(G)/J_\ii\hookrightarrow  H^*(\Gamma_\ii(G),\CC)$$ are just the pullbacks of the above  morphisms $\Gamma_\ii(G)\twoheadrightarrow X_\ii\twoheadrightarrow X_\ii^{adm}\to G/B$ (here $J_\ii^{adm}=J_\ii\cap {\mathcal A}^{adm}(G)$).

\medskip

If $\ii\in R(w)$, then $X_\ii^{adm}$ should be thought of as a ``minimal resolution'' of
singularities of the corresponding Schubert variety in $G/B$.

\end{remark}

\medskip

We now apply  Theorems \ref{th:cuvw} and \ref{th:cuvw refined T} to give a combinatorial proof of
positivity of the (equivariant) Littlewood-Richardson coefficients for a large class Kac-Moody
groups. In \cite{KuNo98}, Kumar and Nori proved that if $A$ is a Cartan matrix of some Kac-Moody
group $G$, then every coefficient $c_{u,v}^w\geq 0$. This result for semisimple groups $G$ is known
via Kleiman's transversality \cite{Kl74} and  transitivity of $G$-action on the flag variety $G/B$.
For equivariant coefficients corresponding to Kac-Moody groups, Graham in \cite{Gra01} proved that
$p^w_{u,v}$ have nonnegative coefficients as polynomials in the basis $\{\alpha_i\}_{i\in I}.$  To
the best of our knowledge, all known positivity proofs rely on the geometry of the flag variety
$G/B$.

In Section \ref{sect:nil hecke} we introduce the notion of compatibility of  a {\it quasi-Cartan} matrix $A$, i.e., an $I\times I$-matrix over $\kk$ such that $a_{ii}=2$ for $i\in I$ and $a_{ij}=0\Leftrightarrow a_{ji}=0$, and a Coxeter group $W=\langle s_i,i\in I\rangle$ by requiring that $W$ acts on the root space $V=\bigoplus\limits_{i\in I} \kk \alpha_i$ by reflections defined in \eqref{eq:reflection}.  We define the {\it generalized} Littlewood-Richardson coefficients $p_{u,v}^w=p_{u,v}^w(A)$, $c_{u,v}^w=\delta_{\ell(w),\ell(u)+\ell(v)}p_{u,v}^w$ for each such a compatible pair $(A,W)$ and all relevant $u,v,w\in W$.

\begin{theorem}\label{th:nonneg}
Let $A$ be a quasi-Cartan matrix over $\RR$ compatible with a Coxeter group $W$ such that
\begin{equation}
\label{eq:nonnegativity for A special}
a_{ij}< 0 \quad \text{and}\quad a_{ij}a_{ji}\geq 4
\end{equation}
for all $i\ne j$.
Then  all $c_{u,v}^w$ are non-negative and all $p_{u,v}^w\in \RR_{\ge 0}[\alpha_i,i\in I]$.
\end{theorem}

The above theorem covers precisely those Kac-Moody groups $G$ whose Weyl group $W$ is a  Coxeter group with no braid relations.  We prove Theorem \ref{th:nonneg} in Section \ref{sect:positivity} by verifying  that each factor of $p_\varphi$ in \eqref{eq:plambda1} is nonnegative. This proof is completely combinatorial and relies on no geometry.



\smallskip

It is easy to show (see Section \ref{sect:examples}) that  for each pair $i\ne j$ the inequality
$c_{u,v}^w\ge 0$ for all $u,v,w\in W_{ij}=\langle s_i,s_j\rangle$ is equivalent to:

\begin{equation}
\label{eq:nonnegativity for A}
\text{$a_{ij}\le 0$ and: either
$a_{ij}a_{ji}\geq 4$ or $a_{ij}a_{ji}=\left(2\cos(\frac{\pi}{n_{ij}})\right)^2$}
\end{equation}
where $n_{ij}\in \ZZ_{>0}$ is the order of $s_is_j$ in $W_{ij}\subset W$. In 1971 E.~B.~Vinberg proved in \cite{vin71} that the condition \eqref{eq:nonnegativity for A} is equivalent to discreteness of the $W$-action on $\RR^I$.

The following conjecture refines Theorem \ref{th:nonneg} and asserts that this necessary condition is also sufficient.

\begin{conjecture}
\label{conj:nonneg} Let $A=(a_{ij})$ be a quasi-Cartan matrix such that \eqref{eq:nonnegativity for
A} holds for all $i\ne j$ (i.e., $W$ acts discretely on $\RR^I$). Then all Littlewood-Richardson
coefficients $c_{u,v}^w$ are nonnegative and all $p_{u,v}^w\in \RR_{\ge 0}[\alpha_i,i\in I]$.
\end{conjecture}

Note that  the quasi-Cartan matrices in the conjecture include all Cartan matrices of Kac-Moody
groups and those involved in Theorem \ref{th:nonneg}. In the case where $W=\langle s_1,s_2\rangle$
is a dihedral group of order $2n$ and $A$ is a $2\times 2$ symmetric matrix with
$a_{12}=a_{21}=2\cos(\frac{\pi}{n})$, the nonnegativity of $c_{u,v}^w$ has been verified by the
first author and M.~Kapovich in \cite[Corollary 13.7]{BerKap}.

We conclude Section \ref{sect:Definitions and main results} with the (yet conjectural) construction of (equivariant) {\it Littlewood-Richardson polynomials} $p_{u,v}^w({\bf A})$ and their strong positivity conjecture. Indeed, our definition of relative coefficients  makes sense for the universal Coxeter group $\widehat W$ generated by $s_i,i\in I$ acting on $\ZZ[{\bf A}]^I$, where $\ZZ[{\bf A}]=\ZZ[{\bf a}_{ij},i\ne j]$ so that each $p_{\ii',\ii''}^\ii$ belong to $\ZZ[{\bf A},\alpha_i,i\in I]$, i.e., $p_{\ii',\ii''}^\ii=p_{\ii',\ii''}^\ii({\bf A})$ is a polynomial of the {\it universal} Cartan matrix ${\bf A}=({\bf a}_{ij})$ and all $\alpha_i$ (i.e., it is a polynomial in  $|I|(|I|-1)+|I|=|I|^2$ variables since ${\bf a}_{ii}=2$ for $i\in I$).

Therefore, given any  Coxeter group $W$ generated by $s_i,i\in I$ we define a polynomial $p_{u,v}^\ii({\bf A})$ for any $u,v\in W$ and any $\ii\in I^m$ (with $\ell(u)+\ell(v)\ge m$) by the following analogue of \eqref{eq:sum of relatives T}:
\begin{equation}
\label{eq:sum of relative polynomials}
p_{u,v}^\ii({\bf A})=\sum p_{\ii',\ii''}^\ii({\bf A})
\end{equation}
with  the summation is over all  sub-sequences $\ii',\ii''$ of $\ii$ such that $\ii'\in R(u)$, $\ii''\in R(v)$. By the construction, if $A$ is a (quasi-)Cartan matrix compatible with $W$, then  $p_{u,v}^\ii({\bf A})|_{{\bf A}=A}=p_{u,v}^w$ for all $u,v,w\in W$ with $\ell(u)+\ell(v)\ge \ell(w)$ and all $\ii\in R(w)$.

We define the polynomial $p_{u,v}^\ii(t)\in \kk[t,\alpha_i,i\in I]$ by the specialization
$$p_{u,v}^\ii(t):=p_{u,v}^\ii((1+t)\cdot A-2t\cdot Id)\ ,$$
where $Id$ is the $I\times I$ identity matrix. By definition, $p_{u,v}^\ii(0)=p_{u,v}^w$ for each $\ii\in R(w)$.
Based on numerous examples (see Section \ref{sect:examples}), we expect a that stronger positivity result holds.

\begin{conjecture}
\label{conj:positive puvw(t)}
For any $A$ and $W$ as in Conjecture \ref{conj:nonneg} each polynomial $p_{u,v}^\ii(t)$ has nonnegative real coefficients.
\end{conjecture}

In fact, the coefficients of $p_{u,v}^\ii(t)$ belong to the sub-ring of $\RR$ generated by all $a_{ij}$, e.g., if $A$ is an integer matrix, then the above conjecture asserts that all $p_{u,v}^\ii(t)\in \ZZ_{\ge 0}[t,\alpha_i,i\in I]$.  We verify the conjecture in Section \ref{sect:positivity} in the case when $W$ is a free Coxeter group.  The conjecture has also been verified by computer calculations for all $p_{u,v}^\ii(t)$ in finite types $A_3$ and $A_4$.

\smallskip

The polynomials $p_{u,v}^\ii(t)$ depends on the choice $\ii\in R(w)$, however, it frequently happens that $p_{u,v}^\ii(t)=p_{u,v}^{\ii'}(t)$ for $\ii'\ne \ii$. Denote by $\sim$ the equivalence relation on $R(w)$ generated by pairs $(\ii,\ii'')$ where $\ii'$ is obtained from $\ii$ by switching a single pair of adjacent indices $i_k$ and $i_{k+1}$ such that $a_{i_k,i_{k+1}}=0$. We refer to this as the {\it commutativity} relation on $R(w)$ and, following \cite{Stem96}, we say that $w\in W$ is {\it fully commutative} if $R(w)$ is a single equivalence class.

\begin{conjecture}
\label{conj:fully commutative} For any $A$ and $W$ as in Conjecture \ref{conj:nonneg} we have for $\ii,\ii'\in R(w)$ such that $\ii\sim \ii'$:
$$p_{u,v}^\ii(t)=p_{u,v}^{\ii'}(t) \ .$$

In particular, if  $w$ is a fully commutative element in $W$, then $p_{u,v}^w(t)$ is well-defined.
\end{conjecture}

In particular, if $\ell(w)=\ell(u)+\ell(v)$ and $w$ is fully commutative, Conjectures \ref{conj:positive puvw(t)} and \ref{conj:fully commutative} imply that  $c_{u,v}^w(t):=p_{u,v}^w(t)$ is a well-defined polynomial in $t$ with nonnegative real coefficients.

If $W$ is a free Coxeter group, then the conjecture trivially is true.  We also verified the conjecture  in finite types $A_3$ and $A_4$.

\section{Twisted group algebras and generalized Littlewood-Richardson coefficients}
\label{sect:prelim}

We begin with some facts on twisted group  algebras.  Let $W$ be a monoid or a group and  let $Q$ be a commutative algebra over a field $\kk$.  For any  covariant $W$-action on $Q$ (i.e., such that $w(q_1\cdot q_2)=w(q_1)\cdot w(q_2)$) we define the {\it twisted} group (or, rather, monoidal) algebra $Q_W:=Q\rtimes \kk W$ generated by $Q$ and $W$ subject to the relations:
$$wq=w(q)\cdot w$$ for all $q\in Q$, $w\in W$.

We regard $Q_W$ as a $Q$-module via the left multiplication.  Note that one has a $\kk$-linear isomorphism
$$\iota:\kk W\otimes Q_W\widetilde \to Q_W\bigotimes\limits_Q Q_W$$
given by $w\otimes qw'\mapsto w\otimes qw'=q(w\otimes w')$. Taking into account that $\kk W\otimes Q_W$ is naturally a $\kk$-algebra, this isomorphism turns $Q_W\bigotimes\limits_Q Q_W$ into an associative algebra as well.
That is, the product in $Q_W\bigotimes\limits_Q Q_W$ is given by (cf. \cite[Section 4.14]{KK86}):
$$(q_1w_1\otimes q_2w_2)(q'_1w'_1\otimes q'_2w'_2)=(w_1\otimes q_1q_2w_2)(w'_1\otimes q'_1q'_2w'_2)=w_1w'_1\otimes q_1q_2w_2q'_1q'_2w'_2$$
(note however, that in general  the product in $Q_W$ and in  $Q_W\bigotimes\limits_Q Q_W$ is {\bf not}  $Q$-linear).


\begin{proposition}
\label{pr:delta} For any  commutative module algebra $Q$ over a monoid $W$ one has:

(a) The algebra $Q_W$ is a co-commutative coalgebra in the category of $Q$-modules with the coproduct $\delta:Q_W\to Q_W\bigotimes\limits_Q Q_W$ and the counit
$\varepsilon:Q_W\to Q$ given respectively by:
$$\delta(qw)=q\delta(w)=w\otimes qw,~\varepsilon(qw)=q$$
for all $q\in Q$, $w\in W$.

(b) The coproduct $\delta$ from (a) is a homomorphism of algebras.

(c) For any $x,y,z\in Q_W$ one has in the algebra $Q_W\bigotimes\limits_Q Q_W$:
\begin{equation}
\label{eq:diagonal product}
\delta(x)\cdot (y\otimes z)=x_{(1)}y\otimes x_{(2)}z\ ,
\end{equation}
where  $\delta(x)=x_{(1)}\otimes x_{(2)}$ in the Sweedler notation.



\end{proposition}

\begin{proof} Prove (a). First, verify that $\delta$ is  $Q$-linear. Indeed,
$$\delta(q_1q_2w)=w \otimes q_1q_2 w=q_1w \otimes q_2 w=q_1 (w \otimes q_2 w)=q_1 \delta(q_2w) \ .$$

Furthermore, the identity
$$(\delta\otimes 1)\delta(qw)=\delta(w)\otimes qw=w\otimes w\otimes qw=w\otimes \delta(qw)=(1\otimes \delta)\delta(qw)$$
verifies the coassociativity of $\delta$. Now verify the counit axiom:
$$(\varepsilon\otimes 1)\delta(qw)=(\varepsilon\otimes 1)(w\otimes qw)=qw=(1\otimes \varepsilon)(qw)=(1\otimes \varepsilon)\delta(qw)\ .$$
Finally, let us verify the co-commutativity. Let $\tau:Q_W\bigotimes\limits_Q Q_W$ be the permutation of factors. Then
$$\tau\delta(qw)=\tau(w\otimes qw)=qw\otimes w=w\otimes qw=\delta(qw)\ .$$
This proves (a).

Prove (b) now.   Indeed,
$$\delta((q_1w_1)(q_2w_2))=\delta((q_1w_1(q_2))w_1w_2)=w_1w_2\otimes (q_1w_1(q_2))w_1w_2$$
$$=w_1w_2\otimes q_1w_1q_2w_2=(w_1\otimes q_1w_1)(w_2\otimes q_2w_2)=\delta(q_1w_1)\delta(q_2q_2)\ .$$
This proves (b).

Prove (c) now. Indeed, it suffices to verify \eqref{eq:diagonal product} for $x=q_1w_1$, $y=q_2w_2$, $z=q_3w_3$:
$$\delta(q_1w_1)\cdot (q_2w_2\otimes q_3w_3)= (w_1\otimes q_1w_1)(w_2\otimes q_2q_3w_3)=w_1w_2\otimes q_1w_1q_2q_3w_3$$
$$=w_1(q_2)w_1w_2\otimes q_1w_1q_3w_3=w_1q_2w_2\otimes q_1w_1q_3w_3=w_1y\otimes q_1w_1z \ .$$
Part (c) is proved.

Therefore, the proposition is proved.
\end{proof}

\begin{remark} Note that $Q_W$ is not a bialgebra in the category of $Q$-modules because neither $Q_W$ nor $Q_W\bigotimes\limits_Q Q_W$ is not an algebra in this category.

\end{remark}

Clearly, if $M$ and $N$ are  free $Q$-modules,  and  $B_M$, $B_N$ are bases respectively in $M$ and $N$, then the set
$B_M\otimes B_N=\{b\otimes b'\,|\,b\in B_M,b'\in B_N\}$ is a basis of $M\bigotimes\limits_Q N$.

In particular, if $B$ is  a basis of $Q_W$, then set $B\otimes B\cong B\times B$ is a basis of $Q_W\bigotimes\limits_Q Q_W$. Using this, for each basis $B=\{x_w,w\in W\}$ of $Q_W$,
we define {\it generalized Littlewood-Richardson coefficients} $p_{u,v}^w\in Q$ by the formula:
\begin{equation}
\label{eq:generalized LR}
\delta(x_w)=\sum_{u,v\in W} p_{uv}^w x_u \otimes x_v \ .
\end{equation}

Dualizing this definition, we obtain the following result.

\begin{proposition}
\label{pr:dual}
Let $f:Q\to Q'$ be a homomorphism of commutative $\kk$-algebras such that the set $\{w\in W:f(p_{u,v}^w)\ne 0\}$ is finite for all $u,v\in W$.
Then there is a unique (associative) commutative $Q'$-algebra ${\mathcal A}_f$ with the free $Q'$-basis $\{\sigma_w\,|\,w\in W\}$ and the following multiplication table:
$$\sigma_u\sigma_v=\sum_{w\in W} f(p_{u,v}^w) \sigma_w$$
for all $u,v\in W$.

\end{proposition}

\begin{proof} We need the following result.

\begin{lemma}
\label{le:from coalgebra to algebra}
Let  $\delta:{\mathcal C}\to {\mathcal C}\bigotimes\limits_Q {\mathcal C}$ be a coalgebra in the category of  $Q$-modules.
Assume that $B$ is a  basis of ${\mathcal C}$ such that
$$\delta(b)=\sum_{b,b'\in B} p_{b',b''}^b b \otimes b' \ ,$$
where all $p_{b',b''}^b\in Q$.
Then for any homomorphism $f:Q\to Q'$  of commutative  $\kk$-algebras such that the set $\{b\in B:f(p_{b',b''}^b)\ne 0\}$ is finite for all $b',b''\in B$ there is a unique associative
$Q'$-algebra ${\mathcal A}={\mathcal A}_f$ with the basis $\{\sigma_b\,|\,b\in B\}$ and the following multiplication table:
$$\sigma_{b'}\sigma_{b''}=\sum_{b\in B} f(p_{b',b''}^b) \sigma_b$$
for all $b',b''\in B$. If, additionally, ${\mathcal C}$ was co-commutative, then ${\mathcal A}_f$ is commutative.
\end{lemma}
\begin{proof}
Indeed,
$$(\delta\otimes 1)\delta(b_1)= \sum_{b,b_4\in B} p^{b_1}_{b,b_4} \delta(b)\otimes {b_4}=\sum_{b,b_4} p^{b_1}_{b,b_4} (\sum_{b_2,b_3\in B} p_{b_2,b_3}^b {b_2} \otimes {b_3})\otimes {b_4}$$
$$=\sum_{b,b_2,b_3,b_4\in B}  p^{b_1}_{b,b_4} p_{b_2,b_3}^b {b_2} \otimes {b_3}\otimes {b_4} \ .$$
$$(1\otimes \delta)\delta({b_1})= \sum_{b,b_2\in B} p^{b_1}_{b_2,b} {b_2}\otimes \delta(b) =\sum_{b,b_2} p^{b_1}_{b_2,b}{b_2}\otimes (\sum_{b_2,b_3\in B} p_{b_3,b_4}^b {b_3} \otimes {b_4})$$
$$=\sum_{b,b_2,b_3,b_4\in B} p^{b_1}_{b_2,b}p_{b_3,b_4}^b {b_2}\otimes  {b_3} \otimes {b_4} \ .$$
Taking into the account that $B\otimes B\otimes B\cong B\times B\times B$ is the basis of ${\mathcal C}\bigotimes\limits_Q {\mathcal C}\bigotimes\limits_Q {\mathcal C}$,  the coassociativity of $\delta$ implies
$$\sum_{b\in B} p^{b_1}_{b,b_4} p_{b_2,b_3}^b=\sum_{b\in B} p^{b_1}_{b_2,b}p_{b_3,b_4}^b $$
for all $b_1,b_2,b_3,b_4\in B$.
Applying $f$, this implies that
$$(\sigma_{b_2}\sigma_{b_3})\sigma_{b_4}=\sum_{b\in B}  f(p_{b_2,b_3}^b) \sigma_b\sigma_{b_4}=\sum_{b,b_1\in B}  f(p_{b_2,b_3}^b p^{b_1}_{b,b_4}) \sigma_{b_1}=\sum_{b,b_1\in B}  f(p^{b_1}_{b_2,b}p_{b_3,b_4}^b) \sigma_{b_1}$$
$$=\sum_{b\in B}  f(p_{b_3,b_4}^b) \sigma_{b_2}\sigma_b=\sigma_{b_2}(\sigma_{b_3}\sigma_{b_4}) \ .$$
Finally, note that co-commutativity of ${\mathcal C}$ is equivalent to $\tau\delta(b)=\delta(b)$ for all $b\in B$, i.e.,
$$\sum_{b',b''} b''\otimes p_{b',b''}^b b'=\sum_{b',b''} p_{b',b''}^b b' \otimes b''\ ,$$
i.e, $p_{b'',b'}^b =p_{b',b''}^b$ for all $b,b',b''\in B$. This implies that ${\mathcal A}_f$ is commutative.
The lemma is proved. \end{proof}

Taking ${\mathcal C}=Q_W$, $B=\{x_w,w\in W\}$,  we finish the proof of Proposition \ref{pr:dual}.
\end{proof}

In the assumptions of Proposition \ref{pr:dual} let $\langle \cdot,\cdot \rangle:{\mathcal A}_f\times Q_W\to Q'$ be the $Q'$-linear pairing given by
\begin{equation}
\label{eq:pairing}
\langle q'\sigma_u,q x_v\rangle =\delta_{u,v}\cdot  q'f(q)
\end{equation}
for all $u,v\in W$, $q\in Q$, $q'\in Q'$.

\begin{corollary}
\label{cor:Billey homomorphism}

In the assumptions of Proposition \ref{pr:dual}, we have

(a) The pairing \eqref{eq:pairing} satisfies:
$$\langle ab,x\rangle=\langle a\otimes b,\delta(x)\rangle= \langle a,x_{(1)}\rangle \langle b,x_{(2)}\rangle$$
for all $a,b\in {\mathcal A}_f$, $x\in Q_W$, where $\delta(x)=x_{(1)}\otimes x_{(2)}$ in the Sweedler notation.

(b) For each $w\in W$ the assignment $a\mapsto \langle a,w\rangle$, $a\in {\mathcal A}_f$ is a $Q'$-algebra homomorphism
$$\xi_w:{\mathcal A}_f\to Q'$$

\end{corollary}

\begin{proof}
Prove (a). It suffices to verify the identity for $a=\sigma_u$, $b=\sigma_v$, $x=x_w$. Indeed,
$$\langle \sigma_u\otimes \sigma_v,\delta(x_w)\rangle=\langle \sigma_u\otimes \sigma_v,\sum_{u',v'\in W} p_{u',v'}^w x_{u'}\otimes x_{v'}\rangle=
\sum_{u',v'}f(p_{u',v'}^w)\langle \sigma_u,x_{u'}\rangle \langle\sigma_v,x_{v'}\rangle$$
$$=\sum_{u',v'}f(p_{u',v'}^w)\delta_{u,u'}\delta_v,x_v'=\langle\sum_{w'} p_{u,v}^w\delta_{w,w'} \sigma_{w'},x_w\rangle=\langle \sigma_u\sigma_v,x_w\rangle .$$

This proves (a).

Prove (b). The $Q'$-linearity of $\xi_w$ is obvious. Prove that $\xi_w$ respects multiplication. Indeed, for all $w\in W$, $a,b\in {\mathcal A}_f$ we have
$$\xi_w(ab)=\langle ab,w\rangle=\langle a\otimes b,\delta(w)\rangle=\langle a\otimes b,w\otimes w\rangle$$
$$=\langle a,w\rangle\langle b,w\rangle=\xi_w(a)\xi_w(b) \ .$$
This proves (b).

The corollary is proved.
\end{proof}

In what follows (Proposition \ref{pr:long recursion}), we introduce the analogues of $p_{uv}^w$ which we refer to as {\it relative} (generalized) Littlewood-Richardson coefficients.


\begin{definition}
\label{def:tame}
Given a subset $S=\{s_i,i\in I\}$ of $W\setminus\{1\}$, we say that a subset $X=\{x_i,i\in I\}$ of
$Q_W$ is $S$-{\it tame} if $X$ is a basis of the (free) $Q$-module $\sum_{i\in I} Q(s_i-1)$.
\end{definition}


For an $S$-tame set $X$ we have:
\begin{equation}
\label{eq:tame coefficients}
x_i=\sum_{j\in I} r_{ij} (s_j-1)\quad \text{and}\quad s_i=1+\sum_{j\in I} q_{ij} x_j
\end{equation}
for some mutually inverse $I\times I$ matrices $(q_{ij})$ and $(r_{ij})$ over $Q$.

For any sequence $\ii:=(i_1,\ldots,i_m)\in I^m$, define a monomial $x_\ii \in Q_W$ by:
$$x_\ii:=x_{i_1}\cdots x_{i_m}$$
with the convention that $x_\emptyset=1$.


The following fact is obvious.

\begin{lemma}
\label{le:counit action}
There is a unique left action of $Q_W$ on $Q$ such that
$$(qw)(q')=q\cdot w(q')$$
for $q,q'\in Q$, $w\in W$. The $Q_W$-action on $Q$ satisfies for all $x\in Q_W$:
$$x(q)=\varepsilon(xq)\ , $$
where $\varepsilon:Q_W\to Q$ is the counit from Proposition \ref{pr:delta}(a).
\end{lemma}

The following result is  a generalization of  Kostant-Kumar recursion from \cite{KK86}

\begin{proposition}
\label{pr:long recursion} For any $S$-tame set $X=\{x_i,i\in I\}$ in $Q_W$ we have:
\begin{equation}
\label{eq:pijk}
\delta(x_\ii)=\sum_{\ii',\ii''} p_{\ii',\ii''}^\ii\ x_{\ii'}\otimes x_{\ii''}
\end{equation}
where the summation is over all pairs of sequences $(\ii',\ii'')\in I^{m'}\times I^{m''}$ with $m',m''\le m$  and the coefficients $p_{\ii',\ii''}^{\ii}$ are determined recursively by
$p_{\emptyset,\emptyset}^{\emptyset}=1$
and:
\begin{equation}
\label{eq:recursion LR}
p_{\ii',\ii''}^\ii =x_{i_1}(p_{\ii',\ii''}^{\widetilde \ii})+\sum_{j\in I} r_{i_1,j} \Big(q_{j,i_1'}s_j(p_{\widetilde\ii',\ii''}^{\widetilde\ii})+ q_{j,i''_1}s_j(p_{\ii',\widetilde\ii''}^{\widetilde\ii})
+q_{j,i'_1}q_{j,i''_1}s_j(p_{\widetilde\ii',\widetilde \ii''}^{\widetilde\ii})\Big) \ ,
\end{equation}
if $m\ge 1$, where $\widetilde \ii$ stands for a sequence obtained from $\ii$ by deleting the first entry $i_1$.

\end{proposition}

\begin{proof} First, compute
$\delta(x_i)$. Indeed, using \eqref{eq:tame coefficients}, we obtain:
$$\delta(x_i)=\sum_{j\in I} r_{ij}(s_j\otimes s_j-1\otimes 1)=\sum_{j\in I} r_{ij}\left((s_j-1)\otimes 1+s_j\otimes (s_j-1)\right)$$
$$=x_i\otimes 1+ \sum_{j,i''\in I}  r_{ij}q_{j,i''}s_j\otimes x_{i''}= x_i\otimes 1+1\otimes x_i+ \sum_{j,i',i''\in I}  r_{ij} q_{j,i'}q_{j,i''} x_{i'}\otimes x_{i''}\ .$$

We need the following result.

\begin{lemma}
\label{le:demazure commutator}
For each $i\in I$, $p\in Q$ we have:
$$x_ip=x_i(p)+\sum_{j,i'} r_{ij} q_{j,i'}s_j(p)x_{i'}\ .$$
\end{lemma}
\begin{proof} Indeed,

\noindent $x_i p=\sum\limits_j r_{ij}(s_j-1) p =\sum\limits_j r_{ij} ((s_j-1)(p)+s_j(p)(s_j-1))=x_i(p)+\sum\limits_{j,i'} r_{ij} s_j(p)q_{j,i'}x_{i'}$.
\end{proof}

Furthermore, for $\ii=(i_1,\ldots,i_m)\in I^m$ denote $\tilde \ii={\ii \setminus \{i_1\}}:=(i_2,\ldots,i_m)$ so that $x_\ii=x_i x_{\tilde \ii}$. Therefore,  using the inductive hypothesis in the form:
$$\delta(x_{\tilde \ii})=\sum_{\tilde\ii',\tilde\ii''} p_{\tilde\ii',\tilde\ii''}^{\tilde \ii} x_{\tilde\ii'}\otimes x_{\tilde\ii''} \ ,$$
we obtain using the above computation of $\delta(x_i)$, Proposition  \ref{pr:delta}(c), and Lemma \ref{le:demazure commutator}:
$$\delta(x_\ii)=\delta(x_{i_1})\delta(x_{\tilde \ii})
=(x_{i_1}\otimes 1+\sum_{j,i_1''}  r_{i_i,j} q_{j,i''} s_j\otimes x_{i_1''})
\sum_{\tilde\ii',\tilde\ii''} p_{\tilde\ii',\tilde\ii''}^{\tilde \ii} x_{\tilde\ii'}\otimes x_{\tilde\ii''}$$
$$=\sum_{\tilde\ii',\tilde\ii''}  x_{i_1}p_{\tilde\ii',\tilde\ii''}^{\tilde \ii} x_{\tilde\ii'}\otimes x_{\tilde\ii''}
+\sum_{j,i''_1,\tilde\ii',\tilde\ii''} r_{i_1,j}q_{j,i''_1}s_jp_{\tilde\ii',\tilde\ii''}^{\tilde \ii} x_{\tilde\ii'}\otimes x_{i_1''}x_{\tilde\ii''}$$
$$=\sum_{\tilde\ii',\tilde\ii''}  x_{i_1}p_{\tilde\ii',\tilde\ii''}^{\tilde \ii} x_{\tilde\ii'}\otimes x_{\tilde\ii''}
+\sum_{j,i''_1,\tilde\ii',\tilde\ii''} r_{i_1,j}q_{j,i''_1}(s_j(p_{\tilde\ii',\tilde\ii''}^{\tilde \ii})+p_{\tilde\ii',\tilde\ii''}^{\tilde \ii}(s_j-1)) x_{\tilde\ii'}\otimes x_{i_1''}x_{\tilde\ii''}$$
$$=\sum_{\tilde\ii',\tilde\ii''}  \Big(x_{i_1}(p_{\tilde\ii',\tilde\ii''}^{\tilde \ii}) x_{\tilde\ii'}\otimes x_{\tilde\ii''}
+\sum_{j,i_1'} r_{i_1,j} q_{j,i_1'}s_j(p_{\tilde\ii',\tilde\ii''}^{\tilde \ii})x_{i_1'}x_{\tilde\ii'}\otimes x_{\tilde\ii''}$$
$$+\sum_{j,i''_1} r_{i_1,j}q_{j,i''_1}s_j(p_{\tilde\ii',\tilde\ii''}^{\tilde \ii}) x_{\tilde\ii'}\otimes x_{i_1''}x_{\tilde\ii''}
+\sum_{j,i'_1,i''_1} r_{i_1,j}q_{j,i''_1}s_j(p_{\tilde\ii',\tilde\ii''}^{\tilde \ii})q_{j,i'_1} x_{i'_1}x_{\tilde\ii'}\otimes  x_{i_1''}x_{\tilde\ii''}\Big)
$$
$$=\sum_{\ii',\ii''}  p_{\ii',\ii''}^{\ii} x_{\ii'}\otimes x_{\ii''}\ .$$

This proves \eqref{eq:pijk}. Therefore, Proposition \ref{pr:long recursion} is proved.

\end{proof}

We refer to $p_{\ii',\ii''}^{\ii}$ as the {\it relative Littlewood-Richardson} coefficients. Since the $x_\ii$ are not linearly independent in general, the relative Littlewood-Richardson are not unique. Nevertheless,
we can restore the uniqueness by replacing $W$ with a larger monoid as follows.

\begin{theorem}
\label{th:folding}
(Folding principle)
Let $Q$ (rep. $\widehat Q$) be a commutative module algebra over a monoid $W$ (resp. $\widehat W$). Let  $\varphi_-: \widehat W\to W$ be a homomorphism of monoids and let $\varphi_+:\widehat Q\to Q$ be an algebra homomorphism commuting with the $\widehat W$-action. Then:

(a) there exists a unique algebra homomorphism $\varphi:\widehat Q_{\widehat W}\to Q_W$ such that
$$\varphi|_{1\rtimes \kk\widehat W}=\varphi_-,~\varphi|_{\widehat Q\rtimes 1}=\varphi_+$$
and the following diagram is  commutative:
\begin{equation}
\label{eq:commutative diagram}
\xymatrix{\widehat Q_{\widehat W}\ \ar[r]^-{\widehat \delta} \ar[d]_-{\varphi}& \widehat Q_{\widehat W}\bigotimes_{\widehat Q} \widehat Q_{\widehat W} \ar[d]^-{\varphi\otimes\varphi} \\
 Q_W\ \ar[r]^-{\delta} & Q_W\bigotimes_Q  Q_W }\end{equation}


(b)  For any $S$-tame set $X=\{x_i,i\in I\}$ in $Q_W$, any $\widehat S$-tame set $\widehat X=\{\widehat x_k,k\in K\}$ in $\widehat Q_{\widehat W}$, and  a map $\pi:K\to  I$ such that
\begin{equation}
\label{eq:folding of tame sets}
\varphi(\widehat x_k)=x_{\pi(k)}
\end{equation}
for all $k\in K$ one has (for all $\ii\in I^m$, $\ii'\in I^{m'}$, $\ii''\in I^{m''}$ with $m',m''\le m$):
\begin{equation}
\label{eq:folded general coefficients}
p_{\ii',\ii''}^\ii=\sum \varphi(\widehat p_{\jj',\jj''}^{\,\jj}) \ ,
\end{equation}
where $\jj\in K^m$ is any sequence such that $\pi(\jj)=\ii$ and the summation is over all sequences $\jj'\in K^{m'}$, $\jj''\in K^{m''}$ such that $\pi(\jj')=\ii'$, $\pi(\jj'')=\ii''$, where  $\widehat p_{\jj',\jj''}^{\,\jj}$  are relative Littlewood-Richardson coefficients for $\widehat Q_{\widehat W}$.

\end{theorem}

\begin{proof} Prove (a). We verify the first assertion. Define a linear map $\varphi:\widehat Q_{\widehat W}\to Q_W$ by:
$$\varphi(\widehat q\widehat w)=\varphi_+(\widehat q)\varphi_-(\widehat w) \ .$$
In order to prove that $\varphi$ is an algebra homomorphism it suffices to show that $\varphi(\widehat w\widehat q)=\varphi(\widehat w)\varphi(\widehat q)$ for all  $\widehat q\in \widehat Q$, $\widehat w\in \widehat W$.
Indeed,
$$\varphi(\widehat w\widehat q)=\varphi(\widehat w(\widehat q)\cdot w)=\varphi_+(\widehat w(\widehat q))\cdot \varphi_-(\widehat w)$$
$$=(\varphi_-(\widehat w))(\varphi_+(\widehat q))\varphi_-(\widehat w)=\varphi_-(\widehat w)\cdot \varphi_+(\widehat q)=\varphi(\widehat w)\varphi(\widehat q)\ .$$

Now verify the commutativity of the diagram \eqref{eq:commutative diagram}. Indeed,

$$\delta(\varphi(\widehat q\widehat w))=\delta(\varphi(\widehat q)\varphi(\widehat w))=\varphi(\widehat w)\otimes \varphi(\widehat q\widehat w)=(\varphi\otimes \varphi)(\widehat w\otimes \widehat q\widehat w)=(\varphi\otimes \varphi)\widehat \delta(\widehat q\widehat w) \ .$$
This proves (a).

Prove (b) now.
We need the following result.

\begin{lemma}
\label{le:free group}
Let $\widehat W$ be the free monoid generated by $S\subset W$, then:

(i)  One has a (unique)  algebra homomorphism $\varphi:Q_{\widehat W}\to Q_W$ such that $\varphi|_S=Id_S$ and $\varphi|_Q=Id_Q$;

(ii) for any $S$-tame set $X=\{x_i\in I\}$ in $Q_W$  the set
$$\widehat X=\{\widehat x_i=\varphi^{-1}(x_i)\cap  \sum_{s\in S} Q\cdot (s-1),i\in I\}$$
is $S$-tame in $Q_{\widehat W}$;

(iii) The monomials $\widehat x_\ii=\widehat x_{i_1}\cdots \widehat x_{i_m}$
are $Q$-linearly independent in $Q_{\widehat W}$.

 (iv) Each relative Littlewood-Richardson coefficient $\widehat p_{\ii',\ii''}^{\,\ii}$ for $Q_{\widehat W}$ with respect to $\widehat X$
 equals to the relative Littlewood-Richardson coefficient $p_{\ii',\ii''}^\ii$ for $Q_W$ and is uniquely determined by the expansion \eqref{eq:pijk}:
\begin{equation}
\label{eq:pijk free}
\widehat \delta(\widehat x_\ii)=
\sum_{\ii',\ii''} p_{\ii',\ii''}^\ii
\widehat x_{\ii'}\otimes \widehat x_{\ii''} \ .
\end{equation}
\end{lemma}

\begin{proof} Indeed, $\varphi:Q_{\widehat W}\to Q_W$ as an algebra homomorphism by Theorem \ref{th:folding}(a). This verifies (i).
Furthermore, since the restriction of $\varphi$ to $\sum_{s\in S} Q\cdot (s-1)$ is the identity map, one can trivially lift each
$x_i\in X$ to a unique element $\widehat x_i\in Q_{\widehat W}$ such that $\varphi(\widehat x_i)=x_i$. This verifies (ii).
Let us show that all monomials $x_\ii$ form a basis in  the subalgebra $Q_{\widehat W_+}$ of $Q_{\widehat W}$ generated by $S$ and $Q$.
Indeed, $Q_{\widehat W_+}$ has a $Q$-basis of the form $w_\ii=s_{i_1}\cdots s_{i_m}$, where $\ii\in I^m$, $m\ge 0$ runs over all sequences.
Denote by ${\mathcal A}_{\le n}$ the $Q$-submodule of $Q_{\widehat W_+}$ spanned by all $w_\ii$, $\ii\in I^m$, $m\le n$.
Also denote by ${\mathcal B}_{\le n}$ the $Q$-submodule of $Q_{\widehat W_+}$ spanned by all $x_\ii$, $\ii\in I^m$, $m\le n$.
Let us show that ${\mathcal A}_{\le n}={\mathcal B}_{\le n}$. Clearly, both ${\mathcal A}_{\le n}$ defines a filtration on the algebra
$Q_{\widehat W_+}$ such that ${\mathcal A}_{\le n}=({\mathcal A}_{\le 1})^n$.
Note  that $$x_iq\in Q+\sum_{j\in I} Q\cdot (s_j-1)\subseteq Q+\sum_{j\in I} Q\cdot x_j={\mathcal B}_{\le 1}$$ for each $i\in I$, $q\in Q$.
This implies that  ${\mathcal B}_{\le n}$ is also a filtration on the algebra $Q_{\widehat W_+}$ such that ${\mathcal B}_{\le n}=({\mathcal B}_{\le 1})^n$.
Since  ${\mathcal A}_{\le 1}={\mathcal B}_{\le 1}$ by definition of the tame set $\widehat X$, we see that ${\mathcal A}_{\le n}={\mathcal B}_{\le n}$.
This proves linear independence of all $x_\ii$ and, thus, verifies part (iii).
Finally, in view of (iii), the coefficients $ p_{\ii',\ii''}^\ii$ are uniquely determined by:
$$\widehat \delta(\widehat x_\ii)=
\sum_{\ii',\ii''} \widehat p_{\ii',\ii''}^{\,\ii}
\widehat x_{\ii'}\otimes \widehat x_{\ii''} \ .$$
This implies that $\widehat p_{\ii',\ii''}^{\,\ii}=p_{\ii',\ii''}^{\,\ii}$ for all relevant $\ii,\ii',\ii''$ because both families
$\{p_{\ii',\ii''}^{\,\ii}\}$ and $\{\widehat p_{\ii',\ii''}^{\,\ii}\}$
satisfy the same recursion \eqref{eq:tame coefficients}. This verifies (iv).

The lemma is proved.
\end{proof}

Furthermore, we prove \eqref{eq:folded general coefficients}. Using  using Lemma \ref{le:free group},
without loss of generality we may assume that $W$ is a free monoid generated by $S=\{s_i,i\in I\}$ and $\widehat W$ is a free monoid generated by $\widehat S=\{\widehat s_1,\ldots,\widehat s_m\}$.
In particular, one has a unique expansion
$$\delta(x_\ii)=\sum_{\ii',\ii''} p_{\ii',\ii''}^\ii x_{\ii'}\otimes  x_{\ii''}$$
where the summation is over all sequences $\ii'\in I^{m'}$ and $\ii''\in I^{m''}$, $m',m''\le m$ and
$$\widehat \delta(\widehat x_\jj)=\sum_{\jj',\jj''} \widehat p_{\jj',\jj''}^{\,\jj}  \widehat x_{\jj'}\otimes \widehat x_{\jj''} \ ,$$
where the summation is over all sequences $\jj'\in K^{m'}$ and $\jj''\in K^{m''}$,  $m',m''\le m$.  Since the diagram \eqref{eq:commutative diagram} is commutative, we obtain:
$$\delta(\varphi(\widehat x_\jj))=(\varphi\otimes \varphi)(\widehat \delta(\widehat x_\jj))$$
Since $\widehat \varphi(\widehat x_{\jj'})=x_{\pi(\jj')}$ for any $\jj'\in K^{m'}$, we obtain:
$$\delta(x_\ii)=\sum_{\jj',\jj''}\varphi(\widehat p_{\jj',\jj''}^{\,\jj})  x_{\varphi(\jj')}\otimes x_{\varphi(\jj'')} \ .$$
Since the tensors $x_{\ii'}\otimes x_{\ii''}$ are $Q$-linearly independent, by collecting the coefficient of each $x_{\ii'}\otimes x_{\ii''}$ we obtain \eqref{eq:folded general coefficients}.
The theorem is proved.
\end{proof}

Dualizing the assertions of Theorem \ref{th:folding}, we obtain the following result.

\begin{proposition}
\label{pr:dual folding}
In the assumption of Theorem \ref{th:folding}, let $\{x_w,w\in W\}$ (resp. $\{\widehat x_{\widehat w},\widehat w\in \widehat W\}$) be a
$Q$-linear (resp. $\widehat Q$-linear) basis of $Q_W$ (resp. of $\widehat Q_{\widehat W}$) such that  for all $\widehat w\in \widehat W$:
$$\varphi(\widehat x_{\widehat w})
=\begin{cases}
x_w & \text{if $\widehat w\in \widehat W_w$}\\
0 & \text{if $\widehat w\notin \widehat W_w$}\\
\end{cases}
$$
where $\widehat W_w\subset \widehat W$, $w\in W$ is  a finite subset of $\widehat W$.
Then:

(a) For all $u,v,w\in W$ and each $\widehat w\in \widehat W_w$  one has:
\begin{equation}
\label{eq:LR folding}
p_{u,v}^w=\sum_{(\widehat u,\widehat v)\in \widehat W_u\times  \widehat W_v} \varphi(\widehat p_{\widehat u,\widehat v}^{\,\widehat w})\ .
\end{equation}

(b) Assume additionally that  $f:Q\to Q'$ and $\widehat \varphi:Q'\to \widehat  Q'$ are homomorphisms of commutative $\kk$-algebras such that:

$\bullet$ the set $\{w\in W:f(p_{u,v}^w)\ne 0\}$ is finite for all $u,v\in W$;

$\bullet$ the set $\{\widehat w\in W:\widehat f(\widehat p_{\widehat u,\widehat v}^{\,\widehat w})\ne 0\}$ is finite for all $\widehat u,\widehat v\in \widehat W$, where $\widehat f=\widehat \varphi \circ f \circ \varphi$.


Then, in the notation of Proposition \ref{pr:dual},  the association
$$\sigma_w\mapsto \sum_{\widehat w\in \widehat W_w}  \widehat \sigma_{\widehat w}$$
defines a homomorphism of $\kk$-algebras $\varphi^*: {\mathcal A}_f\to \widehat  {\mathcal A}_{\widehat f}$ such that $\varphi^*|_{Q'}=\widehat \varphi$.

\end{proposition}

\begin{proof} Prove (a) Indeed, as in the proof of \eqref{eq:folded general coefficients}, applying $\varphi$ to the expansion \eqref{eq:generalized LR} for  $\widehat Q_{\widehat W}$ and using commutativity of \eqref{eq:commutative diagram}, we obtain \eqref{eq:LR folding}.

Prove (b) now. Indeed,
$$\varphi^*(\sigma_u\sigma_v)=\varphi^*(\sigma_u\sigma_v)=\sum_{w\in W} \varphi^*(f(p_{u,v}^w)\sigma_w)=\sum_{w\in W} (\widehat \varphi\circ f)(p_{u,v}^w)\varphi^*(\sigma_w)$$
$$=\sum_{w\in W,\widehat w\in \widehat W_w}(\widehat \varphi\circ f)(p_{u,v}^w) \widehat \sigma_{\widehat w}=\sum_{w\in W,\widehat w\in \widehat W_w,\atop (\widehat u,\widehat v)\in \widehat W_u\times  \widehat W_v} \widehat f(p_{\widehat u,\widehat v}^{\widehat w}) \widehat \sigma_{\widehat w}
=\sum_{(\widehat u,\widehat v)\in \widehat W_u\times  \widehat W_v}  \widehat \sigma_{\widehat u}\widehat \sigma_{\widehat v} =\varphi^*(\sigma_u)\varphi^*(\sigma_v)\ .$$
This proves (b).

The proposition is proved.
\end{proof}


Now we will compute all relative Littlewood-Richardson coefficients for our main class of $S$-tame sets  $X=\{x_i,i\in I\}$, where
\begin{equation}
x_i=\alpha_i^{-1}(s_i-1)
\end{equation}
where $\alpha_i$ are some invertible elements of $Q$ and $s_i\in W\setminus \{1\}$. We sometimes refer to the elements $x_i$ as  {\it Demazure elements}.

\begin{corollary}
\label{cor:demazure} For any $S=\{s_i,i\in I\}$ the Demazure elements $x_i$, $i\in I$ and their monomials $x_\ii$, $\ii\in I^m$ satisfy:
$$\delta(x_\ii)=\sum_{\ii',\ii''} p_{\ii',\ii''}^\ii\ x_{\ii'}\otimes x_{\ii''}$$
where the summation is over all pairs of subsequences $(\ii',\ii'')$ of $\ii$  and the relative Littlewood-Richardson coefficients $p_{\ii',\ii''}^{\ii}$ are determined recursively by $p_{\emptyset,\emptyset}^{\emptyset}=1$ and:
\begin{equation}
\label{eq:pijk demazure}
p_{\ii',\ii''}^\ii =x_{i_1}(p_{\ii',\ii''}^{\widetilde\ii})+\delta_{i_1,i_1'}s_{i_1}(p_{\widetilde\ii',\ii''}^{\widetilde\ii})
+ \delta_{i_1,i''_1}s_{i_1}(p_{\ii',\widetilde\ii''}^{\widetilde\ii})
+\delta_{i_1,i'_1}\delta_{i_1,i''_1}\alpha_{i_1}s_{i_1}(p_{\widetilde\ii',\widetilde \ii''}^{\widetilde\ii})
\end{equation}
if $m\ge 1$, where $\widetilde \ii$ stands for a sequence obtained from $\ii$ by deleting the first entry $i_1$.

\end{corollary}

\begin{proof} Note that for the Demazure elements $x_i$,  in the notation of \eqref{eq:tame coefficients}, we have
$r_{ij}=\delta_{ij} \alpha_i^{-1},~q_{ij}=
\delta_{ij} \alpha_i$ for $i,j\in I$.
Then the recursion \eqref{eq:recursion LR} becomes \eqref{eq:pijk demazure}. Finally, it follows from \eqref{eq:pijk demazure} (by induction in $m$) that
$p_{\ii',\ii''}^\ii=0$ if either $\ii'$ or $\ii''$ is not a sub-sequence of $\ii$.
\end{proof}

\begin{proposition}
\label{pr:A-->Aii}
For any $S=\{s_i,i\in I\}$, a Demazure $S$-tame set $X=\{x_i,i\in I\}\subset Q_W$ and any subalgebra $R\subset Q$ such that all $p_{\ii',\ii''}^\ii\in R$ one has:

(a) There exists a commutative $R$-algebra ${\mathcal A}_{X,R}$ with the basis $\{\sigma_\ii\}$, where $\ii$ runs over all sequences in $I^m$, $m\ge 0$ such that:
$$\sigma_{\ii'} \sigma_{\ii''}= \sum_\ii p_{\ii',\ii''}^\ii \sigma_\ii$$
for any  sequences $\ii'\in I^{m'}$ and $\ii''\in I^{m''}$,
where the summation over all sequences $\ii\in I^m$, $m\ge 0$ containing $\ii'$ and $\ii''$ as sub-sequences and such that $m\le m'+m''$.

(b) For each given $\ii\in I^m$ there exists an associative algebra ${\mathcal A}_{X,\ii,R}$ with the basis
$\{\sigma_\jj^{(\ii)}\}$, where $\jj$ runs over all subsequences of $\ii$ such that:
$$\sigma_{\jj'}^{(\ii)} \sigma_{\jj''}^{(\ii)}= \sum_\jj p_{\jj',\jj''}^\jj \sigma_\jj^{(\ii)}$$
for any subsequences $\jj'$, $\jj''$ of $\ii$,
where the summation over all subsequences $\jj$ of $\ii$.

(c) For each $\ii\in I^m$, $m\ge 0$, the association
\begin{equation}
\label{eq:A-->Aii}
\sigma_\jj\mapsto
\begin{cases}
\sigma_\jj^{(\ii)} & \text{if $\jj$ is  a subsequence of $\ii$}\\
0 & \text{otherwise}\\
\end{cases}
\end{equation}
defines a surjective homomorphism of $R$-algebras
$\pi_\ii:{\mathcal A}_{X,R}\twoheadrightarrow {\mathcal A}_{X,\ii,R}$.
\end{proposition}

\begin{proof} Let now $\widehat W$ be the free monoid generated by $S=\{s_i,i\in I\}$.
Then the algebra ${\mathcal A}_{X,R}$ is dual (over $R$) of the coalgebra $Q_{\widehat W}$, i.e., is obtained by  combining Lemma \ref{le:from coalgebra to algebra} (with $Q=Q'$, $f=id_Q$) and Theorem \ref{th:folding}(b). The commutativity of ${\mathcal A}_{X,R}$ follows from the symmetry $$p_{\ii'',\ii'}^\ii=p_{\ii',\ii''}^\ii\ ,$$
which directly follows from the recursive definition \eqref{eq:recursion LR}. This proves (a).

Prove (b). Denote by  ${\mathcal C}_\ii\subset Q_{\widehat W}$ the $Q$-linear span of all $x_\jj$ where $\jj$ runs over all subsequences of $\ii$. It follows from Corollary \ref{cor:demazure}  that  ${\mathcal C}_\ii$ is closed under the coproduct $\delta:Q_{\widehat W}\to Q_{\widehat W}\bigotimes\limits_Q Q_{\widehat W}$ and thus ${\mathcal C}_\ii$ is a coalgebra in the category of $Q$-modules. Then applying Lemma \ref{le:from coalgebra to algebra} once again,  we finish the proof of (b).

Prove (c) Indeed, the natural inclusion ${\mathcal C}_\ii\hookrightarrow Q_{\widehat W}$ is a homomorphism of coalgebras in the category of $Q$-modules. Its dual is a homomorphism of algebras given by  \eqref{eq:A-->Aii}. This proves (c)

The proposition is proved.
\end{proof}

We say that an index $i_k$ of $\ii=(i_1,\ldots,i_m)\in I^m$ is {\it repetition-free} if $i_\ell\ne i_k$ for all $\ell\in [m]\setminus \{k\}$. And we say that $\ii$ is repetition-free if each $i_k$, $k\in [m]$ is repetition-free, i.e., all indices $i_1,\ldots,i_m$ are distinct (equivalently, $|\{\ii\}|=m$, where  $\{\ii\}=\{i_1,\ldots,i_m\}\subset I$ denotes the underlying set).

For any subsequences $\ii'$, $\ii''$ of $\ii$ and a repetition-free index $i$ of $\ii$ we define the $f_i:=f_i(\ii,\ii',\ii'')\in Q_W$ by
\begin{equation}
\label{eq:fi}
f_i=\begin{cases}
\alpha_is_i& \text{if $i\in \{\ii'\}\cap \{\ii''\}$}\\
x_i& \text{if $i\notin \{\ii'\}\cup \{\ii''\}$}\\
s_i& \text{otherwise}
\end{cases}
\end{equation}

The following result computes all relative Littlewood-Richardson coefficients in the repetition-free case.

\begin{proposition}
\label{pr:free}
Assume that the indices $i_1,i_2,\ldots,i_k$ of $\ii$ are repetition-free. Then  for any subsequences $\ii'$, $\ii''$ of $\ii$ we have:
\begin{equation}
\label{eq:free composition} p_{\ii',\ii''}^\ii=f_{i_1}f_{i_2}\cdots f_{i_k}(p^{(i_{k+1},\ldots,i_m)}_{\ii'\setminus\{i_1,\ldots,i_k\},\ii''\setminus\{i_1,\ldots,i_k\}})
\end{equation}
In particular, if $\ii$ is repetition-free, then
$p_{\ii',\ii''}^\ii$ depends only on $\ii$, $\{\ii'\}\cap \{\ii''\}$, and $\{\ii'\}\cup \{\ii''\}$.
\end{proposition}

\begin{proof} If the index $i_1$ is repetition-free, the recursion \eqref{eq:pijk demazure}
drastically simplifies:
\begin{equation}
\label{pr:free one step}
p_{\ii',\ii''}^\ii =\begin{cases}
\alpha_{i_1}s_{i_1}\big(p_{\ii'\setminus \{i_1\},\ii''\setminus \{i_1\}}^{\ii \setminus \{i_1\}}\big)& \text{if $i'_1=i''_1=i_1$}\\
s_{i_1}(p_{\ii'\setminus \{i_1\},\ii''}^{\ii \setminus \{i_1\}}) & \text{if $i'_1=i_1\not =i''_1$}\\
s_{i_1}(p_{\ii',\ii''\setminus \{i_1\}}^{\ii \setminus \{i_1\}})& \text{if $i''_1=i_1\not =i'_1$}\\
x_{i_1}\big( p_{\ii',\ii''}^{\ii \setminus \{i_1\}}\big) & \text{if $i_1\notin\{i'_1,i''_1\}$}
\end{cases}=f_{i_1}(p^{(i_2,\ldots,i_m)}_{\ii'\setminus\{i_1\},\ii''\setminus\{i_1\}})
\end{equation}
because in each of the cases in \eqref{eq:pijk demazure}, all  non-leading terms are zero (for instance, if $i'_1=i''_1=i_1$, then neither $\ii'$ not $\ii''$ is a sub-sequence of $\ii\setminus \{i_1\}$).
This proves \eqref{eq:free composition} by induction.

The proposition is proved. \end{proof}

When $\ii$ has repetitions, we can  reduce the computation of the relative Littlewood-Richardson coefficients to the repetition-free case by introducing a certain class of relative repetition-free coefficients $p_{K',K''}^{\ii,K}$ which we refer to as {\it generalized Bott-Samelson coefficients}.

\begin{definition}
\label{def:generalized BS}
For any $S=\{s_i,i\in I\}\subset W$
and any sequence  $\ii=(i_1,\ldots,i_m)\in I^m$ let $\widehat W_\ii$ be the free monoid generated by $\widehat S=\{\widehat s_k\}$, $k\in [m]$ and let $\varphi_\ii:\widehat W_\ii\to W$ be the homomorphism of monoids given by $\widehat s_k\mapsto s_{i_k}$ for $k\in [m]$. This makes any $W$-module algebra $Q$ into a $\widehat W_\ii$-module algebra and thus the twisted group algebra Then $Q_{\widehat W_\ii}=Q\rtimes \widehat W_\ii$ is is well-defined. Next, we fix the Demazure $\widehat S$-tame set $\widehat X_\ii=\{\widehat x_k=\frac{1}{\alpha_{i_k}}(\widehat s_k-1),i\in I,k\in [m]\}$.

Then for any subsets $K,K',K''\subset [m]$ we set
$$p_{K',K''}^{\ii,K}:=\widehat p^{\,\bf k}_{{\bf k}',{\bf k}''}$$
where the right hand side is the relative coefficient for the twisted group algebra $Q_{\widehat W_\ii}$ with respect to $\widehat X_\ii$ and ${\bf k}\in [m]^{|K|},{\bf k}'\in [m]^{|K'|},{\bf k}''\in [m]^{|K''|}$ are the sequences naturally obtained from the sets $K,K',K''$ respectively.

\end{definition}

By definition,
\begin{equation}
\label{eq:BS shorthand}
p_{K',K''}^{\ii,K}=p_{\varphi(K'),\varphi(K'')}^{\ii_K,[|K|]}
\end{equation}
for any $K, K',K''$, where $\ii_K$ is as in Definition \ref{def:admissible} and
$\varphi:K\widetilde \to \{1,\ldots,|K|\}$ is the natural order-preserving bijection.

The following is a direct corollary of Theorem \ref{th:folding}.

\begin{corollary}
\label{cor:Bott-Samelson to relative}
For any $\ii\in I^m$ and any subsequences $\ii'$ and $\ii''$ of $\ii$ one has
\begin{equation}
\label{eq:relative LR via BS general}
p_{\ii',\ii''}^\ii=\sum p_{K',K''}^{\ii,[m]}
\end{equation}
where the summation is over all pairs $K',K''\subset [m]$ such that $\ii_{K'}=\ii'$, $\ii_{K''}=\ii''$.
In particular,
if $\ii$ is repetition-free, then:
\begin{equation}
\label{eq:relative LR via BS repetition free}
p_{\ii',\ii''}^\ii=p_{K',K''}^{\ii,[m]}
\end{equation}
where $K'=\{\ii'\}$,  $K''=\{\ii''\}$ (in the notation of \eqref{eq:fi}).
\end{corollary}

Since \eqref{eq:relative LR via BS general} is a copy of  \eqref{eq:relative LR via BS}, this justifies the name {\it generalized Bott-Samelson coefficients} for $p_{K',K''}^{\ii,K}$.
We will make the analogy precise in the following result that generalizes \cite[Proposition 4]{wi04}.

\begin{theorem}
\label{th:general BS}
Let $S=\{s_i,i\in I\}\subset W$ and  $X=\{x_i=\frac{1}{\alpha_i}(s_i-1),i\in I\}\subset Q_W$ be a Demazure $S$-tame set. Then for any sequence  $\ii=(i_1,\ldots,i_m)\in I^m$, and any $W$-invariant subalgebra $R\subset Q$ such that such that $x_{i_k}(R)\subset R$ and $\alpha_{i_k}\in R$ for $k\in [m]$ one has:

(a) There exists a commutative $R$-algebra ${\mathcal BS}_{X,\ii,R}$ with the basis $\{\sigma_K,K\subset [m]\}$ such that
\begin{equation}
\label{eq:BS multiplication}
\sigma_{K'} \sigma_{K''}= \sum_{K\subset [m]} p_{K',K''}^{\ii,K} \sigma_K
\end{equation}
for any subsets $K', K''\subset [m]$,
where the summation over all $K\subset [m]$, such that $K'\cup K''\subset K$ and $|K|\le |K'|+|K''|$.

(b) For each $K\subset [m]$ one has in ${\mathcal BS}_{X,\ii,R}$:
$$\sigma_K=\prod_{k\in K}\sigma_k \ .$$

(c) The algebra ${\mathcal BS}_{X,\ii,R}$ is generated by $\sigma_k:=\sigma_{\{k\}}$, $k\in [m]$ subject to the relations:
\begin{equation}
\label{eq:sigma_k^2}
\sigma_k^2=\alpha_{i_k}\sigma_k+\sum_{\ell<k} x_{i_k}(\alpha_{i_\ell})\sigma_\ell\sigma_k
\end{equation}
for $k\in [m]$.

(d) For each $\ii\in I^m$ the association
\begin{equation}
\label{eq:Aii-->BSii}
\sigma_\jj^{(\ii)} \mapsto \sum_{K\subset [m]:\atop \ii_K=\jj} \sigma_K
\end{equation}
defines an injective homomorphism of $R$-algebras
$\pi_\ii:{\mathcal A}_{X,\ii,R}\hookrightarrow {\mathcal BS}_{X,\ii,R}$.
\end{theorem}

\begin{proof} First note that the recursion \eqref{eq:pijk demazure} guarantees that the algebra $R$ contains all $p^\ii_{\ii',\ii''}$ and all $p^{\ii,K}_{K',K''}$.

Prove (a) now. In the notation of Proposition \ref{pr:A-->Aii} define ${\mathcal BS}_{X,\ii,R}:={\mathcal A}_{\widehat X_\ii,(1,2,\ldots,m),R}$ and abbreviate
$$\sigma_K:=\sigma_\jj^{(1,2,\ldots,m)}\ ,$$
for all $K\subset J$ where $\jj$ is the only subsequence of $(1,2,\ldots,m)$ such that $\{\jj\}=K$.
Here the algebra ${\mathcal A}_{\widehat X_\ii,(1,2,\ldots,m),R}$ is associated to
$Q_{\widehat W_\ii}$ with $\widehat W_\ii$ and $\widehat X_\ii$ as in
Definition \ref{def:generalized BS}. Clearly, \eqref{eq:BS multiplication} holds in ${\mathcal BS}_{X,\ii,R}$. This proves (a).

Prove (b). It suffices to prove that $\sigma_k\sigma_K=\sigma_{\{k\}\cup K}$
whenever $k$ is less than the minimal element of $K$, or, equivalently,
\begin{equation}
\label{eq:K cup k}
p_{\{k\},K}^{\ii,\{k\}\cup K}=1\ .
\end{equation}
Indeed, using \eqref{pr:free one step}, we see that
$$p_{\{k\},K}^{\ii,\{k\}\cup K}=\widehat s_1(p_{(k),(k_1,\ldots,k_\ell)}^{(k,k_1,\ldots,k_\ell)})$$
where $K=\{k_1<k_2<\cdots k_\ell\}$. Taking into the account that $p_{\emptyset,{\bf k}}^{\bf k}=1$ for all sequences ${\bf k}$
(this easily follows by induction from \eqref{eq:pijk demazure}), we finish the proof of \eqref{eq:K cup k}. Part (b) is proved.

Prove (c).
Indeed,
it follows from Definition \ref{def:generalized BS}, \eqref{eq:relative LR via BS repetition free} and \eqref{pr:free one step}
that for any $K=\{k_1<k_2\}\subset [m]$ and $k\in [m]$ one has
$$p^{\ii,\{k_1,k_2\}}_{\{k\},\{k\}}=\delta_{k_1,k}p^{(k,k_2)}_{(k),(k)}+\delta_{k_2,k}p^{(k_1,k)}_{(k),(k)}=
\delta_{k_1,k}\widehat s_k(p^{(k_2)}_{\emptyset,\emptyset})+\delta_{k_2,k}\widehat x_{k_1}(p^{(k)}_{(k),(k)})
=\delta_{k_2,k}x_{i_{k_1}}(\alpha_{i_k})\ .$$
Taking into account that $p^{\ii,\{k\}}_{\{k\},\{k\}}=p^{(i_k)}_{(i_k),(i_k)}=\alpha_{i_k}$, we obtain:
$$\sigma_k^2=\sum_{K\subset [m]\atop k\in K,|K|\le 2} p_{\{k\},\{k\}}^{\ii,K} \sigma_K=\sum_{k'\in [m]} p_{\{k\},\{k\}}^{\ii,\{k,\ell\}}\sigma_{\{k,\ell\}}=\alpha_{i_k}\sigma_k+\sum_{\ell\ne k} p_{\{k\},\{k\}}^{\ii,\{k,\ell\}}\sigma_{\{k,\ell\}}$$
$$=\alpha_{i_k}\sigma_k+\sum_{\ell\ne k}\delta_{\max(k,\ell),k}\sigma_{\{k,\ell\}}=\alpha_{i_k}\sigma_k+\sum_{\ell< k}\sigma_\ell\sigma_k\ .$$
This proves \eqref{eq:sigma_k^2}. It remains to prove that the relations \eqref{eq:sigma_k^2} are defining. This follows from the following result.

\begin{proposition} Let  ${\mathcal A}$ be a commutative $R$-algebra generated by $\sigma_k$, $k\in [m]$ subject to the relations:
\begin{equation}
\label{eq:sigma_k^2 general}
\sigma_k^2=c_k\sigma_k+\sum_{\ell<k} c_{\ell,k} \sigma_\ell\sigma_k
\end{equation}
for all $k\in [m]$, where all $c_k$ and $c_{\ell,k}$ belong to $R$. Then  the set of square-free monomials in $\sigma_1,\ldots,\sigma_m$ is a free $R$-linear basis of ${\mathcal A}$ (hence $dim_R {\mathcal A}=2^m$).
\end{proposition}

\begin{proof} Define the valuation
$\nu:R[\sigma_1,\ldots,\sigma_m]\setminus \{0\}\to
\ZZ_{\ge 0}^m$ by
$$\nu(\sum c_{\bf r}\sigma^{\bf r})=\max \{{\bf r|{\bf r}\in \ZZ_{\ge 0}^m:c_{\bf r}\ne 0}\}$$
where we abbreviated $\sigma^{\bf r}=\sigma_1^{r_1}\cdots \sigma_m^{r_m}$ and the maximum is taken with respect to the the {\it inverse lexicographic} ordering on $\ZZ_{\ge 0}^m$ (i.e., ${\bf r}<{\bf r'}$ if there exists $k\in [m]$ such that $r_k<r'_k$ and $r_\ell=r'_\ell$ for all $\ell>k$).

For each $p=\sum c_{\bf r}\sigma^{\bf r}\in R[\sigma_1,\ldots,\sigma_m]\setminus \{0\}$ define the the leading coefficient $c(p)$ by
$$c(p):=c_{\nu(p)}\ .$$

We say that a subset $B$ of $R[\sigma_1,\ldots,\sigma_m]\setminus \{0\}$ is {\it triangular} if
$$\nu(B)=\ZZ_{\ge 0}^m,c(B)=\{1\}$$

The following fact is obvious.

\begin{lemma} Each triangular subset $B$ of $R[\sigma_1,\ldots,\sigma_m]\setminus \{0\}$ is a free $R$-linear basis of
$R[\sigma_1,\ldots,\sigma_m]$ and the transition matrix between between $B$ and the standard monomial basis $\sigma^{\ZZ_{\ge 0}^m}$ is unitriangular with respect to the inverse lexicographic order.

\end{lemma}

In the notation of \eqref{eq:sigma_k^2 general} define $M_k\in R[\sigma_1,\ldots,\sigma_m]\setminus \{0\}$, $k\in [m]$ by
$$M_k=\sigma_k^2-c_k\sigma_k-\sum_{\ell<k} c_{\ell,k} \sigma_\ell\sigma_k \ .$$
Clearly, the quotient algebra $R[\sigma_1,\ldots,\sigma_m]/J_m$, where $J_m$ is the ideal generated by all $M_k$, $k\in M$
is isomorphic to ${\mathcal A}$.

It is also clear that $\nu(M_k)=2e_k$
for $k=1,\ldots,m$, where $e_1,\ldots,e_m$ is the standard basis of $\ZZ_{\ge 0}^m$.

Denote by ${\mathcal M}$ the $R$-subalgebra of $R[\sigma_1,\ldots,\sigma_m]$ generated by $M_1,\ldots,M_m$.
For each ${\bf t}\in \ZZ_{\ge 0}^m$ define a monomial $M^{\bf t}\in {\mathcal M}$ by:
$$M^{\bf t}:=M_1^{t_1}\cdots M_m^{t_m}\ .$$


Denote by $B_0$ the set of all square-free monomials in $R[\sigma_1,\ldots,\sigma_m]$, i.e., all $\sigma^{\bf r}$ with $\max\limits_{k\in [m]}(r_k)\le 1$.
The following fact is obvious.

\begin{lemma} The set
$$B=\sqcup_{t\in \ZZ_{\ge 0}^m} M^{\bf t}\cdot  B_0$$
is triangular. In particular, ${\mathcal M}\cong R[x_1,\ldots,x_m]$ and $R[\sigma_1,\ldots,\sigma_m]$ is a free ${\mathcal M}$-module with the free ${\mathcal M}$-linear basis $B_0$ of square-free monomials.

\end{lemma}

This implies that $\sqcup_{t\in \ZZ_{\ge 0}^m\setminus \{0\}} M^{\bf t}\cdot  B_0$ is a free $R$-linear basis in the ideal $J_m$. Hence the restriction to $B_0$ of the quotient map $R[\sigma_1,\ldots,\sigma_m]\to R[\sigma_1,\ldots,\sigma_m]/J_m={\mathcal A}$ is injective and the image of $B_0$ is a free $R$-linear basis of ${\mathcal A}$.

%
The proposition is proved. \end{proof}

This finishes the proof of part (c).

Prove (d) now. Indeed, Theorem \ref{th:folding}(b) guarantees that (in the notation of the proof of \ref{pr:A-->Aii}(b)), the homomorphism of monoids $\widehat W_\ii\to W$ given by $\widehat s_k\mapsto s_{i_k}$ extends to  we have a surjective homomorphism of coalgebras $\varphi_\ii:\widehat {\mathcal C}_{(1,2,\ldots,m)}\twoheadrightarrow {\mathcal C}_\ii$. Therefore, dualizing this surjective homomorphism over $R$, we obtain an injective homomorphism  $\pi_\ii:{\mathcal A}_{X,\ii,R}\hookrightarrow {\mathcal BS}_{X,\ii,R}$ given by \eqref{eq:Aii-->BSii}. Part (d) is proved.

Therefore, Theorem \ref{th:general BS} is proved.
\end{proof}

\begin{remark} In fact, the homomorphism \eqref{eq:Aii-->BSii} verifies \eqref{eq:relative LR via BS general}.

\end{remark}

\section{Generalized nil Hecke algebras and proof of Theorem \ref{th:embedding H(G/B) to H(hat G/hat B)}}
\label{sect:nil hecke}
Let $I$ be a finite set of indices.
We say that an $I\times I$ matrix $A=(a_{i,j})$ over $\kk$ is {\it quasi-Cartan} if all $a_{ii}=2$ and $a_{ij}a_{ji}=0$ implies $a_{ij}=a_{ji}=0$. Let $V$ be a $\kk$ vector space with basis  $\{\alpha_i,i\in I\}.$

\begin{definition} We say that a monoid  generated by  $S=\{s_i,i\in I\}$ is a {\it Coxeter semigroup} if $W$ is subject to the relations
$$(s_is_j)^{n_{ij}}=1$$ for all $i,j\in I$,
where $n_{ii}\in \{0,2\}$, $n_{ij}=n_{ji}\in \{0\} \cup \ZZ_{\ge 2}$ for all $i,j$; and: if $n_{ii}= 0$ for some $i$, then  $n_{ij}=0$ for all $j$.

\end{definition}

\begin{remark} In fact, the term ``Coxeter monoid'' has been used by several authors to denote a different object (see e.g., \cite{HiSchThi09, RiSpr90,tsar89,tsar90}) that is never a group. At the
same time, any Coxeter semigroup with all $n_{ii}=2$ is a Coxeter group.

\end{remark}

Note that the free monoid generated by $S$ is a Coxeter semigroup, moreover it is an initial object in the category of Coxeter semigroups generated by $S$.

The following fact is obvious.

\begin{lemma} Let $W$ be a Coxeter semigroup generated by $S=\{s_i,i\in I\}$ and let $I_0=\{i\in I:n_{ii}=0\}$. Then

(a) The sub-monoid of $W$ generated by all $s_i$, $i\in I_0$ is a free monoid $M_0$

(b) The sub-monoid of $W$ generated by all $s_i$, $i\in I\setminus I_0$ is a Coxeter group $W_0$.

(c) One has $W=W_0\star M_0$, where $\star$ stands for the free product of monoids.

\end{lemma}

\begin{definition}
\label{def:admissible for semigroups}
Similarly to Definition \ref{def:admissible}, given a  Coxeter semigroup, we say that a sequence $\ii=(i_1,\ldots,i_m)\in I^m$ is {\it reduced} if the element $w=w_\ii:=s_{i_1}\cdots s_{i_m}\in W$ is shortest possible in $W$ and define its {\it Coxeter length} $\ell(w):=m$.
Given $w\in W$, denote by $R(w)$ the set of all {\it reduced words} of $w$, i.e, all $\ii\in I^{\ell(w)}$ such that $w_\ii=w$.
\end{definition}

\begin{definition}
We say that a Coxeter semigroup $W$ is {\it weakly compatible} with a quasi-Cartan matrix $A$ if for  each $i\ne j$ we have:
$$n_{ij}\ge 2~\text{implies that}~a_{ij}a_{ji}=\zeta_{ij}+\zeta_{ij}^{-1}+2$$
for some  $n_{ij}$-th root of unity $\zeta_{ij}\in \kk^\times$.
\end{definition}

The following result is obvious.

\begin{lemma}
\label{le:Coxeter action}
Let $V$ be a $\kk$ vector space with basis  $\{\alpha_i,i\in I\}$  and let $A=(a_{ij})$ be an $I\times I$ quasi-Cartan matrix weakly  compatible with a Coxeter semigroup $W$ generated by $S=\{s_i,i\in I\}$. Then the association
$$s_i(\alpha_j)=\alpha_i-a_{ij}\alpha_j $$
defines an action of $W$ on $V$.
\end{lemma}

Throughout the section we fix a Coxeter semigroup $W=\langle s_i,i\in I\rangle$ and a weakly compatible quasi-Cartan matrix $A=(a_{ij})$
together with the action $W\times V\to V$ prescribed by Lemma \ref{le:Coxeter action}. Denote by $Q=Frac(S(V))$ the field of fractions of the symmetric algebra $S(V)$.

Clearly, $Q$ is a module algebra over the group (or, rather, monoidal) algebra $\kk W$ so one has a twisted group algebra $Q_W:=Q\rtimes \kk W$ and the coaction
\begin{equation}
\label{eq:delta Coxeter}
\delta:Q_W\to Q_W\bigotimes\limits_Q Q_W
\end{equation}
given by Proposition \ref{pr:delta}.

For any $i\in I$, define a Demazure element $x_i\in Q_W$ by:
$$x_i:=\frac{1}{\alpha_i}(s_i-1)$$
and denote by ${\mathcal H}_A(W)$ the subalgebra of $Q_W$ generated by all $x_i$, $i\in I$. Following Kostant and Kumar (\cite{KK86}), we refer to it  as a  {\it generalized nil Hecke algebra}.

\begin{theorem}
\label{th:braid relations}
Assume that a Coxeter semigroup $W$ and a quasi-Cartan matrix $A$ are weakly compatible and $\sqrt{a_{ij}a_{ji}}\in \kk$ whenever $n_{ij}$ is odd. Then the generalized nil Hecke algebra ${\mathcal H}_A(W)$
is subject to the following relations:
\begin{equation}
\label{eq:relations nil hecke}
x_i^2=0 ~ \text{iff $n_{ii}=2$};~~
\begin{cases}
\underbrace{x_ix_j\cdots x_j}_{n_{ij}}=\underbrace{x_jx_i\cdots x_i}_{n_{ij}} & \text{if $n_{ij}\in 2\ZZ_{>0}$} \\
\underbrace{x_ix_j\cdots x_i}_{n_{ij}}=\sqrt{\frac{a_{ij}}{a_{ji}}}\underbrace{x_jx_i\cdots x_j}_{n_{ij}} & \text{if $n_{ij} \in 1+2\ZZ_{\ge 0}$} \\
\end{cases}
\end{equation}
In particular, the monomials $x_\ii=x_{i_1}\cdots x_{i_m}$ satisfy:
\begin{equation}
\label{eq:nilproduct}
x_\ii=0
\end{equation}
for any non-reduced sequence $\ii$ and
\begin{equation}
\label{eq:equalproducts}
x_\ii=d_{\ii,\ii'} x_{\ii'}
\end{equation}
for any $\ii,\ii'\in R(w)$, $w\in W$, where  $=d_{\ii,\ii'}\in \kk^\times$ is a product of $\sqrt{\frac{a_{ij}}{a_{ji}}}$.
\end{theorem}

\begin{proof} Indeed, if $n_{ii}=2$, i.e., $s_i^2=1$, then the relation $x_i^2=0$ is obvious.

The remaining relations follow from the following rank $2$ computation.

\begin{proposition}
\label{pr:dihedral braid relation}
Assume that $I=\{1,2\}$, $W=\langle s_1,s_2\,|\,s_1^2=s_2^2=(s_1s_2)^n=1\rangle$ is a dihedral group, and
$A=\begin{pmatrix} 2 & a_{12} \\
a_{21} & 2
\end{pmatrix}
$ is  a quasi-Cartan matrix over $\kk$ with $a_{12}a_{21}=\zeta+\zeta^{-1}+2$, where $\zeta\in \kk$ is an $n$-th root of unity and
$\sqrt{a_{12}a_{21}}\in \kk$ if $n$ is odd. Then the generators of the generalized nil Hecke algebra ${\mathcal H}_A(W)$ satisfy:
\begin{equation}
\label{eq:braid relations}
\begin{cases}
\underbrace{x_1x_2\cdots x_2}_n=\underbrace{x_2x_1\cdots x_1}_n & \text{if $n$ is even} \\
\underbrace{x_1x_2\cdots x_1}_n=\sqrt{\frac{a_{21}}{a_{12}}}\underbrace{x_2x_1\cdots x_2}_n & \text{if $n$ is odd} \\
\end{cases} \ .
\end{equation}

\end{proposition}

\begin{proof}  We will follow the proof of  \cite[Proposition 4.2]{KK86}. Let $V=\kk \alpha_1\oplus \kk \alpha_2$ and $\langle \cdot, \alpha_j^\vee\rangle$, $j=1,2$,
be  a linear function $V\to \kk$ given by $\langle \alpha_i,\alpha_j^\vee\rangle=a_{ij}$. Without loss of generality, by rescaling $\alpha_i$ and $\alpha_i^\vee$, $i=1,2$,
we assume that $a_{12}=a_{21}$ if $n$ is odd.

The following result is obvious.

\begin{lemma}
\label{le:dihedral}
In the assumptions of Proposition \ref{pr:dihedral braid relation} and that $a_{12}=a_{21}$ if $n$ is odd, we have:

(a) for any $\alpha\in V$:
$$\alpha \underbrace{x_ix_j\cdots x_{i'}}_m= \underbrace{x_ix_j\cdots x_{i'}}_m \cdot w_m^{-1}(\alpha)-\langle \alpha,\alpha_i^\vee\rangle \cdot \underbrace{x_jx_i\cdots }_{m-1} - \langle w_{m-1}^{-1}(\alpha),\alpha_{i'}^\vee \rangle  \cdot \underbrace{x_ix_j\cdots }_{m-1} $$
for $m\in \ZZ_{>0}$ and $\{i,j\}=\{1,2\}$, where   $i'=i$ if $m$ is odd and $i'=j$ if $m$ is even, and
$$w_m=w_m^{(i)}=\underbrace{s_is_j\cdots s_{i'}}_m\ .$$

(b) $\underbrace{x_ix_j\cdots x_{i'}}_m\in c_m^{(i)} w_m^{(i)}+\sum\limits_{w:\ell(w)<\ell(w_m^{(i)})} Q\cdot w$ for $\{i,j\}=\{1,2\}$, where
$$c_k^{(i)}=\frac{1}{\alpha_i\cdot s_i(\alpha_j)\cdots w_{k-1}^{(i)}(\alpha_{i'})} \ .$$
(c) The action of the longest element $w_\circ:=w_n^{(1)}=w_n^{(2)}$ on $V$ and $V^*$ is given by:
$$w_\circ(\alpha_i)=
\begin{cases} -\alpha_i & \text{if $n$ is even}\\
-\alpha_j & \text{if $n$ is odd}\\
\end{cases}, ~w_\circ(\alpha_i^\vee)=
\begin{cases} -\alpha_i^\vee & \text{if $n$ is even}\\
-\alpha_j^\vee & \text{if $n$ is odd}\\
\end{cases}$$
for $\{i,j\}=\{1,2\}$, where $\ell(w)$ is the Coxeter length of $w$ (see Definition \ref{def:admissible for semigroups}).

\end{lemma}

This implies that
$$\alpha \underbrace{x_1x_2\cdots x_2}_n- \underbrace{x_1x_2\cdots x_2}_n \cdot  w_\circ^{-1}(\alpha)=
\alpha \underbrace{x_2x_1\cdots x_1}_n- \underbrace{x_2x_1\cdots x_1}_n\cdot w_n^{-1}(\alpha)$$
for all $\alpha\in V$. Equivalently
\begin{equation}
\label{eq:w_0 commutation}
\alpha\cdot {\bf D} ={\bf D}\cdot  w_\circ^{-1}(\alpha)
\end{equation}
for all $\alpha\in V$, where ${\bf D}:=\underbrace{x_1x_2\cdots }_n- \underbrace{x_2x_1\cdots }_n$.
Let us prove that $\Delta=0$. First, parts (b) and (c) of Lemma \ref{le:dihedral} imply that $c_n^{(1)}=c_n^{(2)}.$  Therefore,
$${\bf D}=\sum_{w:\ell(w)<\ell(w_\circ)} c_w w$$
for some $c_w\in Q=\kk(\alpha_1,\alpha_2)$. Hence, \eqref{eq:w_0 commutation} becomes:
$$0=\sum_{w:\ell(w)<\ell(w_\circ)} c_w \cdot (\alpha \cdot w-w\cdot w_\circ^{-1}(\alpha))=\sum_{w:\ell(w)<\ell(w_\circ)} c_w \cdot (\alpha-w w_\circ^{-1}(\alpha)) \cdot w$$
for all $\alpha\in V$. This implies that all $c_w=0$ hence ${\bf D}=0$.

The proposition is proved.
\end{proof}

This proves the  relations \eqref{eq:relations nil hecke} in ${\mathcal H}_A(W)$ which immediately imply \eqref{eq:nilproduct} and
\eqref{eq:equalproducts}. Let us verify that the relations \eqref{eq:relations nil hecke} are defining. For each $w\in W$ let us choose a representative $\ii_w\in R(w)$. Then
$$\sum_{w\in W} \kk\cdot x_{\ii_w}={\mathcal H}_A(W)$$
by \eqref{eq:nilproduct} and \eqref{eq:equalproducts}. Note that  $Q\cdot {\mathcal H}_A(W)=Q_W$ since $X=\{x_i,i\in I\}$ is tame, and therefore
\begin{equation}
\label{eq:Q_W as direct sum}
\sum_{w\in W} Q\cdot x_{\ii_w}=Q_W\ .
\end{equation}
It suffices to prove that this sum is direct, i.e., $\{x_{\ii_w}\,|\,w\in W\}$ is a $Q$-basis of $Q_W$.
Indeed, it is easy to see that $Q_W$ is filtered by $Q$-submodules $(Q_W)_{\le m}=\bigoplus\limits_{w:\ell(w)\le m} Q\cdot w$ and that for any $\ii=(i_1,\ldots,i_m)\in R(w)$ one has:
$$x_\ii\equiv \frac{1}{\alpha_{i_1}}s_{i_1}\cdots \frac{1}{\alpha_{i_m}}s_{i_m} \mod (Q_W)_{\le m-1}$$
hence $x_{\ii_w}\in Q^\times w + (Q_W)_{\le \ell(w)-1}$
for all $w\in W$.
Therefore, $\{x_{\ii_w}\,|\,w\in W\}$ is a $Q$-basis of $Q_W$ and Theorem \ref{th:braid relations} is proved.
\end{proof}

\begin{remark} In fact, we can explicitly compute the expansion of elements $\underbrace{x_ix_j\cdots x_{i'}}_k$ in Lemma \ref{le:dihedral}(b)  by generalizing the recursion \cite[Equation (8)]{Ku96}.
Namely, in the notation of Lemma \ref{le:dihedral}, assume that $a_{12}=a_{21}=t+t^{-1}$.
Then the coefficients $d_{w}^{(i)}=d_{w}^{(i;m)}\in \ZZ[a_{12},\alpha_1,\alpha_2]$ of the expansion:
$$\underbrace{x_ix_j\cdots }_m=c_m^{(i)}w_m^{(i)}+\sum_{w\in W:\ell(w)<m} d_w^{(i)} w$$
are given by $d_{{\underbrace{s_js_i \cdots }_{k}}}^{(i)}=-d_{\underbrace{s_is_j \cdots }_{k+1}}^{(i)}$ for $0\le k<m$  and:
$$d_{\underbrace{s_is_j \cdots}_{m-2k}}^{(i)}=\left[\begin{array}{c}m-1\\ k \end{array}\right]_t
c^{(j)}_k\cdot c^{(i)}_{m-k},~d_{\underbrace{s_is_j \cdots}_{m-2k-1}}^{(i)}=-\left[\begin{array}{c}m-1\\ k \end{array}\right]_t
c^{(i)}_k\cdot c^{(j)}_{m-k-1}$$
for $0\le k<\frac{m}{2}$, $\{i,j\}=\{1,2\}$, where
$\left[\begin{array}{c}m-1\\ k  \end{array}\right]_t\in \ZZ[a_{12}]$ are binomial polynomials in $t$ (as in Remark \ref{rem:binomial polynomials}).

\end{remark}

\begin{definition} We say that a Coxeter semigroup $W$ and a weakly compatible quasi-Cartan matrix $A$ are {\it compatible} if $n_{ij}\in 2\ZZ+1$ implies that $a_{ij}=a_{ji}$.

\end{definition}

Clearly, for any $I\times I$ quasi-Cartan matrix $A$, the {\it free} Coxeter group $\widetilde W=\langle s_i,i\in I: s_i^2=1\rangle$ and the free monoid on $S=\{s_i,i\in I\}$ are both  compatible with $A$.

Assume now that $A$ and $W$ are compatible, then $d_{\ii,\ii'}=1$ in \eqref{eq:equalproducts} and for each $w\in W$ there exists an element $x_w\in {\mathcal H}_A(W)$ such that
$$x_w=x_\ii$$
for all $\ii\in R(w)$.
Clearly, the collection  $\{x_w\,|\,w\in W\}$ is determined by  $x_{s_i}=x_i$ for $i\in I$ and
$$x_ux_v=\begin{cases}
x_{uv} & \text{if $\ell(uv)=\ell(u)+\ell(v)$}\\
0 & \text{if $\ell(uv)<\ell(u)+\ell(v)$}\\
\end{cases}.
$$

The following is an immediate corollary from the proof of Theorem \ref{th:braid relations}.

\begin{corollary} Assume that a Coxeter semigroup $W$ and a quasi-Cartan matrix $A$ are compatible.
Then the collection  $\{x_w\,|\,w\in W\}$ is a basis of the generalized nil Hecke algebra ${\mathcal H}_A(W)$.
\end{corollary}

In particular, $B=\{x_w|w\in W\}$ is a left $Q$-basis of $Q_{A,W}$ (in the notation of Section \ref{sect:prelim}).
This defines the Littlewood-Richardson coefficients $p_{u,v}^w=p_{u,v}^w(A)\in Q$ for $u,v,w\in W$ by \eqref{eq:delta Coxeter}
and the formula \eqref{eq:generalized LR}. Similarly, for each admissible (in the sense of Definition \ref{def:admissible})
sequence $\ii\in I^m$ and its subsequences $\ii'$ and $\ii''$ one defines the corresponding relative Littlewood-Richardson coefficient
$p_{\ii',\ii''}^\ii=p_{\ii',\ii''}^\ii(A)$. If $W$ is the free monoid (resp. the free Coxeter group) on $S$, then  the assignment
\begin{equation}
\label{eq:bijection admissible}
\ii=(i_1,\ldots,i_m)\mapsto \widehat w_\ii=s_{i_1}\cdots s_{i_m}
\end{equation}
is a bijection between the set of all (resp. all admissible) sequences and $W$, e.g., $p_{\ii',\ii''}^\ii=p_{\widehat w_{\ii'},\widehat w_{\ii''}}^{\widehat w_\ii}$.

\begin{proposition}
\label{pr:relative versus absolute}
For each quasi-Cartan matrix $A$ and a weakly compatible Coxeter semigroup $W$ one has:

(a) Each $p_{\ii',\ii''}^\ii$ belongs to $S(V)$ and is homogeneous of degree $m'+m''-m$, where $m$, $m'$, and $m''$ are respectively the lengths of $\ii$, $\ii'$, and $\ii''$.

(b) Assume that $A$ and $W$ are compatible. Then for each triple $u,v,w\in W$ with $\ell(u)+\ell(v)\ge \ell(w)$ and for each $\ii\in R(w)$ one has:
\begin{equation}
\label{eq:puvw=piii}
p_{u,v}^w=\sum_{\ii'\in R(u),\ii''\in R(v)} p_{\ii',\ii''}^\ii \ .
\end{equation}
\end{proposition}

\begin{proof} Part (a) directly follows from the recursion \eqref{eq:recursion LR} and the following obvious fact.

\begin{lemma} Under the action of $x_i=\frac{1}{\alpha_i}(s_i-1)$ on $Q=Frac(S(V))$ one has:
$$x_i(\alpha_j)=-a_{ij}, x_i(fg)=x_i(f)g+s_i(f)x_i(g)$$
for all $i,j\in I$, $f,g\in Q$. In particular,
$$x_i(S^k(V))\subset S^{k-1}(V)$$ for each $k\ge 0$.
\end{lemma}

Prove (b). Indeed, in the notation of Theorem \ref{th:folding}(a), let $\widehat W$ be the free Coxeter group generated by $\widehat s_i$, $i\in I$
with the structural surjective  homomorphism $\varphi_-:\widehat W\to W$ given by $\widehat s_i \mapsto s_i$, which,
together with the identity map $Q\to Q$ extends to an algebra homomorphism $\varphi:Q_{\widehat W}\to Q_W$ such that $\varphi(\widehat x_{\widehat w_\ii})=x_w$ for all $\ii\in R(w)$ under the bijection \eqref{eq:bijection admissible}, i.e., $\widehat W_w=R(w)$. Finally, taking into account that   $p_{\widehat w_{\ii'},\widehat w_{\ii''}}^{\widehat w_\ii}=p_{\ii',\ii''}^\ii$, the identity \eqref{eq:LR folding} becomes \eqref{eq:puvw=piii}. This proves (b).

%
The proposition is proved. \end{proof}

As a corollary, we obtain a generalization of \cite[Proposition 4.15]{KK86}.

\begin{corollary} Each each $p_{u,v}^w$ belongs to $S(V)$ and is homogeneous of degree $\ell(u)+\ell(v)-\ell(w)$ (e.g., $p_{u,v}^w=0$ if $\ell(u)+\ell(v)<\ell(w)$).

\end{corollary}

Dualizing the above arguments, we obtain the following result.

\begin{proposition}
\label{pr:dual nil Hecke}
For each quasi-Cartan matrix $A$ and any compatible Coxeter semigroup $W$ we have:

(a) There is a unique commutative $S(V)$-algebra ${\mathcal A}_A(W)$ with the free $S(V)$-basis $\{\sigma_w\,|\,w\in W\}$ and the following multiplication table:
$$\sigma_u\sigma_v=\sum_{w\in W} p_{u,v}^w \sigma_w$$
for all $u,v\in W$.

(b) There is a unique commutative $S(V)$-algebra $\widehat {\mathcal A}_A$ with the free $S(V)$-basis $\{\sigma_\ii\}$
labeled by all sequences in $I^m$, $m\ge 0$ with following multiplication table:
\begin{equation}
\label{eq:product free monoid}
\sigma_{\ii'} \sigma_{\ii''}= \sum_\ii p_{\ii',\ii''}^\ii \sigma_\ii
\end{equation}
where the summation over all sequences $\ii\in I^m$, $m\ge 0$ containing $\ii'$ and $\ii''$ as sub-sequences and such that $m\le m'+m''$.

(c) One has an injective algebra homomorphism ${\mathcal A}_A(W)\hookrightarrow \widehat {\mathcal A}_A$ via:
$$\sigma_w\mapsto \sum_{\ii\in R(w)} \sigma_\ii \ .$$

\end{proposition}

The following is a slight modification ($Q$ is replaced with $S(V)$) of Corollary \ref{cor:Billey homomorphism}.

\begin{corollary}
\label{cor:Billey homomorphism coxeter}
Given a Coxeter semigroup $W$ and a compatible  quasi-Cartan matrix $A$, let
$\langle \cdot,\cdot \rangle:{\mathcal A}_A(W)\times S(V)\cdot {\mathcal H}_A(W)\to S(V)$ be the non-degenerate $S(V)$-bilinear pairing given by
$$
\langle p\sigma_u,q x_v\rangle =\delta_{u,v}\cdot  pq
$$
for all $u,v\in W$, $p,q\in S(V)$. Then:

(a) The above pairing  satisfies:
$$\langle ab,x\rangle=\langle a\otimes b,\delta(x)\rangle= \langle a,x_{(1)}\rangle \langle b,x_{(2)}\rangle$$
for all $a,b\in {\mathcal A}_A(W)$, $x\in S(V)\cdot {\mathcal H}_W$, where $\delta(x)=x_{(1)}\otimes x_{(2)}$ in the Sweedler notation.

(b) For each $w\in W$ the assignment $a\mapsto \langle a,w\rangle$, $a\in {\mathcal A}_A(W)$ is an $S(V)$-algebra homomorphism
$$\varphi_w:{\mathcal A}_A(W)\to S(V) \ .$$

\end{corollary}


\begin{remark}
\label{rem:billey}
If $A$ is a Cartan matrix and $W$ is the Weyl group of a Kac-Moody group $G$, Kostant and Kumar proved that $\varphi_w$ is a homomorphism of algebras (see \cite[Section 11.1.4]{Ku02}). Sara Billey computed $\varphi_w(\sigma_v)$ explicitly in  \cite[Theorem 4]{Bi99}.
\end{remark}

The algebras ${\mathcal A}_A(W)$ are very important in Schubert Calculus due to the following fundamental result.

\begin{theorem} (\cite[Corollary 11.3.17]{Ku02})
\label{the:cohomology of flags}
Let $G$ be a complex semisimple or Kac-Moody group, $T\subset B$ be respectively the Cartan and Borel subgroups of $G$,  $W=Norm_G(T)/T$ be the Weyl group, and let $A$ be the Cartan matrix of $G$. Then the assignment
$$\sigma_w^T\mapsto \sigma_w$$
defines an isomorphism of $S(V)$-algebras $H_T^*(G/B)\widetilde \to {\mathcal A}_A(W)$,
where $H_T^*(G/B)$ is the $T$-equivariant cohomology algebra (over $S(V)=H_T^*(pt)$) of $G/B$ and
$\sigma_w^T$, $w\in W$ is the $T$-equivariant Schubert cocycle corresponding to $w\in W$. In particular, the cup  product in $H_T^*(G/B)$ is given by:
$$\sigma_u^T\cup \sigma_v^T=\sum_{w\in W} p_{u,v}^w \sigma_w^T\ .$$

Therefore, the cup product in the cohomology algebra $H^*(G/B,\CC)=\CC\bigotimes\limits_{S(V)} H_T^*(G/B)$ (where $\CC$ is  an $S(V)$-module via the projection $S(V)\to S^0(V)=\CC$) is given by:
$$\sigma_u\cup \sigma_v=\sum_{w\in W:\atop \ell(w)=\ell(u)+\ell(v)} p_{u,v}^w \sigma_w\ .$$
\end{theorem}

Note that the Littlewood-Richardson coefficients in \eqref{eq:cup GB} are given by
$$c_{u,v}^w=\delta_{\ell(w),\ell(u)+\ell(v)}p_{u,v}^w$$ for all $u,v,w$.
In order to compute $p_{u,v}^w$  we  employ the relative Littlewood-Richardson coefficients
$p_{\ii',\ii''}^\ii$ for $\ii\in R(w)$, $\ii'\in R(u)$, $\ii''\in R(v)$ and use \eqref{eq:puvw=piii}.
%
%
%
%
That is, in view of Proposition \ref{pr:relative versus absolute}, our Theorem \ref{th:cuvw} follows from Theorem \ref{th:cuvw refined T} that we prove in the next section.


\medskip

\noindent {\bf Proof of Theorem \ref{th:embedding H(G/B) to H(hat G/hat B)}.}
Indeed, let $\widehat {\mathcal A}_A$ be as in Proposition \ref{pr:dual nil Hecke}(b). That is, ${\mathcal A}_A$ is dual (over $S(V)$) to the generalized nil Hecke algebra ${\mathcal H}_A(\widehat W)$, where $\widehat W$ is the free monoid generated by $S=\{s_i,i\in I\}$.  Since $\widehat {\mathcal A}_A={\mathcal A}(G)$ whenever $A$ is the Cartan matrix of $G$, this proves Theorem \ref{th:embedding H(G/B) to H(hat G/hat B)}(a).

Furthermore, let ${\mathcal A}_A^{adm}$ be the dual (over $S(V)$) to the generalized nil Hecke algebra ${\mathcal H}_A(\widetilde W)$, where $\widetilde W$ is the the free Coxeter group generated by $S=\{s_i,i\in I\}$. Clearly, the canonical projection $\widehat W\twoheadrightarrow \widetilde W$ extends to a surjective homomorphism of algebras
$${\mathcal H}_A(\widehat W)\twoheadrightarrow {\mathcal H}_A(\widetilde W)$$
commuting with the co-product. Dualizing, we see that ${\mathcal A}_A^{adm}$ is a subalgebra of $\widehat {\mathcal A}_A$ spanned (over $S(V)$) by all $\sigma_\ii$, where $\ii$ runs over all admissible sequences.  This proves Theorem \ref{th:embedding H(G/B) to H(hat G/hat B)}(b).

Furthermore, since ${\mathcal A}_A(W)=H^*(G/B,\CC)$ whenever $A$ is the Cartan matrix and $W$ is the Weyl group of $G$, Proposition \ref{pr:dual nil Hecke}(c) proves Theorem \ref{th:embedding H(G/B) to H(hat G/hat B)}(c).

Prove parts (d) and (e) of Theorem \ref{th:embedding H(G/B) to H(hat G/hat B)} now.
For each $\ii\in I^m$ denote by  $\widehat {\mathcal A}_{\ii,A}$ the algebra ${\mathcal A}_{X,\ii,S(V)}$ from Proposition \ref{pr:A-->Aii}(b), which, by definition is the dual of the subcoalgebra ${\mathcal C}_\ii\subset {\mathcal H}_A(\widehat W)$ spanned by all $x_\jj$, where $\jj$ runs over all subsequences of $\ii$. Thus, $\widehat {\mathcal A}_{A,\ii}$ is quotient algebra $\widehat {\mathcal A}_A/J_\ii$, where $J_\ii$ is the ideal spanned (over $S(V)$) by all $\sigma_{\ii'}$ such that $\ii'$ is not a subsequence of $\ii$.

Furthermore, in the notation of Theorem \ref{th:general BS}, for each sequence $\ii\in I^m$,  denote by ${\mathcal BS}_{A,\ii}$ the Bott-Samelson algebra ${\mathcal BS}_{X,\ii,S(V)}$.

That is, taking into account that $x_i(\alpha_j)=-a_{ij}$,  ${\mathcal BS}_{A,\ii}$ is an $S(V)$-algebra generated by $\sigma_1,\ldots,\sigma_m$ subject to the relations:
\begin{equation}
\label{eq:sigma_k^2 quasi-cartan}
\sigma_k^2=\alpha_{i_k}\sigma_k-\sum_{\ell>k} a_{i_k,i_{\ell}}\sigma_\ell\sigma_k
\end{equation}
for $k\in [m]$.

The following is a generalization of Proposition \ref{pr:willems} to any Coxeter semigroup $W=\langle s_i,i\in I\rangle$ and compatible quasi-Cartan matrix $A$.

\begin{corollary} (of Theorem \ref{th:general BS})
\label{cor:general BS quasi-cartan} Let  $A$ be a  quasi-Cartan matrix.
Then for any sequence  $\ii=(i_1,\ldots,i_m)\in I^m$ the association
\begin{equation}
\label{eq:Aii-->BSii quasi-cartan}
\sigma_\jj^{(\ii)} \mapsto \sum_{K\subset [m]:\atop \ii_K=\jj} \sigma_K
\end{equation}
defines an injective homomorphism of $R$-algebras
$\pi_\ii:\widehat {\mathcal A}_{A,\ii}\hookrightarrow {\mathcal BS}_{A,\ii}$.
\end{corollary}

Furthermore, if $A$ is the Cartan matrix of $G$, then \cite[Proposition 4]{wi04} implies that
$${\mathcal BS}_{A,\ii}=H^*_T(\Gamma_\ii(G),\CC)\ .$$
Applying Corollary \ref{cor:general BS quasi-cartan}, we  finish the proof of parts (d) and (e) of Theorem \ref{th:embedding H(G/B) to H(hat G/hat B)}.

Theorem \ref{th:embedding H(G/B) to H(hat G/hat B)} is proved. \endproof

\section{Proof of Theorems \ref{th:cuvw refined T BS}, \ref{th:cuvw} and  \ref{th:cuvw refined T}}

\label{sect:proofs}

\subsection{Proof of Theorem  \ref{th:cuvw refined T BS}}

\begin{definition}
\label{def:bounded map} Let $L,M\subset [m]$ such that $|L|+|M|\ge m$. We say that a map
$$\varphi:L\to \{0\}\cup [m]\setminus M$$
is {\it bounded} if:

$\bullet$ $\varphi(\ell)<\ell$ for each $\ell\in L$;

$\bullet$ $|\varphi^{-1}(k)|=1$ for all $k\in [m]\setminus M$ (i.e., the restriction of $\varphi$ to $L'=\varphi^{-1}([m]\setminus M)$ is a bijection $L'\widetilde \to [m]\setminus M$).

\end{definition}

Denote by $V^\vee$ the $\kk$-vector space with the basis $\{\alpha_i^\vee,i\in I\}$ and define the pairing $V\times V^\vee\to \kk$ by $\langle \alpha_i,\alpha_j^\vee\rangle =a_{ij}$ for $i,j\in I$.
For each bounded map $\varphi:L\, \widetilde \to\, \{0\}\cup [m]\setminus M$ define $p_\ell^{(\varphi)}\in V\sqcup \kk$ by
\begin{equation}
\label{def:plambda phi uniform}
p_\ell^{(\varphi)}=
\begin{cases}
\langle w_\ell^{(\varphi)}(\alpha_{i_\ell}),-\alpha_{i_{\varphi(\ell)}}^\vee\rangle &\text{if $\varphi(\ell)\ne 0$}\\
w_\ell^{(\varphi)}(\alpha_{i_\ell}) &\text{if $\varphi(\ell)= 0$}
\end{cases}, ~\text{where}~w_\ell^{(\varphi)}= \buildrel\longrightarrow \over {\prod_{r\in M\bigcup\varphi(L_{<\ell}): \atop \varphi(\ell)<r<\ell}} s_{i_r} \end{equation}

Clearly, there is a natural one-to-one correspondence between bounded maps $\varphi:L \to \{0\}\cup [m]\setminus M$ and bounded bijections $\varphi':L'\, \widetilde \to\, [m]\setminus M$, where $L'$ runs over all subsets of $L$ such that $|L'|+|M|=m$. Therefore,   Theorem \ref{th:cuvw refined T} is equivalent to the following result.

\begin{proposition}
\label{pr:free_coeff}
For any repetition-free sequence $\ii=(i_1,\ldots,i_m)$ and any subsequences $\ii'$, $\ii''$ of $\ii$ Theorem \ref{th:cuvw refined T uniform} holds. More precisely,
\begin{equation}
\label{eq:free_coeff}
p_{\ii',\ii''}^\ii=\sum_\varphi \prod_{\ell\in L} p_\ell^{(\varphi)}
\end{equation}
with the summation over all bounded
maps $\varphi:L\rightarrow \{0\}\cup [m]\setminus M$, where $L\subset M\subset [m]$ are determined by $\{\ii'\}\cap \{\ii''\}=\{\ii_L\}$, $\{\ii'\}\cup \{\ii''\}=\{\ii_M\}$ (in the notation of Proposition \ref{pr:free}).

\end{proposition}

\begin{proof}

We prove Proposition \ref{pr:free_coeff} by induction in the length $m$ of $\ii$. If $m=0$, i.e., $\ii=\ii'=\ii''=\emptyset$, $p_{\ii',\ii''}^\ii=1$ and we have nothing to prove.
Assume that $m\ge 1$. We apply the inductive hypothesis to $\widetilde \ii=(i_2,\ldots,i_m)$ and the subsequences $\widetilde \ii'=\ii'\setminus \{i_1\}$, $\widetilde \ii''=\ii''\setminus \{i_1\}$ of $\widetilde \ii$:

\begin{equation}
\label{eq:free_coeff induction}
p_{\widetilde \ii',\widetilde \ii''}^{\widetilde \ii}=\sum_{\widetilde\varphi} \prod_{\ell\in \widetilde L} \widetilde p_\ell^{(\widetilde \varphi)}
\end{equation}
with the summation over all bounded
maps $\widetilde \varphi:\widetilde L\rightarrow \{0\}\cup \{2,\ldots,m\}\setminus \widetilde M$, where $\widetilde L\subset \widetilde M\subset \{2,\ldots,m\}$ are determined by $\{\widetilde \ii'\}\cap \{\widetilde \ii''\}=\{\ii_{\widetilde L}\}$, $\{\widetilde \ii'\}\cup \{\widetilde \ii''\}=\{\ii_{\widetilde M}\}$ (in the notation of Proposition \ref{pr:free}), and
\begin{equation}
\label{def:plambda phi uniform inductive step}
\widetilde p_\ell^{(\widetilde \varphi)}=
\begin{cases}
\langle \widetilde w_\ell^{(\widetilde \varphi)}(\alpha_{i_\ell}),-\alpha_{i_{\widetilde \varphi(\ell)}}^\vee\rangle &\text{if $\widetilde \varphi(\ell)\ne 0$}\\
\widetilde w_\ell^{(\widetilde \varphi)}(\alpha_{i_\ell}) &\text{if $\varphi(\ell)= 0$}
\end{cases}, ~\text{where}~\widetilde w_\ell^{(\widetilde \varphi)}=
\buildrel\longrightarrow \over {\prod_{r\in M\bigcup\widetilde \varphi(L_{<\ell}): \atop \widetilde \varphi(\ell)<r<\ell}} s_{i_r} \end{equation}

Since $i_1$ is repetition-free, applying \eqref{eq:free composition}  to \eqref{eq:free_coeff induction} (with $k=1$), we obtain:
\begin{equation}
\label{eq:free_coeff induction step}
p_{\ii',\ii''}^\ii=\sum_{\widetilde \varphi} f_{i_1}(\prod_{\ell\in \widetilde L} \widetilde p_\ell^{(\widetilde \varphi)})
\end{equation}

Consider three cases.

Case I. $i'_1=i''_1=i_1$ so that  $f_{i_1}=\alpha_{i_1}s_{i_1}$ and $L=\{1\}\cup \widetilde L$, $M=\{1\}\cup \widetilde M$. Clearly, each $\tilde \varphi:L\setminus\{1\}\rightarrow \{0\}\cup [m]\setminus M$ as in \eqref{eq:free_coeff induction} can be uniquely extended to a bounded map $L\rightarrow \{0\}\cup [m]\setminus M$ by: $\varphi(1)=0$. Thus, $p_1^{(\varphi)}=\alpha_{i_1}$,
$p_\ell^{(\varphi)}=s_{i_1}(\widetilde p_\ell^{(\widetilde \varphi)})$ for all $\ell\in \widetilde L$ and, therefore, \eqref{eq:free_coeff induction step} becomes \eqref{eq:free_coeff}. This proves \eqref{eq:free_coeff} in Case I.

Case II. $i'_1\ne i''_1$, $i_1\in \{i'_1,i''_1\}$ so that  $f_{i_1}=s_{i_1}$ and $L=\widetilde L$, $M=\{1\}\cup \widetilde M$. Therefore, each $\tilde \varphi$  as in \eqref{eq:free_coeff induction} is a bounded map $L\rightarrow \{0\}\cup [m]\setminus M$, i.e, $\widetilde \varphi=\varphi$. Thus, $p_\ell^{(\varphi)}=s_{i_1}(\widetilde p_\ell^{(\varphi)})$ for all $\ell\in L$ and, therefore, \eqref{eq:free_coeff induction step} becomes \eqref{eq:free_coeff}. This proves \eqref{eq:free_coeff} in Case II.

Case III. $i_1\notin \{i'_1,i''_1\}$ so that  $f_{i_1}=s_{i_1}$ and $L=\widetilde L$, $M=\widetilde M$. Applying repeatedly the twisted Leibniz rule:
$$x_i(p_1\cdots p_n)=x_i(p_1)s_i(p_2)\cdots s_i(p_n)+p_1x_i(p_2)s_i(p_3)\cdots s_i(p_n)+\cdots p_1\cdots p_{n-1}x_i(p_n)$$
for $p_1,\ldots,p_n\in S(V)$ and
$$x_i(\alpha)=\langle \alpha,-\alpha_i\rangle$$
for $\alpha\in V$, we obtain for each $\widetilde \varphi:L\rightarrow \{0\}\cup \{2,\ldots,m\}\setminus M$  as in \eqref{eq:free_coeff induction}
$$x_{i_1}\left(\prod_{\ell\in \widetilde L} \widetilde p_\ell^{(\widetilde \varphi)}\right)=\sum_{k\in L:\atop \widetilde \varphi(k)=0} \prod_{\ell\in L} p_\ell^{(\widetilde \varphi,k)} \ ,$$
where for each $k\in \widetilde \varphi^{-1}(0)$:
$$p_\ell^{(\widetilde \varphi,k)}=
\begin{cases} s_{i_1}(\widetilde p_\ell^{(\widetilde \varphi)}) & \text{if $k<\ell$}\\
\langle \widetilde w_\ell^{(\widetilde \varphi)}(\alpha_{i_k}),-\alpha_{i_1}^\vee\rangle & \text{if $k=\ell$}\\
\widetilde p_\ell^{(\widetilde \varphi)} & \text{if $k>\ell$}\\
\end{cases}~~= ~~p_\ell^{(\varphi)}\ ,$$
where $\varphi:L\to  \{0\}\cup [m]\setminus M$ is a unique bounded map such that  $\varphi|_{L\setminus\{k\}}=\widetilde \varphi|_{L\setminus\{k\}}$ and $\varphi(k)=1$. By varying $\widetilde \varphi$, we obtain all bounded maps $L\to  \{0\}\cup [m]\setminus M$, i.e.,  \eqref{eq:free_coeff induction step} becomes \eqref{eq:free_coeff}. This proves \eqref{eq:free_coeff} in Case III.

The proposition is proved.
\end{proof}

In view of \eqref{eq:BS shorthand} and \eqref{eq:relative LR via BS repetition free}, the assertion of Proposition \ref{pr:free_coeff} implies  Theorem \ref{th:cuvw refined T BS}. Therefore, Theorem \ref{th:cuvw refined T BS} is proved. \endproof

\subsection{Proof of Theorems  \ref{th:cuvw} and  \ref{th:cuvw refined T}}
Since Theorem \ref{th:cuvw} directly follows from Theorem \ref{the:cohomology of flags}, Proposition \ref{pr:relative versus absolute}, and Theorem \ref{th:cuvw refined T}, we will only prove the latter result in the following form.

\begin{theorem}
\label{th:cuvw refined T uniform}
For each triple of  admissible sequences $(\ii,\ii',\ii'')$ such that $\ii'$ and $\ii''$ are sub-sequences of $\ii$ and the sum of lengths of $\ii'$ and $\ii''$ is greater or equal the length of $\ii$ one has:
\begin{equation}
\label{eq:cuvw refined T uniform}
p_{\ii',\ii''}^\ii=\sum \prod_{\ell\in K'\cap K''} p_\ell^{(\varphi)}
\end{equation}
with the summation over all triples  $(K',K'',\varphi)$, where
\begin{itemize}
\item $K',K''\subset [m]$ such that $\ii_{K'}=\ii'$, $\ii_{K''}=\ii''$;
\item $\varphi:K'\cap K''\ \to \{0\}\cup [m]\setminus (K'\cup K'')$ is an $\ii$-admissible bounded map.
\end{itemize}
\end{theorem}

The proof of the theorem will occupy the remainder of the section.

We now consider the general case where $\ii$ is not assumed to be repetition free.

We need the following result.

\begin{proposition}
\label{pr:cuvw refined T}
For each triple of  admissible sequences $(\ii,\ii',\ii'')$ such that $\ii'$ and $\ii''$ are sub-sequences of $\ii$ and the sum of lengths of $\ii'$ and $\ii''$ is greater or equal the length of $\ii$ Theorem \ref{th:cuvw refined T uniform} holds if one drops the  ``$\ii$-admissible" condition. More precisely, one has:
\begin{equation}
\label{eq:cuvw refined T proof}
p_{\ii',\ii''}^\ii=\sum \prod_{\ell\in K'\cap K''} p_\ell^{(\varphi)}
\end{equation}
with the summation over all triples $(K',K'',\varphi)$, where
\begin{itemize}
\item $K',K''\subset [m]$ such that $\ii_{K'}=\ii'$, $\ii_{K''}=\ii''$;
\item $\varphi:K'\cap K''\ \to \{0\}\cup [m]\setminus (K'\cup K'')$ is a bounded map.
\end{itemize}
\end{proposition}

\begin{proof} The assertion follows from (already proved) Theorem \ref{th:cuvw refined T BS} and formula \ref{eq:relative LR via BS general}.
\end{proof}

Our next task is to show that equation \eqref{eq:cuvw refined T proof} still holds if we restrict the sum to $\ii$-admissible bounded maps.  In order to prove we can make such a restriction, we need to develop some additional notation.  For any subsets $L\subseteq M$ of $[m]$ such that $|L|+|M|\geq m$ denote by $P(L,M)$ the set of all bounded maps of $L\to \{0\}\cup [m]\setminus M$ given by Definition \ref{def:bounded map}.  Let $\ii=(i_1,\ldots,i_m)\in I^m$ be an admissible sequence (not necessarily repetition free) and let $\ii',\ii''$ denote admissible subsequences of $\ii$ such that $|\ii'|+|\ii''|\geq m.$   The following set will be important to the proceeding calculations.

\smallskip

Define $J$ to be the set of all triples $(K',K'',\varphi)$ which satisfy

\begin{itemize}
\item  $K',K''\subseteq [m]$ such that $\ii_{K'}=\ii'$ and $\ii_{K''}=\ii''.$
\item  $\varphi\in P(K'\cap K'',K'\cup K'').$
\end{itemize}

Observe that the set $J$ depends only on the data $(\ii',\ii'',\ii).$  We will use the capitol letter $\Lambda:=(\ii',\ii'',\varphi)$ to denote such triples.  Proposition \ref{pr:cuvw refined T} is now equivalent to the equation
\begin{equation}\label{eq:unfold_p}p_{\ii',\ii''}^{\ii}=\sum_{\Lambda\in J} p_{\Lambda}\end{equation}
where if $\Lambda=(K',K'',\varphi)$, then $\displaystyle p_{\Lambda}:=\prod_{\ell\in K'\cap K''} p_\ell^{(\varphi)}.$

\smallskip

Recall that for any sequence $\jj\in (I\times [m])^m$, we say the bounded map $\varphi$ is $\jj$-admissible if the sequences $\jj_{M\cup(\varphi(L_{< \ell})\setminus\{0\})}$ are admissible for all $\ell\in L.$  For any sets $L,M$ as above and sequence $\jj\in (I\times[m])^m$, let $P_{\jj}(L,M)\subseteq P(L,M)$ denote the set of all $\jj$-admissible bounded maps and let $$J(\jj):=\{(K',K'',\varphi)\in J\ |\ \varphi\in P_{\jj}((K'\cap K'',K'\cup K')\}.$$  Define the sequence $$\jj(k):=((i_1,1),(i_2,1),\ldots,(i_k,1),(i_{k+1},k+1),(i_{k+2},k+2),\ldots,(i_m,m))$$ and the set $J(k):=J(\jj(k)).$  It is easy to see that $J(1)=J$ and that $J(k)\subseteq J(k-1)$ for any $k\in\{2,\ldots, m\}.$  Theorem \ref{th:cuvw refined T uniform} is equivalent to showing that equation \eqref{eq:unfold_p} still holds if we restrict the sum to $J(m).$  We will prove Theorem \ref{th:cuvw refined T uniform} by induction on $k.$  Clearly for $k=1$, we have that $\jj(1)$ repetition free and hence we are in the case of Proposition \ref{pr:cuvw refined T}.  It suffices to prove the following proposition.

\begin{proposition}\label{pr:admiss_induction}With the assumptions in Theorem \ref{th:cuvw refined T uniform}, we have that
$$\sum_{\Lambda\in J(k-1)\setminus J(k)}p_{\Lambda}=0$$ for any $k\in\{2,\ldots, m\}.$
\end{proposition}

The remainder of this section consists of the proof for Proposition \ref{pr:admiss_induction}.  Hence we will fix the integer $k$ and denote the sequence $\jj(k)$ by simply $\jj.$  For $\Lambda=(K',K'',\varphi)\in J(k-1)\setminus J(k),$ define
$$L=K'\cap K''=(\ell_1<\cdots<\ell_n)\quad\text{and}\quad  M=K'\cup K''=(m_1<\cdots<m_{n'}).$$ For any $\ell_r\in L$ define
$$M(r):=M\cup (\varphi(L_{\leq \ell_r})\setminus\{0\}).$$  For any subset $N\subseteq [m],$ we will denote by $(N)$ the sequence of elements of $N$ arranged in increasing order.  We say that a pair $\{n_1,n_2\}\subseteq N$ is non-admissible if $\jj_{n_1}=\jj_{n_2}$ and $n_1,n_2$ are consecutive in the sequence $(N).$  Since $\jj$ is fixed, this definition of non-admissible pair is well defined.

\smallskip

Since $\Lambda=(K',K'',\varphi)\in J(k-1)\setminus J(k)$, the bounded map $\varphi$ is not $\jj$-admissible. Hence, either $\jj_M$ is not admissible or $\jj_{M(r)}$ is not admissible for some $\ell_r\in L.$  If $\jj_M$ is admissible, let $z$ denote the smallest integer for which $\jj_{M(z)}$ is not admissible.  We partition $J(k-1)\setminus J(k)$ into the following sets:
\begin{eqnarray*}
J_1&:=&\{\Lambda\ |\ \text{$\jj_M$ is not admissible}\}.\\
J_2&:=&\{\Lambda\ |\ \text{$\jj_{M(z)}$ is not admissible and $\varphi(\ell_z), \ell_z$ are consecutive in $(M(z))$}\}.\\
J_3&:=&\{\Lambda\ |\ \text{$\jj_{M(z)}$ is not admissible and $\varphi(\ell_z), \ell_z$ are not consecutive in $(M(z))$}\}.
\end{eqnarray*}

Observe that if $\Lambda\in J_1,$ then $M$ has a unique non-admissible pair since $\Lambda\in J(k-1).$  Similarly, if $\Lambda\in J_2\cup J_3,$ then $M(z)$ has a unique non-admissible pair.  We prove Proposition \ref{pr:admiss_induction} in two steps.

\begin{proposition}\label{pr:1reduction}The sum $\displaystyle \sum_{\Lambda\in J_1\cup J_2}p_{\Lambda}=0.$\end{proposition}
\begin{proof}First suppose that $\Lambda\in J_2$.  Then $\{\varphi(\ell_z)<\ell_z\}$ is the unique non-admissible pair in $M(z).$  Define $\Lambda_1:=(K'_1,K''_1,\varphi_1),\ \Lambda_2:=(K'_2,K''_2,\varphi_2)$ by
$$(K'_1,K''_1):=(K',K''\ominus\{\varphi(\ell_z),\ell_z\})\quad\text{and}\quad (K'_2, K''_2):=(K'\ominus\{\varphi(\ell_z),\ell_z\},K'')$$ where $\ominus$ denotes the symmetric difference operation.  This implies that
$$L_1=L_2=L\setminus\{\ell_z\}\quad\text{and}\quad M_1=M_2=M\cup\{\varphi(\ell_z)\}$$ and hence we define $$\varphi_1=\varphi_2=\varphi|_{L_1}.$$
Clearly we have $\Lambda_1,\Lambda_2\in J_1$ and that $p_{\Lambda_1}=p_{\Lambda_2}.$  Moreover,
\begin{equation}
\label{eq:triplecancel}
p_{\Lambda}+p_{\Lambda_1}+p_{\Lambda_2}=0
\end{equation}
since $\langle\alpha_{\ell_z},\alpha_{\varphi(\ell_z)}\rangle =2.$

\smallskip

Conversely, if $\Lambda_1=(K_1',K_1'',\varphi_1)\in J_1,$ then let $\{m_{z-1}<m_z\}\subseteq M$ denote the non-admissible pair in $(M).$  Note that $\{m_{z-1},m_z\}\cap L_1=\emptyset$ since $\jj_{K'_1}$ and $\jj_{K''}$ are admissible.  Without loss of generality, assume that $m_z\in K'$ (hence $m_{z-1}\in K''$) and define $$\Lambda_2:=(K'_1\ominus\{m_{z-1},m_z\},K''_1\ominus\{m_{z-1},m_z\},\varphi_1)$$ and $$\Lambda:=(K'_1\ominus\{m_{z-1},m_z\},K''_1,\varphi)$$ where $\varphi=\varphi_1\sqcup\{\varphi(m_z)=m_{z-1}\}.$  It is easy to see that $\Lambda_2\in J_1,\ \Lambda\in J_2.$ and the triple $(\Lambda_1,\Lambda_2,\Lambda)$ satisfies \eqref{eq:triplecancel}.  Furthermore, the pairs $\{\Lambda_1, \Lambda_2\}$ form an equivalence relation on $J_1$ and the correspondence $ \{\Lambda_1,\Lambda_2\}\leftrightarrow\Lambda$ is a bijection between the set $J_1$ modulo this equivalence relation and $J_2.$  This proves the proposition.
\end{proof}

\begin{proposition}\label{pr:2reduction}The sum $\displaystyle \sum_{\Lambda\in J_3}p_{\Lambda}=0$.
\end{proposition}

\begin{proof}We prove the proposition by defining an involution on the set $J_3.$  For any $\Lambda\in J_3,$ define the set $$\nu_{\Lambda}:=\{\nu_1<\nu_2\}$$ to be the non-admissible pair in the sequence $(M(z)).$  The set $\nu_{\Lambda}$ is well defined since the sequence $\jj_{M(z-1)}$ is admissible.  Furthermore, $\varphi(\ell_{z})\in\nu_{\Lambda}$ and $\ell_{z}\notin\nu_{\Lambda}$ since $\Lambda\notin J_2.$

For any subset $N\subseteq[m]$ and  $\Lambda\in J_3$ define
$$\sigma_{\Lambda}(N):=\begin{cases}N\ominus\nu_{\Lambda}\quad \text{if $|N\cap\nu_{\Lambda}|=1$}\\
N\quad\qquad\, \text{if $|N\cap\nu_{\Lambda}|\neq 1$}\end{cases}$$  where $\ominus$ denotes the symmetric difference operation.   If $N=\{N_0\}$ is a set with a single element, then we will denote $\sigma_{\Lambda}(N_0):=\sigma_{\Lambda}(\{N_0\})$ (dropping the brackets).


\smallskip

We define an involution $\sigma:J_3\rightarrow J_3$ by: $$\sigma(\Lambda):=(\sigma_{\Lambda}(K'),\sigma_{\Lambda}(K''),\psi)$$
where $\Lambda=(K',K'',\varphi)$ and $\psi$ is defined as follows.  It is easy to check that
$$\sigma_{\Lambda}(K')\cap\sigma_{\Lambda}(K'')=\sigma_{\Lambda}(L)\quad \text{and}\quad \sigma_{\Lambda}(K')\cup\sigma_{\Lambda}(K'')=\sigma_{\Lambda}(M).$$
Define $\psi:\sigma_{\Lambda}(L)\rightarrow \{0\}\cup [m]\backslash\sigma_{\Lambda}(M)$ by
$$\psi(\sigma_{\Lambda}(\ell_k)):=\begin{cases}\sigma_{\Lambda}(\varphi(\ell_k))\quad \text{if $\varphi(\ell_k)\neq 0$}\\
\quad\quad  0\hspace{.5in}  \text{otherwise.}\end{cases}$$
The following properties are due to the fact that $\nu_{\Lambda}$ is an admissible pair in $M(\ell_z),$ and $\varphi(\ell_{z})\in\nu_{\Lambda}.$  For any $\Lambda$ and $\sigma(\Lambda)$ we have
\begin{itemize}
\item $\sigma_{\Lambda}(\ell_{r})=\ell_{r}$ for all $r\geq z$
\item $\psi(\sigma_{\Lambda}(\ell_{r}))=\varphi(\ell_{r})$ for all $r> z$
\item $\sigma_{\Lambda}(M)(r)=\sigma_{\Lambda}(M(r))$ for all $r\in\{0,1,\ldots,n\}.$
\end{itemize}

Clearly, by squaring we get $\sigma^2(\Lambda)=\Lambda$ since $\sigma_{\Lambda}^2(N)=N$ for any subset $N\subseteq[m].$  The following lemma proves that the image of $\sigma$ is contained in $J_3$ and hence $\sigma$ is an involution.

\begin{lemma}If $\Lambda\in J_3$, then $\sigma(\Lambda)\in J_3.$\end{lemma}

\begin{proof}Since $\Lambda$ is fixed, we will denote $\sigma_{\Lambda}(N)$ by simply $\sigma(N)$ for any subset $N$ in this proof. We first show that $\sigma(\Lambda)\in J.$  Observe that $\ii_{\sigma(K')}=\ii_{K'}$ and $\ii_{\sigma(K'')}=\ii_{K''}$ since $\ii_{\nu_1}=\ii_{\nu_2}$.  What we need to show is that $\psi:\sigma(L)\rightarrow\{0\}\cup [m]\setminus \sigma(M)$ is a bounded map.  It suffices to consider $\ell_r$ for which $\varphi(\ell_r)\neq 0.$

\smallskip

If $r> z,$ then we have that $\sigma(M)(r)=M(r)$ since $\nu_{\Lambda}\subseteq M(r)$.  Hence $$\psi(\sigma(\ell_r))=\varphi(\ell_r)<\ell_r=\sigma(\ell_r).$$

\smallskip

If $r=z,$ then $\varphi(\ell_r)\in\nu_{\Lambda}.$  But the fact that $\{\sigma(\varphi(\ell_r)), \varphi(\ell_r)\}=\nu_{\Lambda}$ are consecutive in $M(r)$ implies $$\psi(\sigma(\ell_r))=\sigma(\varphi(\ell_r))<\ell_r=\sigma(\ell_r).$$

\smallskip

If $r<z$, then $|M(r)\cap\nu_{\Lambda}|\leq 1.$  Hence $\sigma$ fixes at least one of $\ell_r$ or $\varphi(\ell_r).$  Thus $$\psi(\sigma(\ell_r))=\sigma(\varphi(\ell_r))<\sigma(\ell_r)$$ since $\nu_1, \nu_2$ are consecutive in $M(z)$ and $M(r)\subseteq M(z).$  This implies that $\psi$ is a bounded map.  Hence $\sigma(\Lambda)\in J.$

\smallskip

Since $\jj_{\nu_1}=\jj_{\nu_2}$, we have that $\jj_{\sigma(M(k))}=\jj_{M(k)}$ for all $k\in\{0,1,\ldots,n\}.$  This implies that $\Lambda\in J(k-1)\setminus J(k)$ if and only if $\sigma(\Lambda)\in  J(k-1)\setminus J(k).$  In particular, it also implies that $\Lambda\in J_3$ if and only if $\sigma(\Lambda)\in J_3$.  This proves the lemma.\end{proof}


Before we prove the proposition, we need one more observation.  Note that each summand in equation \eqref{eq:cuvw refined T proof} has a natural factorization
\begin{equation}\label{eq:factor}\prod_{\ell\in K'\cap K''} p_\ell^{(\varphi)}=\left(\prod_{\ell\in \varphi^{-1}(0)} p_\ell^{(\varphi)}\right)\left(\prod_{\ell'\in \varphi^{-1}([m]\setminus M)} p_{\ell'}^{(\varphi)}\right).\end{equation} We will denote the first factor by $p_\varphi^0$ and the second factor by $p_\varphi^+.$

\begin{lemma}\label{lem:equiv mu}For any $\Lambda\in J_3$ with $p_{\Lambda}=p_\Lambda^0\cdot p_\Lambda^+$ and $p_{\sigma(\Lambda)}=p_{\sigma(\Lambda)}^0\cdot p_{\sigma(\Lambda)}^+,$ we have that $p_\Lambda^0=p_{\sigma(\Lambda)}^0$.\end{lemma}

\begin{proof}It suffices to check the case where $\varphi^{-1}(0)\neq \sigma(\varphi^{-1}(0)).$  Otherwise $p_\Lambda^0=p_{\sigma(\Lambda)}^0$ since $\nu_1$ and $\nu_2$ act identically on $Q.$

\smallskip

If $\varphi^{-1}(0)\neq \sigma(\varphi^{-1}(0)),$ then $\{\varphi(\ell_{z}),\ell_r\}$ must be a non-admissible pair in $M(z)$ for some $r\neq z.$  Moreover, $\varphi(\ell_r)=0$ and $r<z,$ otherwise $\varphi$ would not be bounded.  Since $\{\varphi(\ell_{z}),\ell_r\}$ are a non-admissible pair in $M(z)$, they must be a non-admissible pair in $M(r)\cup\{\varphi(\ell_{z})\}.$  Thus $$p^{(\varphi)}_{\ell_r}=p^{(\psi)}_{\sigma(\ell_r)}.$$  It is easy to see that other factors of $p_\Lambda^0$ and $p_{\sigma(\Lambda)}^0$ are equal.  Thus the lemma is proved. \end{proof}

We are now ready to prove the proposition.  It suffices to show that $p_{\Lambda}=-p_{\sigma(\Lambda)}$ for any $\Lambda\in J_3.$  We first assume that $\nu_2=\varphi(\ell_{z})$.  Then $\nu_1\in M(z-1)$ and hence $$\nu_{\Lambda}\cap\sigma_{\Lambda}(M(z-1))=\{\nu_2\}\ \text{ and }\ \nu_{\Lambda}\subseteq\sigma_{\Lambda}(M(z))=M(z).$$
For any $i_\ell\in I$ we will denote $\alpha_{i_\ell}$ by simply $\alpha_\ell.$  By equation \eqref{eq:plambda1}, if $$p_{\Lambda}=p_\Lambda^0\cdot p_\Lambda^+=p_\Lambda^0\cdot\prod_{\ell_r\in\varphi^{-1}([m]\setminus M)} \langle w_{\ell_r}^{(\varphi)}(-\alpha_{\ell_r}),\alpha_{\varphi(\ell_r)}^{\vee}\rangle,$$
then by Lemma \ref{lem:equiv mu}, we have

\begin{eqnarray*}p_{\sigma(\Lambda)}&=&p_\Lambda^0\cdot\langle s_{\nu_2}w^{(\psi)}_{z}(-\alpha_{\ell_{z}}),\alpha_{\varphi(\ell_{z})}^{\vee}\rangle \prod_{r\neq z} \langle w^{(\psi)}_{\ell_r}(-\alpha_{\ell_r}),\alpha_{\varphi(\ell_r)}^{\vee}\rangle\\
&=&p_\Lambda^0\cdot\langle w^{(\varphi)}_{z}(\alpha_{\ell_{z}}),\alpha_{\varphi(\ell_{z})}^{\vee}\rangle\prod_{r\neq z} \langle w^{(\varphi)}_{\ell_r}(-\alpha_{\ell_r}),\alpha_{\varphi(\ell_r)}^{\vee}\rangle\\ &=&-p_{\Lambda}\end{eqnarray*}

since $\jj_{\nu_1}=\jj_{\nu_2}$.  Note that the other terms ($r\neq z$) in the above product remain unchanged after applying the involution since $\sigma_{\Lambda}(M(r))=M(r)$ for $r\geq z$ and if $r<z$, then the relative position of $\nu_1$ and $\nu_2$ is the same within the sequence $(M(r)).$  A similar argument proves that $p_{\Lambda}=-p_{\sigma(\Lambda)}$ in the case where $\nu_1=\varphi(\ell_{z}).$  This completes the proof of Proposition \ref{pr:2reduction}\end{proof}

Propositions \ref{pr:1reduction} and \ref{pr:2reduction} together prove Proposition \ref{pr:admiss_induction}.  Hence the inductive step in the proof of Theorem \ref{th:cuvw refined T uniform} is complete.

\section{Positivity of Littlewood-Richardson coefficients and proof of Theorem \ref{th:nonneg}}
\label{sect:positivity}

In this section we prove that the generalized Littlewood-Richardson coefficients are positive for a large class of quasi-Cartan matrices.  The following is the main result of this section.


\begin{proposition}\label{pr:positiveparts}

Let $A$ be an $I\times I$ quasi-Cartan matrix such that \eqref{eq:nonnegativity for A special} holds.  Then for any admissible sequence $\ii=(i_1,\ldots, i_m)\in I^m$, we have
\begin{equation}
\label{eq:positive}
w(\alpha_j)\in \sum\limits_{i\in I} \RR_{\ge 0}\cdot \alpha_i\quad \text{and}\quad \langle w(\alpha_{i_m}),\alpha_{i_1}^{\vee}\rangle \leq 0
\end{equation}
where $w=s_{i_2}s_{i_3}\cdots s_{i_{m-1}}.$
\end{proposition}




\begin{proof}

First, we consider the case where the quasi-Cartan matrix is of rank 2.  Let $I=\{1,2\}$ and
$$A:=\left[\begin{array}{cc}2& -a\\-b&2 \end{array}\right].$$
Define the sequences $A_k$ and $B_k$ by
\begin{equation}\label{eq:rank2roots}A_k:=a B_{k-1}-A_{k-2}\quad \text{and}\quad  B_k:=b A_{k-1}-B_{k-2}\end{equation}
where $A_0=B_0=0$ and $A_1=B_1=1$.
These sequences are analogues of Chebyshev polynomials of second kind and are constructed so that if $\ii=(\underbrace{1,2,1,2,\ldots}_{m}),$ then
$$w(\alpha_{i_m})=A_{m-1}\,\alpha_1+B_m\,\alpha_2\quad \text{and}\quad \langle w(\alpha_{i_m}),\alpha_{i_1}^{\vee}\rangle=A_{m-2}-A_{m}$$
and if $\ii=(\underbrace{2,1,2,1,\ldots}_{m})$, then
$$w(\alpha_{i_m})=A_m\,\alpha_1+B_{m-1}\,\alpha_2\quad \text{and}\quad \langle w(\alpha_{i_m}),\alpha_{i_1}^{\vee}\rangle=B_{m-2}-B_{m}.$$

We remark that the sequences $A_k$ and $B_k$ are used by N.~Kitchloo in \cite{Kit} in his study of cohomology of rank 2 Kac-Moody groups.
The following lemma proves Proposition \ref{pr:positiveparts} (and hence Theorem \ref{th:nonneg}) in the rank 2 case.

\begin{lemma}\label{lem:rank2pos}Let $a,b$ be positive real numbers such that $ab\geq 4$, then for any admissible $\ii\in I^m$, we have
$$A_{k}\geq A_{k-2}\quad \text{and}\quad B_{k}\geq B_{k-2}.$$\end{lemma}
\begin{proof}We prove the lemma by induction on $k$.  The lemma is clearly true for $k=2$ since $a,b$ are positive.  In general we have that
$$A_{k+1}=a B_k-A_{k-1}=(ab-1)(A_{k-1})-aB_{k-2}\geq 3A_{k-1}-aB_{k-2}.$$ By induction, we have that $B_{k}\geq B_{k-2}$.  Hence
 $$A_{k+1}\geq 3A_{k-1}-aB_{k-2}\geq 2A_{k-1}+(A_{k-1}-aB_{k})=2A_{k-1}-A_{k+1}.$$  This implies that $2A_{k+1}\geq 2A_{k-1}.$
 A similar argument proves the proposition for the sequence $B_k.$  This completes the proof.\end{proof}

We now consider the case of a quasi-Cartan matrix of arbitrary rank.
For any $j,k\in I$ let $W_{j,k}$ denote the dihedral subgroup of $W$ generated by $s_j,s_k$.
We need the following well-known fact about Coxeter groups.

\begin{lemma}
\label{le:Coxeter cancelation}
For any $w\in W$ and $j,k\in I$
there exist elements $w'\in W$, $w''\in W_{j,k}$ such that
\begin{equation}
\label{eq:Coxeter cancelation}
w=w'w'',\quad \ell(w)=\ell(w')+\ell(w''),\quad \ell(w's_j)=\ell(w's_k)=\ell(w)+1\ .
\end{equation}
In particular, the pair $(w',w'')$ is unique and
$$\ell(ws_j)-\ell(w)=\ell(w''s_j)-\ell(w''),\quad \ell(ws_k)-\ell(w)=\ell(w''s_k)-\ell(w'')\ .$$
\end{lemma}


Now we prove Proposition \ref{pr:positiveparts} by induction in $\ell(w)$.  If $w\in W_{j,k}$ for some $j,k \in I,$ then we are done by Lemma \ref{lem:rank2pos}.
Otherwise, by Lemma \ref{le:Coxeter cancelation} there exists $w'\in W\setminus \{1\}$  and
$w''\in W_{jk}$ satisfying \eqref{eq:Coxeter cancelation}. Since $\ell(w'')<\ell(w)$ and $w''$ satisfies the assumptions of the proposition, we obtain:
$$w''(\alpha_j)\in \RR_{\ge 0}\cdot \alpha_j+\RR_{\ge 0}\cdot \alpha_k$$
Since $\ell(w')<\ell(w)$ and $w'$ also satisfies the assumption of the proposition, the inductive hypothesis \eqref{eq:positive} applies to this $w''$ and we obtain:
$$w(\alpha_j)=w'w''(\alpha_j)\in w'(\RR_{\ge 0}\cdot \alpha_j+\RR_{\ge 0}\cdot \alpha_k)$$
$$=\RR_{\ge 0}\cdot w'(\alpha_j)+\RR_{\ge 0}\cdot w'(\alpha_k)\subset \sum_{i\in I}\RR_{\ge 0}\cdot \alpha_i.$$ This proves the first part of \eqref{eq:positive}.
To prove the second part of \eqref{eq:positive}, note that $\ell(s_iws_j)-\ell(ws_j)=\ell(s_iw'')-\ell(w'')=1.$  Therefore, the inductive hypothesis \eqref{eq:positive} applies to this $w'$ and we obtain
$$ \langle w(\alpha_j),\alpha_i^{\vee}\rangle\in \langle \RR_{\ge 0}\cdot w'(\alpha_j)+\RR_{\ge 0}\cdot w'(\alpha_k),\alpha_i^{\vee}\rangle$$
$$=\RR_{\ge 0}\cdot \langle w'(\alpha_j),\alpha_i^{\vee}\rangle+\RR_{\ge 0}\cdot \langle w'(\alpha_k),\alpha_i^{\vee}\rangle\subset \RR_{\ge 0}\cdot \RR_{\le 0}+\RR_{\ge 0}\cdot \RR_{\le 0}=\RR_{\le 0}\ .$$  The proposition is proved.\end{proof}


By replacing the quasi-Cartan matrix $A$ with $(1+t)A-2t\cdot  Id$, we obtain the following result.

\begin{proposition}
\label{pr:positiveparts t}
In the notation of Proposition \ref{pr:positiveparts}, let $A_t=(1+t)\cdot A-2t \cdot Id$ be the $I\times I$ quasi-Cartan matrix over $\RR[t]$, where $A$ is a quasi-Cartan matrix over $\RR$  such that for each $i\neq j$ we have $a_{ij}\leq 0$ and $a_{ij}a_{ji}\geq 4.$  Then for any admissible sequence $\ii=(i_1,\ldots, i_m)\in I^m$, we have

\begin{equation*}
w(\alpha_{i_m})\in \sum\limits_{i\in I} \RR_{\ge 0}[t]\cdot \alpha_i\quad \text{and}\quad \langle w(\alpha_{i_m}),\alpha_{i_1}^{\vee}\rangle \in \RR_{\le 0}[t]\ .
\end{equation*}
where $w=s_{i_2}\cdots s_{i_{m-1}}.$
\end{proposition}

\begin{proof} Define a partial order on $\RR[t]$ by saying that $p\ge q$ if $p-q\in \RR_{\ge 0}[t]$. Following the proof of Lemma \ref{lem:rank2pos}
we obtain (by replacing $(a,b)$ with $((t+1)a,(t+1)b)$ and sequences $\{A_k\}$, $\{B_k\}$ with $\{A_k(t)\},~\{B_k(t)\}\subset \RR[t]$).  We prove by induction the following two statements
$$A_{k}(t)\geq A_{k-2}(t)\quad \text{and}\quad B_{k}(t)\geq B_{k-2}(t).$$

The lemma is clearly true for $k=2$ since $A_2(t)=at+a,\ B_2(t)=bt+b$ and $a,b$ are positive.  In general we have that
\begin{eqnarray*}A_{k+1}(t)&=&(at+a )B_k(t)-A_{k-1}(t)\\
&=&(ab(t^2+2t)+ab-1)A_{k-1}(t)-(at+a)B_{k-2}(t)\\
&\geq& (ab(t^2+2t)+3)A_{k-1}(t)-(at+a)B_{k-2}(t).\end{eqnarray*}
By induction, we have that $B_{k}(t)\geq B_{k-2}(t)$.  Hence
\begin{eqnarray*} A_{k+1}(t)&\geq&(ab(t^2+2t)+3)A_{k-1}(t)-(at+a)B_{k-2}(t)\\
&\geq& (ab(t^2+2t)+2)A_{k-1}(t)+(A_{k-1}(t)-(at+a)B_{k}(t))\\
&=& (ab(t^2+2t)+2)A_{k-1}(t)-A_{k+1}(t).\end{eqnarray*}  This implies that
$$2(A_{k+1}(t)-A_{k-1}(t))\geq ab(t^2+2t)A_{k-1}.$$
Similarly the polynomials $B_k(t)$ satisfies the same inequality.  This proves the proposition.\end{proof}



Now we are ready to prove Theorem \ref{th:nonneg} and verify Conjecture \ref{conj:positive puvw(t)} in a number of cases. Indeed, for any $(K',K'',L,\varphi)$ as in Theorem \ref{th:cuvw refined T}, the sequence $\ii_{(K'\cup K'')_{<\ell}\cap \varphi(L_{<\ell})}$ is admissible for all $\ell\in K'\cap K''$, therefore, $w_\ell(\alpha_{i_\ell})\in \sum\limits_{i\in I} \RR_{\ge 0}\cdot \alpha_i$ for all $\ell\in (K'\cap K'')\setminus L$ by \eqref{eq:positive}  and $\langle w_\ell(\alpha_{i_\ell}),-\alpha_{\varphi_{i_\ell}}\rangle \ge 0$ for all $\ell\in L$, again, by  \eqref{eq:positive}.

This proves Theorem  \ref{th:nonneg}. \endproof

Same argument, in conjunction with Proposition \ref{pr:positiveparts t} verifies Conjecture \ref{conj:positive puvw(t)} in the assumption that \ref{eq:nonnegativity for A special} holds.

\section{Examples}\label{sect:examples}

In this section we apply Theorem \ref{th:cuvw} to compute Littlewood-Richardson coefficients in several cases.  In the first example, we consider any rank 2 quasi-Cartan matrices and demonstrate that Theorem \ref{th:cuvw} agrees with formulas developed in \cite{Kit} and \cite{BerKap}.  The following examples we look at particular computations in finite Coxeter types $A_n$ and $H_3.$  The computer algebra program MuPAD Pro and 'Combinat' package was used in many of these calculations.

\subsection{The rank 2 case}\label{subsect:rank2ex}   We give a full analysis in the case where $A$ is a rank 2 quasi-Cartan matrix.  Let $I=\{1,2\}$ and consider the quasi-Cartan matrix
$$A:=\left[\begin{array}{cc}2& -a\\-b&2 \end{array}\right]$$ as in the previous section.  Define $$u_m=\underbrace{\cdots s_1s_2s_1}_{m}\quad \text{and}\quad v_m=\underbrace{\cdots s_2s_1s_2}_{m}$$ to be the unique elements in $W$ corresponding to the two admissible sequences of length $m.$  We first compute non-equivariant coefficients $c_{u,v}^w$ in the case where $\ell(u)+\ell(v)=\ell(w).$  Let $k\leq m.$  Theorem \ref{th:cuvw} implies that
$$c^{u_m}_{u_k,u_{m-k}}=c^{v_{m+1}}_{v_{k+1},u_{m-k}}=c^{v_{m+1}}_{u_{k},v_{m-k+1}},$$
$$c^{v_m}_{v_k,v_{m-k}}=c^{u_{m+1}}_{u_{k+1},v_{m-k}}=c^{u_{m+1}}_{v_{k},u_{m-k+1}}$$
and
$$c^{v_m}_{u_k,u_{m-k}}=c^{u_{m}}_{v_{k},v_{m-k}}=0.$$
Hence it suffices to compute coefficients $c^{u_m}_{u_k,u_{m-k}}$ and $c^{v_m}_{v_k,v_{m-k}}.$  Recall the sequences $A_k$ and $B_k$ defined in \eqref{eq:rank2roots}. For $k\leq m$, define the binomial coefficients
$$C(k,m):=\frac{A_mA_{m-1}\cdots A_1}{(A_kA_{k-1}\cdots A_1)(A_{m-k}A_{m-k-1}\cdots A_1)}$$
$$D(k,m):=\frac{B_mB_{m-1}\cdots B_1}{(B_kB_{k-1}\cdots B_1)(B_{m-k}B_{m-k-1}\cdots B_1)}.$$

\begin{theorem}\label{th:Kitchloo}Let A be a rank 2 quasi-Cartan matrix.  The coefficients $$c^{u_m}_{u_k,u_{m-k}}=C(k,m)\quad \text{and}\quad c^{v_m}_{v_k,v_{m-k}}=D(k,m).$$\end{theorem}

We remark that the above formula has been proved by Kitchloo in \cite[Section 10]{Kit} in the case where $A$ is the Cartan matrix of some Kac-Moody group and by the first author and Kapovich in \cite[Section 13]{BerKap} in the case where $A$ is symmetric.  We show that Theorem \ref{th:cuvw} implies Theorem \ref{th:Kitchloo} for any rank 2 quasi-Cartan matrix.  First, it is easy to check that Theorem \ref{th:Kitchloo} is true for $m=1$ and 2.  We will show that the coefficients $c^{u_m}_{u_k,u_{m-k}}$ and $c^{v_m}_{v_k,v_{m-k}}$ can be constructed by a second order recurrence relation using Theorem \ref{th:cuvw}.  We will then show that $C(k,m)$ and $D(k,m)$ also satisfy this relation.

\smallskip

Let $\ii=(\ldots,1,2,1)$ be the reduced expression of $u_m$.  If ${\bf u}, {\bf v}\subset [m]$ are such that $\ii_{{\bf u}}$ and $\ii_{{\bf v}}$ are reduced expressions for $u_m$ and $u_{m-k}$ respectively, then there is at most one admissible bounded bijection
$$\varphi:{\bf u}\cap{\bf v}\rightarrow [m]\setminus({\bf u} \cup{\bf v}).$$  Moreover, if $\varphi$ exists, then $[m]\setminus({\bf u} \cup{\bf v})=(1,2,\ldots,|{\bf u}\cap{\bf v}|).$ Define $$\mathcal{J}(m,k):=\{({\bf u},{\bf v})\ |\ (\ii_{{\bf u}},\ii_{{\bf v}})\in R(u_m)\times R(u_{m-k})\ \text{and}\  \varphi\ \text{exists}\}.$$
If ${\bf u}\cap{\bf v}=\emptyset$, our convention will be that $\varphi$ exists.  If ${\bf z}\in \mathcal{J}(m,k)$, then let $\varphi_{\bf z}$ denote the corresponding $\ii$-admissible bounded bijection.  Theorem \ref{th:cuvw} says that
$$c^{u_m}_{u_k,u_{m-k}}=\sum_{{\bf z}\in \mathcal{J}(m,k)}p_{\varphi_{\bf z}}.$$
Define the subset
$$ \mathcal{J}_1:=\{({\bf u},{\bf v})\in \mathcal{J}(m,k)\ |\ m\in {\bf u}\cap {\bf v}\}.$$
If ${\bf z}\in J_1,$ then $\varphi_z(m)=1$ since $\varphi_{\bf z}$ is $\ii$-admissible.  Hence the partition $\mathcal{J}(m,k)= \mathcal{J}_1\sqcup \mathcal{J}(m,k)\setminus  \mathcal{J}_1$ induces the recursion

\begin{eqnarray*}c^{u_m}_{u_k,u_{m-k}}&=&\sum_{{\bf z}\in J_1}p_{\varphi_{\bf z}}+\sum_{{\bf z}'\in\mathcal{J}(m,k)\setminus J_1}p_{\varphi_{{\bf z}'}}\\
&=&\langle v_{m-2}(-\alpha_1),\alpha_{i_1}^{\vee}\rangle\cdot c^{v_{m-2}}_{v_{k-1},v_{m-k-1}}+(c^{u_{m-2}}_{u_{k-2},u_{m-k}}+c^{u_{m-2}}_{u_{k},u_{m-k-2}}).\end{eqnarray*}
Now assume that Theorem \ref{th:Kitchloo} is true for all integers less than $m.$  Then
\begin{equation}\label{eq:binomrecur}c^{u_m}_{u_k,u_{m-k}}=\langle w(-\alpha_{i_m}),\alpha_{i_1}^{\vee}\rangle D(k-1,m-2)+C(k-2,m-2)+C(k,m-2).\end{equation}
The following lemma will be important to the proceeding calculations.
\begin{lemma}\label{lem:seq_ident}Let $A_m$ and $B_m$ be sequence defined in \eqref{eq:rank2roots}.  Then the following identities are true:
\begin{enumerate}
\item  If $m$ is odd, then $A_m=B_m$. If $m$ is even, then $bA_m=aB_m.$
\item For any $k\leq m$, if $k$ is odd and $m$ is even, then $$bC(k,m)=aD(k,m).$$  Otherwise $$C(k,m)=D(k,m).$$
\item For any $k\leq m$, if $k$ and $m$ are both even, then $$bA_kA_m=a(A_{m+k-1}+A_{m+k-3}+\cdots+A_{m-k+1}).$$ Otherwise
$$A_kA_m=A_{m+k-1}+A_{m+k-3}+\cdots+A_{m-k+1}.$$\end{enumerate} \end{lemma}

\begin{proof} Part (1) follows from a simple inductive argument and the construction of $A_m$ and $B_m$ in \eqref{eq:rank2roots}.  Part (2) is a direct consequence of part (1).  For part (3) we observe that for any $1<k\leq m$, we have $$A_2B_m=A_{m+1}+A_{m-1}$$ and $$A_kB_m=A_2A_{k}A_{m-1}-A_kB_{m-2}.$$  Part (3) now follows from another inductive argument and part (1).\end{proof}

We prove Theorem \ref{th:Kitchloo} by considering three cases.  First assume that $m$ is odd.  By Lemma \ref{lem:seq_ident}, equation \eqref{eq:binomrecur} becomes
\begin{eqnarray}\label{eq:binomrecur2}c^{u_m}_{u_k,u_{m-k}}&=&C(m-2,k)+(A_m-A_{m-2})C(m-2,k-1)+C(m-2,k-2)\notag\\
&=&\tilde A \left(\frac{A_{m-2}A_{m-3}\cdots A_{m-k+1}}{A_kA_{k-1}\cdots A_1}\right)\end{eqnarray}
where
$$\tilde A=A_{m-k}A_{m-k-1}+(A_m-A_{m-2})A_kA_{m-k}+A_kA_{k-1}.$$  Using Lemma \ref{lem:seq_ident} part (3), $\tilde A$ simplifies to $$\tilde A=A_mA_{m-1}$$ and thus $c^{u_m}_{u_k,u_{m-k}}=C(m,k).$

\smallskip

If $k$ and $m$ are both even, then equation \eqref{eq:binomrecur2} for $c^{u_m}_{u_k,u_{m-k}}$ still holds by replacing $\tilde A$ with
\begin{eqnarray*}\tilde A'&=&A_{m-k}A_{m-k-1}+(B_m-B_{m-2})A_kA_{m-k}+A_kA_{k-1}\\
&=&A_{m-k}A_{m-k-1}+\frac{b}{a}(A_m-A_{m-2})A_kA_{m-k}+A_kA_{k-1}.\end{eqnarray*}
But this expression still simplifies to equal $A_mA_{m-1}$ by applying Lemma \ref{lem:seq_ident} part (3) in the case where $k$ and $m-k$ are both even.

\smallskip

Finally, if $k$ is odd and $m$ is even, then \begin{eqnarray*}c^{u_m}_{u_k,u_{m-k}}&=&C(m-2,k)+(B_m-B_{m-2})D(m-2,k-1)+C(m-2,k-2)\notag\\
&=&C(m-2,k)+\frac{b}{a}(A_m-A_{m-2})\frac{a}{b}C(m-2,k-1)+C(m-2,k-2)\notag\\
&=&\tilde A \left(\frac{A_{m-2}A_{m-3}\cdots A_{m-k+1}}{A_kA_{k-1}\cdots A_1}\right)\end{eqnarray*}  with $\tilde A$ again simplifying to equal $A_mA_{m-1}.$  To complete the proof of Theorem \ref{th:Kitchloo} we observe that this same argument applies to computing $c^{v_m}_{v_k,v_{m-k}}.$

\begin{remark}
\label{rem:binomial polynomials}
In \cite[Section 13]{BerKap}, the first author and Kapovich consider the case where $a=b=t+t^{-1}$ where $t$ is some formal parameter.  In this case $$A_k=B_k=[k]_t:=t^{k-1}+t^{k-3}+\cdots+t^{1-k}$$ and $$C(k,m)=D(k,m)=\left[\begin{array}{c}m\\k\end{array}\right]_t:=\frac{[m]_t!}{[k]_t![m-k]_t!}$$ are $t$-binomial coefficients used in the study of quantum groups.  Theorem \ref{th:cuvw} provides an interesting decomposition identity for these binomial coefficients.\end{remark}

We conclude our rank 2 examples by computing some equivariant Littlewood-Richardson coefficients $c_{u,v}^w$ where $\ell(u)+\ell(v)>\ell(w).$  Let $w=u_5$, $u=u_3$ and $v=u_4.$  Let $[5]=(1',2',3',4',5')$ denote the index sequence of $\ii=(1,2,1,2,1).$  It is easy to see that $u_3$ appears as a subsequence four times given by the subsequences $$(1',2',3'),\ (1',2',5'),\ (1',4',5'),\ (3',4',5')$$ and $u_4$ appears once as the subsequence $(2',3',4',5').$  These subsequences yield the following quadruples $({\bf u}, {\bf v}, L, \varphi)$ as in Theorem \ref{th:cuvw refined T}.  In the table below, we list the set $L':=({\bf u}\cap {\bf v})\setminus L.$

\bigskip

\begin{center}\begin{tabular}{|c|c|c|c|c|c|}\hline
${\bf u}$& ${\bf v}$&$L'$&$\varphi$ &$ p_{\varphi}$&$\alpha$\\ \hline
$(1',2',3')$&$(2',3',4',5')$&$(2',3')$&$L=\emptyset$ &$1$ & $u_1(\alpha_2)\cdot v_2(\alpha_1)$\\ \hline
$(1',2',5')$&$(2',3',4',5')$&$(2',5')$&$L=\emptyset$ &$1$ & $u_1(\alpha_2)\cdot v_4(\alpha_1)$\\ \hline
$(1',4',5')$&$(2',3',4',5')$&$(4',5')$&$L=\emptyset$ &$1$ & $u_3(\alpha_2)\cdot v_4(\alpha_1)$\\ \hline
$(3',4',5')$&$(2',3',4',5')$&$(3',4')$&$(5')\mapsto (1')$ &$\langle-v_3(\alpha_1),\alpha_1^{\vee}\rangle$ & $v_1(\alpha_1)\cdot u_2(\alpha_2)$\\ \hline
$(3',4',5')$&$(2',3',4',5')$&$(3',5')$&$(4')\mapsto (1')$ &$\langle-u_2(\alpha_2),\alpha_1^{\vee}\rangle$ & $v_1(\alpha_1)\cdot v_4(\alpha_1)$\\ \hline
$(3',4',5')$&$(2',3',4',5')$&$(4',5')$&$(3')\mapsto (1')$ &$\langle-v_1(\alpha_1),\alpha_1^{\vee}\rangle$ & $u_3(\alpha_2)\cdot v_4(\alpha_1)$\\ \hline
\end{tabular}\end{center}

\bigskip

In this case, all bounded bijections are also $\ii$-admissible bounded bijections.  Summing these terms gives
\begin{eqnarray*}c_{u,v}^w&=&u_1(\alpha_2)\cdot v_2(\alpha_1)+u_1(\alpha_2)\cdot v_4(\alpha_1)+u_3(\alpha_2)\cdot v_4(\alpha_1)+(A_5-A_3)\, v_1(\alpha_1)\cdot u_2(\alpha_2)\\& & +(A_4-A_3)\, v_1(\alpha_1)\cdot v_4(\alpha_1)+(A_3-1)\, u_3(\alpha_2)\cdot v_4(\alpha_1).\end{eqnarray*}

\bigskip

For another example, let $w=u_5$, $u=v=u_3.$  Again, let $[5]=(1',2',3',4',5')$ denote the index sequence $\ii=(1,2,1,2,1).$  Using the notation of Theorem \ref{th:cuvw refined T} we have the following quadruples $({\bf u}, {\bf v}, L, \varphi)$ (with $L'=({\bf u}\cap {\bf v})\setminus L$).

\bigskip

\begin{center}\begin{tabular}{|c|c|c|c|c|c|}\hline
${\bf u}$& ${\bf v}$&$L'$&$\varphi$ &$ p_{\varphi}$&$\alpha$\\ \hline
$(1',2',3')$&$(1',4',5')$&$(1')$&$L=\emptyset$ &$1$ & $\alpha_1$\\ \hline
$(1',4',5')$&$(1',2',3')$&$(1')$&$L=\emptyset$ &$1$ & $\alpha_1$\\ \hline
$(1',2',3')$&$(3',4',5')$&$(3')$&$L=\emptyset$ &$1$ & $v_2(\alpha_1)$\\ \hline
$(3',4',5')$&$(1',2',3')$&$(3')$&$L=\emptyset$ &$1$ & $v_2(\alpha_1)$\\ \hline
$(1',2',5')$&$(3',4',5')$&$(5')$&$L=\emptyset$ &$1$ & $v_4(\alpha_1)$\\ \hline
$(3',4',5')$&$(1',2',5')$&$(5')$&$L=\emptyset$ &$1$ & $v_4(\alpha_1)$\\ \hline
$(3',4',5')$&$(3',4',5')$&$(3')$&$(4',5')\mapsto (2',1')$ &$\langle-u_1(\alpha_2),\alpha_2^{\vee}\rangle\cdot\langle-v_3(\alpha_1),\alpha_1^{\vee}\rangle$ & $\alpha_1$\\ \hline
$(3',4',5')$&$(3',4',5')$&$(4')$&$(3',5')\mapsto (2',1')$ &$\langle-\alpha_1,\alpha_2^{\vee}\rangle\cdot\langle-v_3(\alpha_1),\alpha_1^{\vee}\rangle$ & $u_2(\alpha_2)$\\ \hline
$(3',4',5')$&$(3',4',5')$&$(5')$&$(3',4')\mapsto (2',1')$ &$\langle-\alpha_1,\alpha_2^{\vee}\rangle\cdot\langle-u_2(\alpha_2),\alpha_1^{\vee}\rangle$ & $v_4(\alpha_1)$\\ \hline
\end{tabular}\end{center}

\bigskip

Summing these nine terms gives
$$c_{u,v}^w=2(\alpha_1+v_2(\alpha_1)+v_4(\alpha_1))+(A_3-1)((A_5-A_3)\, \alpha_1+A_2u_2(\alpha_2))+A_2(A_4-A_2)\, v_4(\alpha_1).$$
In this case, there would be 20 terms in the above sum if we did not make the ``admissible" restriction.

\subsection{Finite Type $A$ examples}\label{sect:typeA}  In this section we demonstrate some calculations in finite type $A_n.$   Let $I=\{1,2,\ldots,n\}$ and $A=(a_{i,j})$ be matrix where
$$a_{i,i}=2, \quad a_{i,i+1}=a_{i,i-1}=-1,\quad \text{and}\quad  a_{i,j}=0\ \text{if}\ |i-j|>1.$$
In this case $W$ is the symmetric group generated by order 2 simple reflections $\{s_i\ |\ i\in I\}$ with Coxeter relations
$$(s_is_{i+1})^3=(s_is_j)^2=1$$

where $|i-j|>1.$  Let $\ii=(3,2,1,3,2)$ and $w=s_3s_2s_1s_3s_2$.  We compute $c_{u,v}^w$ where

$$u=s_1s_3=s_3s_1\quad \text{and}\quad v=s_1s_3s_2=s_3s_1s_2.$$  Let $[5]=(1',2',3',4',5')$ denote the index sequence of the reduced sequence $\ii.$  By Theorem \ref{th:cuvw}, we need to find all triples $({\bf u},{\bf v}, \varphi)$ which satisfy the conditions given in \eqref{eq:cuvw}.  In this case, there are four triples given by the table below.

\bigskip

\begin{center}\begin{tabular}{|c|c|c|c|}\hline
${\bf u}$& ${\bf v}$&$\varphi$ &$ p_{\varphi}$\\ \hline
$(3',4')$&$(1',3',5')$&$(3')\mapsto (2')$&$\langle-\alpha_1,\alpha_2^{\vee}\rangle=1$\\ \hline
$(1',3')$&$(3',4',5')$&$(3')\mapsto (2')$&$\langle-\alpha_1,\alpha_2^{\vee}\rangle=1$\\ \hline
$(3',4')$&$(3',4',5')$&$(3',4')\mapsto (1',2').$&$\langle-\alpha_1,\alpha_3^{\vee}\rangle\cdot\langle-s_1\alpha_3,\alpha_3^{\vee}\rangle=0\cdot 1=0$\\ \hline
$(3',4')$&$(3',4',5')$&$(3',4')\mapsto (2',1').$&$\langle-\alpha_1,\alpha_2^{\vee}\rangle\cdot\langle-s_2s_1\alpha_3,\alpha_3^{\vee}\rangle=1\cdot -1=-1$\\ \hline \end{tabular}\end{center}

\bigskip

Summing the numbers $p_{\varphi}$, we get that $$c_{u,v}^w=1+1+0-1=1.$$
This example demonstrates that for some triples, $({\bf u},{\bf v}, \varphi)$, we can have $p_{\varphi}<0$ under the conditions given in Theorem \ref{th:cuvw}.  Hence nonnegativity is not immediately implied by Theorem \ref{th:cuvw} for finite type $A$ coefficients.  Observe that the decomposition sum in \eqref{eq:cuvw} for $c_{u,v}^w$ depends strongly on the choice of the reduced word of $w$. Instead, if we choose reduced word $\ii'=(2,3,1,2,1)\in R(w)$, then there is only one term in the decomposition sum \eqref{eq:cuvw} given by

\bigskip

\begin{center}\begin{tabular}{|c|c|c|c|}\hline
${\bf u}$& ${\bf v}$&$\varphi$ &$ p_{\varphi}$\\ \hline
$(2',5')$&$(2',3',4')$&$(2')\mapsto (1')$&$\langle-\alpha_3,\alpha_2^{\vee}\rangle=1$\\ \hline  \end{tabular}\end{center}

\bigskip

Again we get that $c_{u,v}^w=1$, however the decomposition is obviously simpler and trivially positive.  To measure the complexity of these decompositions we compute the polynomials $c_{u,v}^\ii(t)$ for each $\ii\in R(w)$  (recall these polynomials are defined at the end of the introduction).  We get

\bigskip

\begin{center}\begin{tabular}{|c|c|}\hline
$\ii$ & $c_{u,v}^\ii(t)$\\ \hline
(2,3,1,2,1) &  $t + 1$\\
(2,1,3,2,1) &  $t + 1$\\
(2,3,2,1,2) &  $(t + 1)^2$\\
(3,2,3,1,2) &  $(t + 1)^3$\\
(3,2,1,3,2) &  $(t + 1)^3$\\ \hline  \end{tabular}\end{center}

\bigskip

Observe that, in this example, the polynomial $c_{u,v}^\ii(t)$ is invariant under commuting relations in $R(w)$.  Also, each polynomial has nonnegative coefficients and evaluation at $t=0$ recovers the corresponding Littlewood-Richardson coefficient.


\smallskip

For a larger example, let $\ii=(5,2,3,4,3,1,2,1)$ and $w=s_5s_2s_3s_4s_3s_1s_2s_1$.  Let
$$u=s_4s_2\quad \text{and}\quad v=s_3s_4s_3s_1s_2s_1.$$
In this case, $u$ has two reduced words and $v$ has 19 reduced words.  Of these, there are only three triples $({\bf u},{\bf v},\varphi)$ which satisfy the conditions in \eqref{eq:cuvw}.  We get that $$c_{u,v}^w=0+1+1=2.$$  If we take the reduced word $\ii'=(5,2,4,3,2,1,2,4)\in R(w)$, then there are ten triples which yield $$c_{u,v}^w=-1+0+0+0+0+0+0+1+1+1=2.$$  As in the previous example, the polynomials $c_{u,v}^\ii(t)$ are invariant in the commutativity classes in $R(w)$.  Of the 64 reduced word decompositions of $w,$ we get 5 distinct polynomials, which correspond to the 5 commutativity classes in $R(w)$.  These polynomials, along with the size of each commutativity class, is listed below.

\bigskip

\begin{center}\begin{tabular}{|c|c|c|}\hline
$[\ii]$ &$|[\ii]|$ &$c_{u,v}^\ii(t)$\\ \hline
$[(5,2,3,4,3,1,2,1)]$ & 14 &$\left(t+1\right)\cdot \left(2t^2+4t+1\right)\cdot \left(t^2+2t+2\right)$\\
$[(5,2,4,3,4,1,2,1)]$ & 30 &$\left(3t^2+6t+2\right)\cdot \left(t+1\right)^3$\\
$[(5,4,3,2,3,4,1,2)]$ & 5 &$ 2(t + 1)^6 $\\
$[(5,2,4,3,4,2,1,2)]$ & 12 &$\left(2 t^4+8t^3+11t^2+6t+2\right)\cdot \left(t+1\right)^3$\\
$[(5,2,3,4,3,2,1,2)]$ & 3 &$\left(t+1\right)\cdot \left(t^2 + 2t + 2\right)\cdot\left(2t^4+8t^3+10t^2+4t+1\right)$\\ \hline  \end{tabular}\end{center}

\bigskip

Once again, observe that the coefficients of $c_{u,v}^\ii(t)$ are nonnegative.

\subsection{Finite type $H_3$ examples}\label{sect:H3}  Let $\displaystyle \rho:=2\cos\left(\frac{\pi}{5}\right)$ and consider the quasi-Cartan matrix $$A:=\left[\begin{array}{ccc}2&-\rho&0\\-\rho&2&-1\\0&-1&2 \end{array}\right].$$ The group $W$ has three order $2$ generators $s_1,s_2,s_3$ which satisfy the following relations:
$$(s_1s_2)^5=(s_1s_3)^2=(s_2s_3)^3=1.$$

The group $W$ has 120 elements with the longest element having a Coxeter length of 15.  The group $W$ is referred to as the finite Coxeter group $H_3.$  While $H_3$ appears in the classification of finite irreducible Coxeter groups, it does not appear in the classification of finite root systems in Lie theory.  Hence $H_3$ is not a Weyl group of any Lie group or Kac-Moody group.  We give all structure coefficients $c_{u,v}^w$ where $\ell(u)=\ell(v)=2$ and $\ell(w)=4$ in the form of a multiplication table of dual elements $\{\sigma_u\ |\ \ell(u)=2\}$.  There are 5 elements of length 2 elements which can be represented by the following sequences: $$(12), (21),(23),(32),(13)$$ and 9 elements of length 4 represented by the sequences: $$(1212), (2121),(1321),(2312),(2123),(3212),(1323),(1213),(2123).$$

\begin{center}\begin{tabular}{|c|c|c|}\hline
 &$\sigma_{12}$&$ \sigma_{21}$\\ \hline
$ \sigma_{12}$&$\sigma_{1212}+\rho\,\sigma_{2312}+\rho^2\,\sigma_{3212}$& \\ \hline
$ \sigma_{21}$&$\rho\,(\sigma_{1212}+\sigma_{2121})+\sigma_{1321}+\rho\,\sigma_{3212}$&$\sigma_{2121}+2\rho\, \sigma_{1321}$\\ \hline
$ \sigma_{13}$&$\rho\,(\sigma_{2123}+\sigma_{3212}+\sigma_{2312})+\sigma_{1321}+\sigma_{1231}$&$\rho\,(\sigma_{1321}+\sigma_{1231})+\rho\,\sigma_{2321}$ \\ \hline
$ \sigma_{23}$&$\sigma_{2312}+\sigma_{1323}+\rho\,\sigma_{2123}$&$\sigma_{2321}+\sigma_{2123}+\rho\,\sigma_{1231}$  \\ \hline
$ \sigma_{32}$&$\rho\,(\sigma_{2312}+\sigma_{3212})$&$\sigma_{2312}+\sigma_{3212}+\rho\,\sigma_{1321}$  \\ \hline
\end{tabular}\end{center}

\begin{center}\begin{tabular}{|c|c|c|c|}\hline
 &$\sigma_{13}$&$ \sigma_{23}$&$\sigma_{32}$\\ \hline
$ \sigma_{13}$&$\rho^2\,\sigma_{1231}+\rho\,(\sigma_{2123}+\sigma_{2321})$& &\\ \hline
$ \sigma_{23}$&$\rho^2\,\sigma_{2123}+\sigma_{1231}$ & $\rho^2\, \sigma_{2123}$ & \\ \hline
$ \sigma_{32}$&$\sigma_{2312}+\sigma_{2321}+\sigma_{1321}$&  $\rho\, \sigma_{1323}$& $\rho\, \sigma_{2312}$\\ \hline
\end{tabular}\end{center}

\bigskip

Clearly, the Littlewood-Richardson numbers computed above are not all integral, however they are nonnegative since $\rho$ is positive.   This evidence supports Conjecture \ref{conj:nonneg} on nonnegativity given in the introduction.  We end by giving an example of the polynomial $c_{u,v}^\ii(t)$ for $H_3.$  Let $w=(1,2,1,2,3,1,2)$, $u=(3,1,2,3)$ and $v=(1,3,2)$.  In this case we get that $c_{u,v}^w=\rho^2.$  The set $R(w)$ has 5 elements with 3 commutativity classes.  The polynomials  $c_{u,v}^\ii(t)$ are given by the following table where $\bar\rho:=\rho-1$ and $$P:=(\rho t+\rho)^2(\rho t+\bar\rho)(t+\rho).$$

\smallskip

\begin{center}\begin{tabular}{|c|c|}\hline
$\ii$ & $c_{u,v}^\ii(t)$\\ \hline
(1,2,1,2,3,1,2)& $P\cdot((\rho^2+1)(t^2+2t)+\rho) (\rho^2 t^2 + 2\rho^2 t +\bar\rho) $\\
(1,2,1,2,1,3,2)& $P\cdot((\rho^2+1)(t^2+2t)+\rho) (\rho^2 t^2 + 2\rho^2 t +\bar\rho) $\\
(2,1,2,1,2,3,2)& $P\cdot(t+\rho)(\rho^2t^2+2\rho^2 t+\bar\rho)(\rho^2t^4+4\rho^2 t^3+\rho^5 t^2+\rho^2\bar\rho^3 t+1) $\\
(2,1,2,1,3,2,3)& $P\cdot(t+\rho)((\rho^2+1)(t^2+2t)+\bar\rho) $\\
(2,1,2,3,1,2,3)& $P\cdot(t+\rho)((\rho^2+1)(t^2+2t)+\bar\rho) $\\ \hline  \end{tabular}\end{center}

\smallskip

Since $\rho$ and $\bar\rho$ are both positive numbers, all the above polynomials have positive coefficients.

\end{document}